\documentclass[a4paper,12pt,reqno]{amsart}
\usepackage{amsfonts}
\usepackage{amsmath}
\usepackage{amssymb}
\usepackage{mathrsfs}
\usepackage{nomencl}
\usepackage{delarray}
\usepackage{amsthm}
\usepackage{blkarray}
\usepackage{latexsym}
\usepackage[colorlinks]{hyperref}

\DeclareMathOperator{\Sp}{Sp}

\DeclareMathOperator{\Op}{Op}
\DeclareMathOperator{\inv}{inv}
\DeclareMathOperator{\Ker}{Ker}

\DeclareMathOperator{\Tr}{Tr}

\DeclareMathOperator{\supp}{supp}
\DeclareMathOperator{\diag}{diag}
\DeclareMathOperator{\rank}{rank}
\DeclareMathOperator{\card}{card}
\DeclareMathOperator{\Span}{span}

\newcommand{\iu}{{i\mkern1mu}}

\numberwithin{equation}{section}
\theoremstyle{plain}
\newtheorem{thm}{Theorem}[section]
\newtheorem{claim}[thm]{Claim}
\newtheorem{prop}[thm]{Proposition}
\newtheorem{cor}[thm]{Corollary}
\newtheorem{lem}[thm]{Lemma}

\theoremstyle{definition}
\newtheorem{defn}[thm]{Definition}

\newtheorem{rem}[thm]{Remark}
\newtheorem{ex}[thm]{Example}
\newcommand{\inlinemaketitle}{{\let\newpage\relax\maketitle}}
\newcommand{\be}{\begin{equation}}
\newcommand{\ee}{\end{equation}}

\newenvironment{psmallmatrix}
  {\left(\begin{smallmatrix}}
	    {\end{smallmatrix}\right)}

\def\eps{\varepsilon}
\def\dualSU2{\frac12\NN_0}
\def\RR{{\mathbb R}}
\def\HH{{\mathbb H}}
\def\ZZ{{\mathbb Z}}
\def\Gh{{\widehat{G}}}
\def\dpi{{d_\pi}}
\def\HS{{\mathtt{HS}}}
\def\op{{\mathtt{op}}}
\def\TT{{\mathbb T}}
\def\NN{{\mathbb N}}
\def\C{{\mathbb C}}

\def\FT{{\mathcal F}}

\def\M{\mathcal{M}}

\def\Rcal{{\mathcal R}}

\def\sp{{\rm sp}}
\def\H{\mathcal{H}}

\def\jp#1{{\left\langle{#1}\right\rangle}}

\def\VN{{\rm VN}}
\def\SU2{{\rm SU(2)}}

\begin{document}
\title[$L^p$-$L^q$ Multipliers on locally compact groups]
{$L^p$-$L^q$ Multipliers on locally compact groups}

\author[Rauan Akylzhanov]{Rauan Akylzhanov}
\address{Rauan Akylzhanov:
  \endgraf
  Department of Mathematics
  \endgraf
  Imperial College London
  \endgraf
  180 Queen's Gate, London SW7 2AZ
  \endgraf
  United Kingdom
  \endgraf
  {\it E-mail address} {\rm r.akylzhanov14@imperial.ac.uk}
}

\author[Michael Ruzhansky]{Michael Ruzhansky}
\address{
  Michael Ruzhansky:
  \endgraf
  Department of Mathematics
  \endgraf
  Imperial College London
  \endgraf
  180 Queen's Gate, London SW7 2AZ
  \endgraf
  United Kingdom
  \endgraf
  {\it E-mail address} {\rm m.ruzhansky@imperial.ac.uk}
  }

\thanks{The second
 author was supported by the EPSRC Grant EP/K039407/1 and by the Leverhulme Grant RPG-2014-02. No new data was collected or
generated during the course of research.}
\date{\today}


\maketitle
\begin{abstract}
	In this paper we discuss the $L^p$-$L^q$ boundedness of both spectral and Fourier multipliers
	on general locally compact separable unimodular groups $G$ for
	the range  $1<p\leq q<\infty$. We prove a Lizorkin type multiplier theorem for $1<p\leq q<\infty$, and then refine it as a H\"ormander type multiplier theorem for $1<p\leq 2\leq q<\infty$.
	In the process, we establish versions of Paley, Nikolsky and  
	Hausdorff-Young-Paley inequalities on 
	general locally compact separable unimodular groups. 
	As a consequence of the H\"ormander type multiplier theorem we derive a spectral multiplier theorem on general locally compact separable unimodular groups. We then apply it to  obtain embedding theorems as well as time-asymptotics for the $L^p$-$L^q$ norms of the heat kernels for general positive unbounded invariant operators on $G$.
	We illustrate the obtained results for sub-Laplacians on compact Lie groups and on the Heisenberg group. We show that our results imply the known results for $L^p$-$L^q$ multipliers such as H\"ormander's Fourier multiplier theorem 
	on $\RR^{n}$ or known results for Fourier 
	multipliers on compact Lie groups. The new approach developed in this paper relies on the analysis in the group von Neumann algebra for the derivation of the desired multiplier theorems.
\end{abstract}


\section{Introduction}

The aim of this paper is to give sufficient conditions for the $L^p$-$L^q$ boundedness of Fourier and spectral multipliers  on  locally compact separable unimodular groups. It is known that in this case we must have $p\leq q$ and two classical results are available on $\RR^n$, namely, H\"ormander's multiplier theorem \cite{Hormander:invariant-LP-Acta-1960} for $1<p\leq 2 \leq q<\infty$, and Lizorkin's multiplier theorem \cite{Lizorkin1967} for $1<p\leq q<\infty$. There is a philosophical difference between these results: H\"ormander's theorem does not require any regularity of the symbol and applies to $p$ and $q$ separated by $2$, while Lizorkin theorem applies also for $1<p\leq q\leq 2$ and $2\leq p \leq q<\infty$ but imposes certain regularity conditions on the symbol. 

In this paper we are able to prove versions of these theorems on general locally compact separable unimodular groups based on developing a new approach relying on the analysis in the noncommutative Lorentz spaces on the group von Neumann algebra. This suggested approach seems very effective, implying as special cases known results expressed in terms of symbols, in settings when the symbolic calculus is available. The obtained results are for general Fourier multipliers, in particular also implying new results for spectral multipliers.

The class of groups covered by our analysis is very wide. In particular, it contains abelian, compact, nilpotent groups, exponential, real algebraic or semi-simple Lie groups, solvable groups (not all of which are type I, but we do not need to assume the group to be of type I or II), and many others. As far as we are aware our results are new in all of these non-Euclidean settings.

In this paper we focus on the $L^p$-$L^q$ multipliers as opposed to the $L^p$-multipliers when theorems of Mihlin-H\"ormander or Marcinkiewicz types provide results for both Fourier and spectral multipliers in some settings, based on the regularity of the multiplier.  $L^p$-multipliers have been intensively studied on different kinds of groups, however, mostly $L^p$ spectral multipliers, for which a wealth of results is available: e.g. \cite{MS-1994,MRS-1995} on Heisenberg type groups, \cite{CM-1996} on solvable Lie groups, \cite{MT-2007} 
on nilpotent and stratified groups, to mention only very very few.
$L^p$ Fourier multiplies have been also studied but to a lesser extent due to lack of symbolic calculus that was not available until recently, e.g. Coifman and Weiss \cite{Coifman-Weiss:SU2-Argentina-1970,coifman+weiss_lnm} on SU(2), \cite{RuWi2013,RuWi2015} and then \cite{Fischer-Lp} on compact Lie groups, or \cite{Fischer-Ruzhansky:FM-graded,CardonaRuzhansky2016} on graded Lie groups.
A characteristic feature of the $L^p$-$L^q$ multipliers is that less regularity of the symbol is required. Therefore, in this paper we concentrate on the $L^p$-$L^q$ multiplier theorems, however aiming at obtaining unifying results for general locally compact groups. We give several short applications of the obtained results to questions such as embedding theorems and dispersive estimates for evolution PDEs.

The approach to the $L^{p}$-Fourier multipliers 
is different from the technique proposed in this paper
allowing us to avoid making an assumption that the group is compact or nilpotent. 
In this paper we are interested in both Fourier multipliers and spectral multipliers, for the latter some $L^p$-$L^q$ results being available in some special settings, see e.g. \cite{Cowling-Giulini-Meda:DMJ-1993}, and also
\cite{Cowling:PhD}, as well as \cite{ANR2016} for the case of SU(2),
and for the discussion of some relations between those in the group setting we can refer to
\cite{RuWi2015} and references therein.
Fourier multipliers in the context of group von Neumann algebras have been studied in \cite{GJungeParcet2015}.
By the combinatorial method it is possible to establish the $L^p$-$L^q$ estimates for the Poisson-type semigroup $\mathcal{P}_t$ on discrete groups $G$ \cite{Junge2013}.
Finally we note that multiplier estimates on
noncommutative groups are in general considerably more delicate than those in the commutative
case, recall e.g. the asymmetry problem and its resolution in
\cite{Dooley-Gupta-Ricci:assymetry-JFA-2000}. A link between Fourier multipliers and Lorentz spaces on group von Neumann algebras has been outlined in \cite{Akylzhanov2016}.

We now proceed to making a more specific description of the considered problems.

\subsection{H\"ormander's theorem on locally compact groups}
To put this 
in context, we recall that in \cite[Theorem 1.11]{Hormander:invariant-LP-Acta-1960}, 
Lars H\" ormander has shown that for $1<p\leq 2 \leq q<\infty$, if the symbol $\sigma_A\colon \RR^n\to \C$ of a Fourier multiplier $A$ on $\RR^n$ satisfies the condition
\begin{equation}\label{EQ:Horm}
\sup_{\substack{s>0}}
s\left(\,\int\limits_{\substack{\xi\in\RR^n\colon |\sigma_A(\xi)|\geq s}}d\xi\right)^{\frac1p-\frac1q}
<+\infty,
\end{equation}
then $A$ is a bounded operator from $L^p(\RR^n)$ to $L^q(\RR^n)$. 
Here, as usual, the Fourier multiplier $A$ on $\RR^n$ acts by multiplication on the Fourier transform side,
i.e. 
\begin{equation}\label{EQ:Rn-FM}
\widehat{Af}(\xi)=\sigma_{A}(\xi)\widehat{f}(\xi), \quad \xi\in\RR^{n}.
\end{equation} 
Moreover, it then follows that
\begin{equation}
\label{EQ:Hormander-estimate}
\|A\|_{L^p(\RR^n)\to L^q(\RR^n)}
\lesssim
\sup_{\substack{s>0}}s
\left(
\int\limits_{\substack{\xi\in\RR^n\\ |\sigma_A(\xi)|\geq s}}d\xi
\right)^{\frac1p-\frac1q},\quad 1<p\leq 2 \leq q <+\infty.
\end{equation}
The $L^{p}$-$L^{q}$ boundedness of Fourier multipliers has been also recently
investigated in the context of compact Lie groups, and we now briefly recall the result.
Let $G$ be a compact Lie group and $\Gh$ its unitary dual. For $\pi\in\Gh$, we write $\dpi$ for the dimension of the (unitary irreducible) representation $\pi$.
In \cite{ANRNotes2016}
the authors have shown that, for a Fourier multiplier $A$ acting via  $$\widehat{Af}(\pi)=\sigma_A(\pi)\widehat{f}(\pi)$$ by its global symbol $\sigma_A(\pi)\in \C^{\dpi\times\dpi}$ we have
\begin{equation}\label{EQ:comp}
\|A\|_{L^p(G)\to L^q(G)}
\lesssim
\sup_{s>0}
s
\left(
\sum\limits_{\substack{\pi\in\Gh\\ \|\sigma_A(\pi)\|_{\op}\geq s}}d^2_{\pi}
\right)^{\frac1p-\frac1q},\quad 1<p\leq 2 \leq q \leq \infty.
\end{equation}
Here for $\pi\in\Gh$, the Fourier coefficients are defined as 
$$\widehat{f}(\pi)=\int_G f(x)\pi(x)^* dx,$$
and $\|\sigma_A(\pi)\|_{\op}$ is the operator norm of $\sigma_A(\pi)$ as the linear transformation of the representation space of $\pi\in\Gh$ identified with $\C^\dpi$.
For a general development of global symbols and the corresponding global quantization
of pseudo-differential operators on compact Lie groups we can refer to
\cite{Ruzhansky+Turunen-IMRN,RT}.

One of the results of this paper 
generalises both multiplier theorems \eqref{EQ:Hormander-estimate} and 
\eqref{EQ:comp} 
to the setting of general locally compact separable unimodular groups $G$.

By a {\em left Fourier multiplier in the setting of general locally compact unimodular groups we will mean left invariant operators that are measurable with respect to the right group von Neumann algebra $\VN_R(G)$}, see Section \ref{SEC:FM-def} for a discusson.

Thus, in Theorem \ref{THM:upper-bound} we prove the following inequality
\begin{equation}
\label{EQ:upper-estimate-intro}
\|A\|_{L^p(G)\to L^q(G)}
\lesssim
\sup_{\substack{s>0}}s
\left[
\int\limits_{\substack{t\in\RR_+\colon \mu_t(A)\geq s}}dt
\right]^{\frac1p-\frac1q},\quad 1<p\leq 2 \leq q<+\infty,
\end{equation}
where $\mu_t(A)$ are the $t$-th generalised singular values of $A$, see \cite{ThierryKosaki1986} (and also Definition \ref{DEF:mu-t}) for definition and properties.
The proof of inequality \eqref{EQ:upper-estimate-intro} is based on a version of the Hausdorff-Young-Paley inequality on locally compact separable groups that we establish for this purpose. 

The key idea behind this extension is that H\"ormander's theorem \eqref{EQ:Hormander-estimate} can be reformulated as
\begin{equation}
\label{EQ:Hormander-estimate-new}
\|A\|_{L^p(\RR^n)\to L^q(\RR^n)}
\lesssim
\sup_{\substack{s>0}}s
\left(
\int\limits_{\substack{\xi\in\RR^n\\ |\sigma_A(\xi)|\geq s}}d\xi
\right)^{\frac1p-\frac1q}\simeq \|\sigma_A\|_{L^{r,\infty}(\RR^n)}\simeq \|A\|_{L^{r,\infty}(\VN(\RR^n))},
\end{equation}
where $\frac1r=\frac1p-\frac1q$, $\|\sigma_A\|_{L^{r,\infty}(\RR^n)}$ is the Lorentz space norm of the symbol $\sigma_A$, and $\|A\|_{L^{r,\infty}(\VN(\RR^n))}$
is the norm of the operator $A$ in the Lorentz space on the group von Neumann algebra $\VN(\RR^n)$ of $\RR^n$. In turn, our estimate \eqref{EQ:upper-estimate-intro} is equivalent to the estimate
\begin{equation}
\label{EQ:Hormander-estimate-new-G}
\|A\|_{L^p(G)\to L^q(G)}
\lesssim
 \|A\|_{L^{r,\infty}(\VN_R(G))}
 \simeq
\sup_{\substack{s>0}}s
\left[
\int\limits_{\substack{t\in\RR_+\colon \mu_t(A)\geq s}}dt
\right]^{\frac1r},
\end{equation}
where $\|A\|_{L^{r,\infty}(\VN_R(G))}$
is the norm of the operator $A$ in the noncommutative Lorentz space on the right group von Neumann algebra $\VN_R(G)$ of $G$. Thus, the Lorentz spaces become a key point for the extension of H\"ormander's theorem to the setting of locally compact (unimodular) groups.

\smallskip
In Remark \ref{REM:AR-H} and Proposition \ref{PROP:comparison} we show that
the multiplier theorem \eqref{EQ:upper-estimate-intro} implies both
\eqref{EQ:Hormander-estimate} and 
\eqref{EQ:comp} in the respective settings of $\RR^{n}$ and compact Lie groups.


We assume for simplicity that $G$ is unimodular but we do not make assumption that  $G$ is either of type I or type II. The assumption for the locally compact group to be separable and unimodular may be viewed as
natural allowing one to use basic results of von Neumann-type 
Fourier analysis, such as, for example,
Plancherel formula (see Segal \cite{Segal1950}).
However, the unimodularity assumption may be in principle avoided, see e.g.
\cite{Duflo-Moore:JFA-1976}, but the exposition becomes much more technical.
For a more detailed discussion of pseudo-differential operators in such settings
we refer to \cite{MR2015}, but we note that compared to the analysis there in this paper we do not need to assume that
the group is of type I.

\subsection{Spectral multipliers on locally compact groups}

Let us illustrate the use of the Fourier multiplier theorem \eqref{EQ:upper-estimate-intro} in the important case of spectral multipliers on locally compact groups. Later, in Theorem \ref{THM:varphi-L} we will give a spectral multiplier result on general semifinite von Neumann algebras, however, we now formulate its special case for the case of group von Neumann algebras associated to locally compact groups.

Interestingly, this result asserts that the $L^p$-$L^q$ norms of spectral multipliers $\varphi(|\mathcal{L}|)$ depend essentially only on the rate of growth of traces of spectral projections of the operator $|\mathcal L|$:

\begin{thm}
\label{THM:varphi-L-intro}
Let $G$ be a locally compact separable unimodular group and let 
$\mathcal{L}$ be a left Fourier multiplier on $G$.
Assume that $\varphi$ is a monotonically decreasing continuous function on $[0,+\infty)$ such that
\begin{eqnarray*}
\label{EQ:phi-normalization}
\varphi(0)=1,
\\
\label{EQ:phi-empty-energy}
\lim_{u\to+\infty}\varphi(u)=0.
\end{eqnarray*}
Then we have the inequality
\begin{equation}\label{EQ:lplqest1}
\|\varphi(|\mathcal{L}|)\|_{L^p(G)\to L^q(G)}
\lesssim
\sup_{u>0}\varphi(u)
\left[\tau(E_{(0,u)}(|\mathcal{L}|))\right]^{\frac1p-\frac1q},\quad 1<p\leq 2 \leq q<\infty.
\end{equation}
\end{thm}

Here  $E_{(0,u)}(|\mathcal L|)$ are the spectral projections associated to the operator $|\mathcal L|$ to the interval $(0,u)$, see Section \ref{SEC:prelim} for precise definitions, and $\tau$ is the canonical trace on the right group von Neumann algebra $\VN_R(G)$, see Section \ref{SEC:traces} for a discussion.

Also we note that more general statements, without the above assumptions on $\varphi$, are possible, see Corollary \ref{COR:gen}.

The estimate \eqref{EQ:lplqest1} says that if the supremum on the right hand side is finite then the operator $\varphi(|\mathcal{L}|)$ is bounded from $L^p(G)$ to $L^q(G)$. Moreover, the estimate for the operator norm can be used for deriving asymptotics for propagators for equations on $G$. For example, we get the following consequences for the $L^p$-$L^q$ norm for the heat kernel of $\mathcal L$, applying Theorem \ref{THM:varphi-L-intro} with $\varphi(u)=e^{-tu}$, or embedding theorems for $\mathcal L$ with $\varphi(u)=\frac1{(1+u)^{\gamma}}$.

We note that estimates of the type \eqref{EQ:asymptotics-condition0d} are exactly those leading to subsequent Strichartz estimates. Here, our method is very different from the usual ones as we do not get it by interpolation from the end-point case.

\begin{cor}
\label{cor:heat-equation}
Let $G$ be a locally compact unimodular separable group and let $\mathcal{L}$ be a positive left Fourier multiplier such that for some $\alpha$ we have
\begin{equation}
\label{EQ:asymptotics-condition0}
\tau(E_{(0,s)}(\mathcal{L}))\lesssim s^{\alpha},\quad s\to \infty.
\end{equation}
Then for any $1<p\leq 2 \leq q<\infty$ there is a constant $C=C_{\alpha,p,q}>0$ such that we have
\begin{equation}\label{EQ:asymptotics-condition0d}
\|e^{-t\mathcal{L}}\|_{L^p(G)\to L^q(G)}
\leq
C t^{-\alpha \left(\frac1p-\frac1q\right) },\quad t>0.
\end{equation}
We also have the embeddings
\begin{equation}\label{EQ:0-s-embedding2t1}
\|f\|_{L^q(G)}
\leq
C
\|(1+\mathcal L)^{\gamma}f\|_{L^p(G)},
\end{equation}
provided that
\begin{equation}\label{EQ:0-s-embedding2t2}
\gamma\geq \alpha\left(\frac1p-\frac1q\right),\quad 1<p\leq 2 \leq q<\infty.
\end{equation}

\end{cor}
The number $\alpha$ in \eqref{EQ:asymptotics-condition0} is determined based on the spectral properties of $\mathcal L$. For example, we have
\begin{itemize}
\item[(a)] if $\mathcal L$ is the sub-Laplacian on a compact Lie group $G$ then $\alpha=\frac{Q}{2}$, where $Q$ is the Hausdorff dimension of $G$ with respect to the control distance associated to $\mathcal L$;
\item[(b)] if $\mathcal L$ is the sub-Laplacian on the Heisenberg group $G=\HH^n$ then $\alpha=\frac{Q}{2}$, where $Q=2n+2$ is the homogeneous dimension of $\HH^n$.
\end{itemize} 
Consequently, in both of these sub-Laplacian cases,
Corollary \ref{cor:heat-equation} implies that for any 
 $1<p\leq 2 \leq q<\infty$ there is a constant $C=C_{p,q}>0$ such that we have
\begin{equation}
\|e^{-t\mathcal{L}}\|_{L^p(G)\to L^q(G)}
\leq
C t^{-\frac{Q}{2} \left(\frac1p-\frac1q\right) },\quad t>0.
\end{equation}

The embeddings \eqref{EQ:0-s-embedding2t1} under conditions \eqref{EQ:0-s-embedding2t2}
show that the {\em statement of Theorem \ref{THM:varphi-L-intro} is in general sharp.}
Taking $\varphi(s)=\frac1{(1+s)^{a/2}}$ and applying \eqref{EQ:lplqest1} to the sub-Laplacian $\Delta_{sub}$ in either of examples (a) or (b) above, we get that the operator $\varphi(-\Delta_{sub})=(I-\Delta_{sub})^{-a/2}$ is $L^p(G)$-$L^q(G)$ bounded and the inequality 
\begin{equation}
\label{EQ:0-s-embedding0}
\|f\|_{L^q(G)}
\leq
C
\|(1-\Delta_{sub})^{a/2})f\|_{L^p(G)}
\end{equation}
holds true provided that
\begin{equation}\label{EQ:0-s-embedding20}
a\geq Q\left(\frac1p-\frac1q\right),\quad 1<p\leq 2 \leq q<\infty.
\end{equation}
However, this yields the Sobolev embedding theorem which is well-known to be sharp at least in the case (b) of the Heisenberg group (\cite{Folland1975}), showing {\em the sharpness of Theorem \ref{THM:varphi-L-intro} and hence also of the Fourier multiplier theorem \eqref{EQ:Hormander-estimate-new-G}}. 
More examples are given in Section \ref{SEC:heat}.

\subsection{Lizorkin theorem}

The classical Lizorkin theorem \cite{Lizorkin1967} applies for the range $1<p\leq q<\infty$.
Let $A$ be a Fourier multiplier on $\RR$ with the symbol $\sigma_A$ as in \eqref{EQ:Rn-FM}.
Assume that for some $C<\infty$ the symbol $\sigma_A(\xi)$ satisfies the following conditions
\begin{eqnarray}\label{Liz1}
\sup_{\xi\in\RR}|\xi|^{\frac1p-\frac1q}|\sigma_A(\xi)|\leq C,
\\ \label{Liz2}
\sup_{\xi\in\RR}|\xi|^{\frac1p-\frac1q+1}\left|\frac{d}{d \xi}\sigma_A(\xi)\right|\leq C.
\end{eqnarray}
Then $A\colon L^p(\RR)\to L^q(\RR)$ is a bounded linear operator and
\begin{equation}
\|A\|_{L^p(\RR)\to L^q(\RR)}\lesssim C.
\end{equation}
An extension to $G=\RR^n$ with sharper conditions on the symbol has been obtained in \cite{Sarybekova2010}.
A number of papers \cite{Ydyrys2016,Persson2008,Persson2012} deal with the same problem on $G=\TT^n$ and $G=\RR^n$.

In Section \ref{SEC:Lizorkin} we establish versions of this result in two settings: Lizorkin type
Fourier multiplier theorem on general locally compact separable unimodular groups, and 
Fourier multiplier theorem, spectral multiplier theorem, and $L^p$-$L^q$ boundedness for general, not necessarily invariant operators, on compact Lie groups.
Our proofs are based on several new ingredients: Nikolskii inequality, approximations by trigonometric type functions, Abel transform, and a a new type of difference operators on the unitary duals of compact Lie groups for measuring the required regularity of symbols.

\subsection{Other results}

Our proof of Fourier (and then also spectral) multiplier theorems are based on two major new ingredients which are of interest on its own: Paley/Hausdorff-Young-Paley and Nikolskii inequalities for the H\"ormander and Lizorkin versions of multiplier statements, respectively. 

Recall briefly that in \cite{Hormander:invariant-LP-Acta-1960} H\"ormander has shown the following version of the {\em Paley inequality} on $\RR^n$:
if a positive function $\varphi\geq 0$ satisfies
\begin{equation}
\label{EQ:Horm}
	|\{\xi\in\RR^n\colon \varphi(\xi)\geq t\}|\leq\frac{C}{t}\quad\textrm{ for } t>0,
\end{equation}
then 
\begin{equation}\label{EQ:Paley-Rn}
\left(\,\,
\int_{\RR^n}
\left|\widehat{u}(\xi)\right|^p
\varphi(\xi)^{2-p}\,d\xi
\right)^{\frac1p}
\lesssim
\|u\|_{L^p(\RR^n)},\quad 1 < p\leq 2.
\end{equation}
For the special case of $\varphi(\xi)=(1+|\xi|)^{-n}$, this inequality implies the classical  
Hardy-Littlewood inequality \cite{HL} giving necessary condition for $u$ to be in $L^p$ in terms of its Fourier 
coefficients for $1<p\leq 2$. For functions with monotone Fourier coefficients such results serve as an extension of the Plancherel identity to $L^p$-spaces: for example, on the circle $\TT$, Hardy and Littlewood have shown that 
for $1<p<\infty$, if the Fourier coefficients $\widehat{f}(m)$ are monotone, then one has
\begin{equation}
\label{H-L-equivalence}
f\in L^p(\TT)
	\quad\textrm{ if and only if }\quad
	\sum\limits_{m\in \ZZ}(1+|m|)^{p-2}|\widehat{f}(m)|^p<\infty.
\end{equation}
Hardy-Littlewood inequalities on locally compact groups have been studied e.g. by Kosaki \cite{Kosaki1981}, see Theorem \ref{THM:HL-LCG}.
In Section \ref{SEC:P-HYP} we establish a version of Paley inequality, and consequently of the Hausdorff-Young-Paley inequality on locally compact groups, yielding extensions of the Euclidean version \eqref{EQ:Paley-Rn} as well as of Kosaki's results. The established Hausdorff-Young-Paley inequality (Theorem \ref{THM:HYP-LCG}) is a crucial ingredient in our proof of H\"ormander's version of multiplier theorem in Section \ref{SEC:Horm-LCG}.

The crucial idea for our proof of Lizorkin theorem is the Nikolskii inequality, sometimes also called the reverse H\"older inequality in the literature.
Originating in Nikolskii's work \cite{Nikolskii1951} in 1951 for trigonometric inequalities on the circle, on $\mathbb{R}^n,$ the {\em Nikolskii inequality} takes the form
\begin{equation}
\Vert f\Vert_{L^{q}(\mathbb{R}^n)}\leq C[{\rm vol}[\textrm{conv}[\text{supp}(\widehat{f})]]^{\frac{1}{p}-\frac{1}{q}}\Vert f\Vert_{L^{p}(\mathbb{R}^n)},\quad
1\leq p\leq q\leq \infty,
\end{equation}
for every function $f\in{L}^p(\mathbb{R}^n)$ with Fourier transform $\widehat{f}$ of compact support, where $\textrm{conv}(E)$ denotes the convex hull of the set $E.$ 
The Nikolskii inequality plays an important role in many questions of function theory, harmonic analysis, and approximation theory. Its versions on compact Lie groups (and compact homogeneous manifolds) and on graded Lie groups have been established in
\cite{NRT2014} (\cite{NRT2015}) and in \cite{CardonaRuzhansky2016}, with further applications to Besov spaces and to Fourier multipliers acting in Besov spaces in those settings.

In Section \ref{SEC:Nik} we prove a version of the Nikolskii inequality on general locally compact separable unimodular groups. An interesting question in this setting already is how to understand trigonometric functions in such generality. In Theorem \ref{THM:Nikolsky-LCG}
we show that 
\begin{equation}
\label{EQ:Nikolsky-LCG0}
\|f\|_{L^q(G)}
\lesssim
\left(
\tau(P_{\supp^R[\widehat{f}]})
\right)^{\frac1p-\frac1q}
\|f\|_{L^p(G)},\quad 1<p\leq \min(2,q), \; 1\leq q\leq \infty,
\end{equation}
with trigonometric function interpreted as having $\tau(P_{\supp^R[\widehat{f}]})<\infty$, where $P_{\supp^R[\widehat{f}]}$ denotes  the orthogonal projector onto the support $\supp^R[\widehat{f}]$ of the operator-valued Fourier transform of $f$. 
The estimate \eqref{EQ:Nikolsky-LCG0} will play a crucial role in proving Lizorkin type multiplier theorems in Section \ref{SEC:Lizorkin}.
%
%


\section{Notation and preliminaries}
\label{SEC:main}
\label{SEC:prelim}

In this section fix the notation and briefly recall some 
preliminaries on von Neumann algebras to be used for developing subsequent harmonic analysis on locally compact groups. For exposition purposes it seems beneficial to recall several
general notions in the context of general von Neumann algebras $M$. However, for
our application to multipliers on locally compact groups $G$ we will be later setting
$M$ to be the right group von Neumann algebra ${\VN}_{R}(G)$. In particular, we will
be able to readily apply the notion of noncommutative Lorentz spaces on $M$ as developed in
\cite{Kosaki1981}, one of the key ingredients for our analysis.

\medskip
Let $M\subset \mathcal{L}(\H)$ be a semifinite von Neumann algebra acting in a Hilbert space $\H$ with a trace $\tau$.
The semifinite assumption simplifies the formulations and is satisfied in our main
example $M={\VN}_{R}(G)$.

\begin{defn}[Affiliated operators]
\label{DEF:affiliation}
A linear closed operator $A$ (possibly unbounded in $\H$) is said to be {\em affiliated with $M$}, symbolically $A\nu M$, if it commutes with the elements of the commutant $M^!$ of $M$, i.e.
\begin{equation}
AU=UA,\quad \textrm{ for all }\; U\in M^{!}.
\end{equation}
\end{defn}
This relation $\nu$ is a natural relaxation of the relation $\in$:
if $A$ is a bounded operator affiliated with $M$, then by the double commutant theorem $A\in M$.
One of the original motivations \cite{RO1936,RO1937} of John von Neumann was to build a mathematical foundation for quantum mechanics. In this framework, the observables with unbounded spectrum correspond to closed densely defined unbounded operators. Although the algebra $M$ consists primarily of bounded operators, the technique of projections makes it possible to approximate unbounded operators.

\begin{defn}[$\tau$-measurable operators $S(M)$]
\label{DEF:operator-measurability}
A closeable operator $A$ (possibly unbounded) affiliated with $M$ is said to be 
{\em $\tau$-measurable} if for each $\eps>0$ there exists a projection $p$ in $M$ such that $p\H\subset D(A)$ and $\tau(I-p)\leq \eps$. Here $D(A)$ is the domain of $A$ in $\H$.
We denote by $S(M)$ the set of all $\tau$-measurable operators. 
\end{defn}

We note that the notion of $\tau$-measurability does not appear in the classical theory of Schatten classes since for $M=\mathcal{L}(H)$ we have $S(\mathcal{L}(H))=\mathcal{L}(H)$.

\begin{ex}
\label{EX:abelian-measurability}
Let $M=\{M_{\varphi}\colon L^2(X,\mu)\ni f \mapsto M_{\varphi}f=\varphi f\in L^2(X,\mu)\}_{\varphi\in L^{\infty}(X,\mu)}$ and take $\tau(M_{\varphi}):=\int\limits_{X}\varphi d\mu$, where $(X,\mu)$ is a measure space. Then an operator $M_{\varphi}$ is $\tau$-measurable if and only if $\varphi$ is a $\mu$-almost everywhere finite function.  
\end{ex}
The $*$-algebra $S(M)$ is a basic constructon for the noncommutative integration. 
Let $A=U|A|$ be the polar decomposition. The spectral theorem yields that
\begin{equation}
|A|=\int\limits_{\Sp(|A|)}\lambda dE_{\lambda}(|A|),
\end{equation}
where $\{E_{\lambda}(|A|)\}_{\lambda\in \Sp(|A|)}$ are the spectral projections associated with the operator $|A|$. Here $dE_{\lambda}(|A|)$ should be understood as the relative dimension function first constructed in \cite{RO1936}. Since $A$ is affiliated with $M$, the projections satisfy $E_{\lambda}(|A|)\in M$. Now, we are ready to `measure the speed of decay' of the operator $A$.

\begin{defn}[Generalised $t$-th singular numbers] \label{DEF:mu-t}
For an operator $A\in S(M)$, define the distribution function $d_{\lambda}(A)$  by 
\begin{equation}
\label{EQ:d-A-s}
d_{\lambda}(A):=\tau(E_{(\lambda,+\infty)}(|A|)),\quad \lambda\geq0,
\end{equation}
where $E_{(\lambda,+\infty)}(|A|)$ is the spectral projection of $|A|$ corresponding to the interval $(\lambda,+\infty)$.
For any $t>0$, we define the generalised $t$-th singular numbers by
\begin{equation}
\label{EQ:mu-t}
\mu_t(A):=\inf\{\lambda\geq 0 \colon d_{\lambda}(A)\leq t\}.
\end{equation}
\end{defn}

\begin{ex}
\label{EX:mu-t}
For the operator $M_{\varphi}$ in Example \ref{EX:abelian-measurability}, from Defintion \ref{DEF:mu-t} we can show its generalised $t$-th singular numbers to be
$$
\mu_t(M_{\varphi})=\varphi^*(t),
$$
where $\varphi^*(t)$ is the classical function rearrangement (see e.g. \cite{BeSh1988}).
\end{ex}
%
As a noncommutative extension \cite{Kosaki1981} of the classical Lorentz spaces, we define Lorentz spaces $L^{p,q}(M)$ associated with a semifinite von Neumann algebra $M$ as follows:

\begin{defn}[Noncommutative Lorentz spaces]
\label{DEF:Lorenz-spaces}
For $1\leq p <\infty$, $1\leq q <\infty$, denote by $L^{p,q}(M)$ the set of all operators $A\in S(M)$ satisfying
\begin{equation}
\|A\|_{L^{p,q}(M)}
:=
\left(
\int\limits^{+\infty}_0 
\left(
t^{\frac1p}\mu_t(A)
\right)^q
\frac{dt}{t}
\right)^{\frac1q}<+\infty.
\end{equation}
For $q=\infty$, we define $L^{p,\infty}(M)$ as the space of all operators $A\in S(M)$ satisfying

\begin{equation}
\|A\|_{L^{p,\infty}(M)}
:=
\sup_{t>0}t^{\frac1p}\mu_t(A).
\end{equation}
With this, for $1\leq p<\infty$, we can also define $L^{p}$-spaces on $M$ by
$$
\|A\|_{L^{p}(M)}:=\|A\|_{L^{p,p}(M)}=\left(
\int\limits^{+\infty}_0 
\mu_t(A)^p\
dt
\right)^{\frac1p}.
$$
\end{defn}

The classical Lorentz spaces $L^{p,q}(X,\mu)$ correspond to the case of commutative von Neumann algebra. Modulus technical details \cite[p. 132, Theorem 1]{VNA-Dixmier-1981}, an arbitrary abelian von Neumann algebra in a Hilbert space $\H$ is isometrically isomorphic to  the algebra $\{M_{\varphi}\}_{\varphi\in L^{\infty}(X,\mu)}$ from Example \ref{EX:abelian-measurability}.
Then noncommutative Lorentz spaces coincide with the classical ones:

\begin{ex}[Classical Lorentz spaces]
Let $M$ be the algebra $\{M_{\varphi}\}_{\varphi\in L^{\infty}(X,\mu)}$ from Example \ref{EX:abelian-measurability} 
consisting of all the multiplication operators $M_{\varphi}\colon L^2(X,\mu)\ni f \mapsto M_{\varphi}f=\varphi f\in L^2(X,\mu)$. 
By Example \ref{EX:mu-t}, we have
$$
\mu_t(M_{\varphi})=\varphi^*(t).
$$
Thus, the Lorentz space $L^{p,q}(M)$ consists of all operators $ M_{\varphi}$ such that
$$
\int\limits^{+\infty}_{0}[t^{\frac1p}\varphi^*(t)]^q\frac{dt}{t}<+\infty,
$$
which gives the classical Lorentz space.
\end{ex}

Concerning the structure of semifinite von Neumann algebras,
given an arbitrary semifinite von Neumann algebra $M$ with a trace $\tau$, there is an isomorphism of $M$ onto a certain Hilbert algebra $\mathcal{U}$ (\cite[p. 99, Theorem 2]{VNA-Dixmier-1981}). Thus, we construct the trace on the Hilbert algebra yielding the trace on $M$ due to ismomorphism. We refer to \cite{VNA-Dixmier-1981}, \cite{Najmark1972} as well as to Section
\ref{SEC:traces} for more details on this.

\medskip
Let now $G$ be a locally compact unimodular separable group. Denote by $\pi_{L}(g)$ and $\pi_R(g)$ the left and the right action of $G$ on $L^2(G)$, respectively:
\begin{eqnarray*}
\pi_{L}(g)f(x):=f(g^{-1}x),
\\
\pi_{R}(g)f(x):=f(xg),
\end{eqnarray*}
and by $\VN_{L}(G)$ the group von Neumann algebra generated by all the $\pi_L(g)$ with $g\in G$, i.e.
$$
{\VN}_{L}(G):=\{\pi_L(g)\}^{!!}_{g\in G},
$$
and similary
$$
{\VN}_{R}(G):=\{\pi_R(g)\}^{!!}_{g\in G},
$$
where ${!!}$ is the bicommutant of the self-adjoint subalgebras $\{\pi_L(g)\}_{g\in G},\{\pi_R(g)\}_{g\in G}\subset \mathcal{L}(L^2(G))$.
It has been shown in \cite{Segal1949} that
\begin{eqnarray}
\label{EQ:VN-L-R}
\VN_L(G)^!=\VN_R(G),
\\
\label{EQ:VN-R-L}
\VN_R(G)^!=\VN_L(G).
\end{eqnarray}
We do not make assumption that  $G$ is either of type I or type II. 
The decomposition theory for unitary representations of locally compact separable unimodular groups has been established 
in \cite{Ernest1961, Ernest1962}.

\medskip
From now on we take $M=\VN_R(G)$.

\medskip
For $f\in L^1(G)\cap L^2(G)$, we say that $f$ on $G$ {\em has a Fourier transform} 
whenever the convolution operator 
\begin{equation}
\label{EQ:L_f}
R_fh(x):=(h\ast f)(x)=\int\limits_{G}h(g)f(g^{-1}x)\,dg
\end{equation}
 is a $\tau$-measurable operator with respect to $\VN_R(G)$, i.e. $R_f\in S(\VN_R(G))$. 
The Plancherel identity takes (\cite[Theorem 3 on page 282]{Segal1950}) the form
\begin{equation}
\label{EQ:plancherel}
\|R_f\|_{L^2(\VN_{R}(G))}
=
\|f\|_{L^2(G)}.
\end{equation}
In this setting, the Hausdorff-Young inequality has been established in \cite{Kunze1958} in the form
\begin{equation}
\label{EQ:HY-LCG}
\|R_f\|_{L^{p'}(\VN_R(G))}
\leq
\|f\|_{L^p(G)},\quad 1<p\leq 2.
\end{equation}
In \cite{Kosaki1981}, as an application of the technique of the $t$-th generalised singular values, the Hardy-Littlewood theorem (\cite{HL})
has been generalised to an arbitrary locally compact separable unimodular group $G$:
\begin{thm}[\cite{Kosaki1981}]
\label{THM:HL-LCG}
Let $1<p\leq 2$ and $f\in L^p(G)$. Then we have
\begin{equation}
\label{EQ:HL-LCG}
\|R_f\|_{L^{p',p}(\VN_R(G))}
\leq
\|f\|_{L^p(G)}.
\end{equation}
\end{thm}
\begin{rem} 
The Plancherel equality \eqref{EQ:plancherel}
by Segal \cite{Segal1950} and
Kosaki's version \cite{Kosaki1981} of Hardy-Littlewood inequality \eqref{EQ:HL-LCG}
have been originally established for the left convolution $L_fh=f\ast h$. However, the same line of reasoning yields inequalities \eqref{EQ:HL-LCG} and \eqref{EQ:plancherel}
with the right convolution $R_f$.  We work with the right convolution operators $R_f$ here since it naturally corresponds to left-invariant operators when analysing the Fourier multipliers on groups.
\end{rem}
Using the technique of the $t$-th generalised singular values developed in \cite{ThierryKosaki1986}, we can formulate both the Hausdorff-Young \eqref{EQ:HY-LCG}  and Hardy-Littlewood \eqref{EQ:HL-LCG} inequalities  in the forms (for $1<p\leq 2$):
\begin{eqnarray}
\label{EQ:HY-LCG-2}
\left(
\int\limits^{+\infty}_{0}\mu_t(R_f)^{p'}dt
\right)^{\frac1{p'}}
\equiv
\|R_f\|_{L^{p'}(\VN_R(G))}
\leq
\|f\|_{L^p(G)},
\\
\label{EQ:HL-LCG-2}
\left(
\int\limits^{+\infty}_{0}t^{p-2}\mu_t(R_f)^pdt
\right)^{\frac1{p}}
\equiv
\|R_f\|_{L^{p',p}(\VN_R(G))}
\leq
\|f\|_{L^p(G)}.
\end{eqnarray}

In the sequel, when we prove Paley inequality in Theorem \ref{THM:Paley-LCG}, the Hardy-Littlewood inequalities \eqref{EQ:HL-LCG} and \eqref{EQ:HL-LCG-2} (for the right convolution $R_f$) will also follow as its special cases.


\subsection{Fourier multipliers on locally compact groups}
\label{SEC:FM-def}

Let $G$ be a locally compact separable unimodular group. 
The first question is how to understand the notion of Fourier multipliers.
In the first instance we adopt the following definition:

\begin{defn}
\label{DEF:FM}
 A linear operator $A$  is said to be a left {\em Fourier multiplier on $G$} if $A\in S({\VN}_R(G))$.
\end{defn}
If we now recall Definition \ref{DEF:affiliation} we can see that $A$  is a left Fourier multiplier on $G$ if and only if  $A$ is affiliated with the right group von Neumann algebra  $\VN_R(G)$
and is $\tau$-measurable.
We can then clarify Definition \ref{DEF:FM} further:
\begin{rem} 
\label{REM:FM-affiliation}
For $M=\VN_R(G)$ the operators affiliated with $M$ are precisely those $A$ that are left-invariant on $G$, namely,
\begin{equation}
A \text{ is affiliated with } \VN_R(G) \iff A\pi_L(g)=\pi_L(g)A,\; \textrm{ for all } g\in G.
\end{equation}
Summarising this observation with Definition \ref{DEF:FM}, 
{\em left Fourier multipliers on $G$
are precisely the left-invariant operators that are measurable} 
(in the sense of Definition \ref{DEF:operator-measurability}).
\end{rem}

\begin{proof}[Proof of Remark \ref{REM:FM-affiliation}]
$\Longrightarrow$.
By Definition \ref{DEF:FM}, we have
\begin{equation}
AU=UA,\quad \textrm{for all }\; U\in \VN_R(G)^!.
\end{equation}
Then by \eqref{EQ:VN-R-L}, and by taking $U=\pi_L(g)$, $g\in G$,
we see that $A$ must be left-invariant.

$\Longleftarrow$. 
We have
\begin{equation*}
A\pi_L(g)=\pi_L(g)A,\quad \textrm{for all }\; g\in G.
\end{equation*}
By definition, the algebra $\VN_L(G)$ is the closure of the involutive subalgebra 
$$\{\pi_L(g)\}_{g\in G}\subset \mathcal{L}(L^2(G))$$ in the strong operator topology.
Therefore, we obtain
\begin{equation}
AU=UA,\quad \textrm{for all }\;  U\in \VN_R(G)^!,
\end{equation}
where we used \eqref{EQ:VN-R-L}.
This completes the proof of Remark \ref{REM:FM-affiliation}.
\end{proof}

\section{Paley and Hausdorff-Young-Paley inequalities}
\label{SEC:P-HYP}


Our analysis of $L^{p}$-$L^{q}$ multipliers will be based on a version of the 
Hausdorff-Young-Paley inequality that we establish in this section in the context of locally compact groups.
It will be obtained by interpolation between the Hausdorff-Young inequality and Paley inequality that we discuss first.

We start first with an inequality that can be regarded as a Paley type inequality.

\begin{thm}[Paley inequality]
\label{THM:Paley-LCG}
Let $G$ be a locally compact unimodular separable group. 
Let $1<p\leq 2$. Suppose that a positive function $\varphi(t)$ satisfies the condition
\begin{equation}
\label{EQ:weak-cond}
M_{\varphi}:=\sup_{s>0}s\int\limits_{\substack{t\in\RR_+\\ \varphi(t)\geq s}}dt<+\infty.
\end{equation}
Then for all $f\in L^p(G)$ we have
\begin{equation}
\label{EQ:Paley-LCG}
\left(
\int\limits^{+\infty}_0 \mu_t(R_f)^p\varphi(t)^{2-p}\,dt
\right)^{\frac1p}
\leq
M_{\varphi}^{\frac{2-p}{p}}
\|f\|_{L^p(G)}.
\end{equation}
\end{thm}

As usual, the integral over an empty set in \eqref{EQ:weak-cond} is assumed to be zero.

We note that taking $\varphi(t)=\frac1t$
we recover Kosaki's Hardy-Littlewood inequality \eqref{EQ:HL-LCG-2}.
 In this sense, the Paley inequality can be viewed as an
extension of (one of) the Hardy-Littlewood inequalities. As a small byproduct of
our proof of Theorem \ref{THM:Paley-LCG} we thus get a simple proof of
Theorem \ref{THM:HL-LCG}.

\begin{proof}[Proof of Theorem \ref{THM:Paley-LCG}]
Let $\nu$ give measure $\varphi^2(t)$ to the set consisting of the single point $\{t\}$, $t\in\RR_+$, 
 i.e.
\begin{equation}
\label{EQ:nu-definition-2}
\nu(t):=\varphi^2(t)dt.
\end{equation}
We define the corresponding space $L^p(\RR_+,\nu)$, $1\leq p<\infty$, 
as the space of complex (or real) valued functions
$f=f(t)$ such that
\begin{equation}\label{EQ:Lpmu}
\|f\|_{L^p(\RR_+,\nu)}
:=
\left(
\int\limits_{\substack{\RR_+}}
|f(t)|^p
\varphi^2(t)\,dt
\right)^{\frac1p}
<\infty.
\end{equation}
We will show that the sub-linear operator
$$
T\colon 
L^p(G)\ni f  \mapsto Tf:=
\mu_t(R_{f})/\varphi(t)\in L^p(\RR_+,\nu)
$$
is well-defined and bounded from $L^p(G)$ to $L^p(\RR_+,\nu)$ for $1<p\leq 2$. 
In other words, we claim that we have the estimate
\begin{equation}
\label{Paley_inequality_alt}
\|Tf\|_{L^p(\RR_+,\nu)}
=
\left(
\int_{\substack{\RR_+}}
\left(
\frac{
\mu_t(R_{f})
}{\varphi(t)}
\right)^p
\varphi^2(t)\,dt
\right)^{\frac1p}
\lesssim
M^{\frac{2-p}{p}}_{\varphi}
\|f\|_{L^p(G)},
\end{equation}
which would give \eqref{EQ:Paley-LCG},
and where we set $M_{\varphi}:=\sup_{t>0}t\int\limits_{\substack{t\in\RR_+\\ \varphi(t)\geq s}}dt$. 
We will show that $T$ is of weak-type (2,2) and of weak-type (1,1). 
More precisely, with the distribution function $\nu$ 
,
we show that
\begin{eqnarray}
\label{EQ:THM:Paley_inequality_weak_1}
\nu(y;Tf)
&\leq &
\left(\frac{M_2\|f\|_{L^2(G)}}{y}\right)^2  \quad\text{with norm } M_2 = 1,\\
\label{EQ:THM:Paley_inequality_weak_2}
\nu(y;Tf)
&\leq &
\frac{M_1\|f\|_{L^1(G)}}{y}  \qquad\text{with norm } M_1 = M_{\varphi}, 
\end{eqnarray}
where $\nu$ is defined in \eqref{EQ:nu-definition-2}.
Recall that the distribution function $\nu(y;Tf)$ with respect to the weight $\varphi^2$ is defined as 
\begin{equation*}
\nu(y;Tf)
=
\int\limits_{\substack{t\in\RR_+ \\ \frac{\mu_t(R_{f})}{\varphi(t)}\geq y}}\varphi^2(t)\,dt.
\end{equation*}
Then \eqref{Paley_inequality_alt} would follow from \eqref{EQ:THM:Paley_inequality_weak_1} and \eqref{EQ:THM:Paley_inequality_weak_2} by the Marcin\-kiewicz interpolation theorem.
Now, to show \eqref{EQ:THM:Paley_inequality_weak_1}, using Plancherel's identity \eqref{EQ:plancherel},
we get
\begin{multline*}
y^2
\nu(y;Tf)
\leq
\|Tf\|^2_{L^p(\RR_+,\nu)}
=
\int\limits_{\substack{\RR_+}}
\left(
\frac{\mu_t(R_{f})}{\varphi(t)}
\right)^2
\varphi^2(t)
\,dt
\\=
\int\limits_{\substack{\RR_+}}
\mu^2_t(R_{f}) dt
=
\|R_{f}\|^2_{L^2(VN_R(G))}
=
\|f\|^2_{L^2(G)}.
\end{multline*}
Thus, $T$ is of weak-type (2,2) with norm $M_2\leq1$.
Further, we show that $T$ is of weak-type (1,1) with norm $M_1=M_{\varphi}$; more precisely, we show that
\begin{equation}
\label{weak_type}
\nu\{t\in\RR_+ \colon \frac{\mu_t(R_{f})}{\varphi(t)} > y\}
\lesssim
M_{\varphi}\,
\dfrac{\|f\|_{L^1(G)}}{y}.
\end{equation}
The left-hand side here is the integral $\int\varphi^2(t)\,dt$ taken over those $t\in\RR_+$ for which $\dfrac{\mu_t(R_{f})}{\varphi(t)}>y$. From the definition of the Fourier transform  it follows that
\begin{equation}
\label{EQ:mu-t-L1}
\mu_t(R_{f})\leq\|f\|_{L^1(G)}.
\end{equation}
Indeed, from the Definition \ref{DEF:mu-t}, we have
\begin{equation*}
\mu_t(R_{f})
\leq
\|R_{f}\|_{L^2(G)\to L^2(G)}.
\end{equation*}
The Young inequality for convolution (e.g. \cite[p. 52, Proposition 2.39]{Folland2016}) yields
\begin{equation*}
\|R_{f}g\|_{L^2(G)}
\leq
\|f\|_{L^1(G)}
\|g\|_{L^2(G)}.
\end{equation*}
Thus
\begin{equation*}
\|R_{f}\|_{L^2(G)\to L^2(G)}
\leq \|f\|_{L^1(G)}.
\end{equation*}
This proves \eqref{EQ:mu-t-L1}.
Therefore, we have
$$
y<\frac{\mu_t(R_{f})}{\varphi(t)}
\leq
\frac{\|f\|_{L^1(G)}}{\varphi(t)}.
$$
Using this, we get
$$
\left\{
t\in\RR_+
\colon 
\frac{\mu_t(R_{f})}{\varphi(t)}>y
\right\}
\subset
\left\{
t\in\RR_+
\colon 
\frac{\|f\|_{L^1(G)}}{\varphi(t)}>y
\right\}
$$
for any $y>0$. Consequently,
$$
\nu\left\{
t\in\RR_+
\colon 
\frac{\mu_t(R_{f})}{\varphi(t)}>y
\right\}
\leq
\nu\left\{
t\in\RR_+
\colon 
\frac{\|f\|_{L^1(G)}}{\varphi(t)}>y
\right\}.
$$
Setting $v:=\frac{\|f\|_{L^1(G)}}{y}$, we get

\begin{equation}
\label{PI_intermed_est_1}
\nu\left\{
t\in\RR_+
\colon 
\frac{\mu_t(R_{f})}{\varphi(t)}>y
\right\}
\leq
\int\limits_{\substack{t\in\RR_+ \\ \varphi(t)\leq v}}
\varphi^2(t)\,dt.
\end{equation}
We now claim that
\begin{equation}\label{EQ:aux1}
\int\limits_{\substack{t\in\RR_+ \\ \varphi(t)\leq v}}
\varphi^2(t)\,dt
\lesssim 
M_{\varphi}
v.
\end{equation}
Indeed, first we notice that we have
$$
\int\limits_{\substack{t\in\RR_+ \\ \varphi(t)\leq v}}
\varphi^2(t)\,dt
=
\int\limits_{\substack{t\in\RR_+ \\ \varphi(t)\leq v}}
dt\int\limits^{\varphi^2(t)}_0 d\tau.
$$
We can interchange the order of integration to get
$$
\int\limits_{\substack{t\in\RR_+ \\ \varphi(t)\leq v}}
dt\int\limits^{\varphi^2(t)}_0 d\tau
=
\int\limits^{v^2}_0 d\tau 
\int
\limits_{\substack{t\in\RR_+ \\ \tau^{\frac12}\leq \varphi(t)\leq v}}
dt.
$$
Further, we make a substitution $\tau=s^2$, yielding
\begin{equation*}
\int\limits^{v^2}_0 d\tau \int\limits_{\substack{t\in\RR_+ \\ 
		\tau^{\frac12}\leq \varphi(t)\leq v}}dt
=
2\int\limits^{v}_0 s\,ds
\int\limits_{\substack{s\in\RR_+ \\ s \leq \varphi(t)\leq v}}dt
\leq
2\int\limits^{v}_0 s\,ds 
\int\limits_{\substack{t\in\RR_+ \\ s \leq \varphi(t)}}dt.
\end{equation*}
Since 
$$
s\int\limits_{\substack{t\in\RR_+ \\ s \leq \varphi(t) } } dt
\leq 
\sup_{s>0}s\int\limits_{\substack{t\in\RR_+ \\ s \leq \varphi(t) } } dt
=M_{\varphi}
$$
is finite by the assumption that $M_{\varphi}<\infty$, we have
$$
2\int\limits^{v}_0 s\,ds
\int\limits_{\substack{t\in\RR_+ \\ s \leq \varphi(t) } } dt
\lesssim 
M_{\varphi} v.
$$
This proves \eqref{EQ:aux1} and hence also \eqref{weak_type}.
Thus, we have proved inequalities
\eqref{EQ:THM:Paley_inequality_weak_1} and
\eqref{EQ:THM:Paley_inequality_weak_2}.
Then by using the Marcinkiewicz interpolation theorem  with $p_1=1$, $p_2=2$ and 
$\frac1p=1-\theta+\frac{\theta}2$ we now obtain
$$
\left(
\int\limits_{\substack{\RR_+}}
\left(
\frac{\mu_t(R_{f})}{\varphi(t)}
\right)^p
\varphi^2(t)dt
\right)^{\frac1p}
 =
\|Af\|_{L^p(\RR_+,\nu)}
\lesssim
M^{\frac{2-p}{p}}_{\varphi}
\|f\|_{L^p(G)}.
$$
This completes the proof of Theorem \ref{THM:Paley-LCG}.
\end{proof}


Further, we recall a result on the interpolation of weighted spaces from \cite{BL2011}:
\begin{thm}[Interpolation of weighted spaces]
\label{THM:L_p-weighted-interpolation}
 Let  $d\mu_0(x)=\omega_0(x)d\mu(x),\\ d\mu_1(x)=\omega_1(x)d\mu(x)$, and write $L^p(\omega)=L^p(\omega d\mu)$ for the weight $\omega$. Suppose that $0<p_0,p_1<\infty$. Then 
$$
	(L^{p_0}(\omega_0), L^{p_1}(\omega_1))_{\theta,p}=L^p(\omega),
$$
where $0<\theta<1,\frac1p=\frac{1-\theta}{p_0}+\frac{\theta}{p_1}$, and 
$\omega=\omega^{p\frac{1-\theta}{p_0}}_0 \omega^{p\frac{\theta}{p_1}}_1$.
\end{thm}

From this, interpolating between the Paley-type inequality \eqref{EQ:Paley-LCG}  in Theorem \ref{THM:Paley-LCG} 
and  Hausdorff-Young inequality  \eqref{EQ:HY-LCG-2}, 
we readily obtain an inequality that will be crucial for our 
consequent analysis of $L^{p}$-$L^{q}$ multipliers:
\begin{thm}[Hausdorff-Young-Paley inequality]
	\label{THM:HYP-LCG}
	Let $G$ be a locally compact unimodular separable group. 
	Let $1<p\leq b \leq p'<\infty$. 
	If a positive function $\varphi(t)$, $t\in\RR_+$, satisfies condition 
	 \begin{equation}
 \label{EQ:weak_symbol_estimate2}
	 M_\varphi:=\sup_{s>0}s\int\limits_{\substack{t\in\RR_+\\ \varphi(t)\geq s }}dt<\infty,
	 \end{equation}
then for all $f\in L^p(G)$ we have
	\begin{equation}
	\label{EQ:HYP-LCG}
	\left(
	\int\limits_{\RR_+}
	\left(
	\mu_t(R_f)
	{\varphi(t)}^{\frac1b-\frac{1}{p'}}
	\right)^b
	dt
	\right)^{\frac1b}
	\lesssim
	M_\varphi^{\frac1b-\frac1{p'}}
	\|f\|_{L^p(G)}.
	\end{equation}
\end{thm}

Naturally, this reduces to the Hausdorff-Young inequality \eqref{EQ:HY-LCG-2} when $b=p'$ and to 
the Paley inequality in \eqref{EQ:Paley-LCG} when $b=p$.

\section{Traces and singular numbers on von Neumann algebras}
\label{SEC:traces}

Some properties of traces shall be used in the proofs of our theorems. Moreover, we will need to use Haagerup's version of traces in the case of non-unimodular groups.
Therefore, we give a brief background on traces summarising the results that will be used in the sequel. For some description of measurable fields of operators and links to the representation theory and general von Neumann and $C^*$-algebras we refer to \cite[Appendices B and C]{FR2016}.
The following definition is taken from \cite[Definition I.6.1, p.93]{VNA-Dixmier-1981}:

\begin{defn} Let $M$ be a von Neumann algebra. A trace on the positive part $M_+=\{A\in M\colon A^*=A>0\}$ of $M$ is a functional $\tau$ defined on $M_+$, taking non-negative, possibly infinite, real values, possessing the following properties:
\begin{itemize}
\item If $A\in M_+$ and $B\in M_+$, we have $\tau(A+B)=\tau(A)+\tau(B)$;
\item If $A\in M_+$ and $\lambda\in\RR_+$, we have $\tau(\lambda A)=\lambda\tau(A)$ (with the convention that $0\cdot+\infty=0$);
\item If $A\in M_+$ and if $U$ is a unitary operator of $M$, then $\tau(UAU^{-1})=\tau(A)$.
\end{itemize}
We say that $\tau$ is faithful (or exact) if the condition $A\in M_+,\,\tau(A)=0$, imply that $A=0$.
We say that $\tau$ is finite if $\tau(A)<+\infty$ for all $A\in M_+$.
We say that $\tau$ is { semifinite} if, for each $A\in M_+$, $\tau(A)$ is the supremum of the numbers $\tau(B)$ over those $B\in M_+$ such that $B\leq A$ and $\tau(B)<+\infty$.
We say that $\tau$ is normal if, for each increasing filtering set $\mathcal{S}\subset M_+$ with supremum $S\in M_+$, $\tau(S)$ is the supremum of $\{\tau(B)\}_{B\in\mathcal{S}}$.
%
A von Neumann algebra $M$ is said to be {\em semifinite} if there exists a semifinite faithful normal trace $\tau$ on $M_+$.
\end{defn}

Let $X$ be a Borel space, $\nu$ a positive measure on $X$ and let $\lambda \mapsto \H_{\lambda}$ be a measurable field of Hilbert spaces $\H_{\lambda}$. For every $\lambda\in X$, let $A_{\lambda}$ be an element of $B(\H_{\lambda})$, i.e. a linear bounded operator on $\H_{\lambda}$. The mapping $\lambda\mapsto A_{\lambda}$ is called a field of bounded linear operators over $X$. Measurable fields of operators associated to group von Neumann algebras have been discussed in detail in \cite[Appendix B]{FR2016}.

\begin{defn} A measurable field $\{A_{\lambda}\}_{\lambda\in X}$ is said to be essentially bounded if the essential supremum of the function $\lambda \mapsto \|A_{\lambda}\|_{B(\H_{\lambda})}$ is finite.
A linear operator $A\colon \H\to \H$ is said to be decomposable if it is defined as an essentially bounded measurable field $\{A_{\lambda}\}_{\lambda\in X}$. We then write
$$
A=
\bigoplus\int A_{\lambda}d\nu(\lambda).
$$
\end{defn}
For every $\lambda\in X$, let $M_{\lambda}$ be a von Neumann algebra in $\H_{\lambda}$. The mapping $\lambda\mapsto M_{\lambda}$ is called a field of von Neumann algebras over $X$.
\begin{defn} A von Neumann algebra $M\subset B(\H)$ is said to be decomposable if it is defined by a measurable field $\lambda \mapsto M_{\lambda}$ of von Neumann algebras. We then write
$$
M
=
\bigoplus\int M_{\lambda}d\nu(\lambda),
$$
\end{defn}
Every abelian algebra $\mathcal{C}$ is isometrically isomorphic \cite[Theorem 1, p.132]{VNA-Dixmier-1981} to $L^{\infty}(X,\nu)$, where $X$ is a locally compact space with a positive measure $\nu$ on $X$.
The importance of abelian subalgebras is motivated by the following
\begin{thm}[{{\cite[Theorem 7, p. 460]{RO1949}}}] 
Let $M$ be a von Neumann algebra and let $\mathcal{C}=M\cap M^{!}\cong L^{\infty}(X,\nu)$ be its center. Then $M=\bigoplus\limits_{X}\int M_{\lambda}d\nu(\lambda)$,
where each $M_{\lambda}$ is a factor, i.e. $M_{\lambda}\cap M^{!}_{\lambda}=\text{multiplies of the identity operator $I$}$.
\end{thm}

We recall the  basic result on central decomposition of traces of von Neumann algebras.

\begin{thm} Let $M=\bigoplus\limits M_{\lambda}d\nu(\lambda)$ be a semifinite decomposable von Neumann algebra. Suppose that $\nu$ is standard. Then we have
\begin{enumerate}
\item The $M_{\lambda}$'s are semifinite almost everywhere.
\item Let $\tau$ be a semifinite faithful normal trace on $M_+$. Then there exists a measurable field $\lambda\mapsto \tau_{\lambda}$ of semifinite faithful normal traces on the $(M_{\lambda})_+$'s such that 
\begin{equation}
\tau
=
\int\limits \tau_{\lambda}d\nu(\lambda).
\end{equation}
\end{enumerate}
\end{thm}
It can be seen that $A\in M^+$ if and only if $A=(A^{1/2})^*A^{1/2}$.
\begin{ex} Let $G$ be a locally compact unimodular group with $\VN_R(G)$ the group von Neumann algebra generated by the right regular representation $\pi_R$ of $G$.  Let $A$ be a linear bounded operator commuting with the left regular representation. Then by the double commutant theorem $A\in \VN_R(G)$ and its action is given
$$
L^2(G)\ni h\mapsto Ah=h\ast K_A\in L^2(G),
$$
where $K_A$ is its convolution kernel.
We can define a trace $\tau$ on $\VN^+_R(G)$ by
\begin{equation}
\tau(A):=
\begin{cases}
\|K_{A^{\frac12}}\|^2_{L^2(G)},\quad \text{ if $K_{A^{1/2}}\in L^2(G)$},
\\
\infty,\quad \text{ otherwise}.
\end{cases}
\end{equation}

Let $\mathcal{C}=\VN_R(G)\cap \VN_R(G)^{!}$. 
The abelian algebra $\mathcal{C}$ can be identified with the unitary dual $\Gh$ via a canonical map. For more details we refer to \cite{Dixmier1977} or \cite[Theorem 7.37, p. 227]{Folland2016}.
%
The reduction theory allows us to decompose $\tau$ with respect to the center $\VN_R(G)\cap \VN_R(G)^{!}$ of the group algebra $\VN_R(G)$
\begin{equation}
\tau(A)
=
\int\limits_{\Gh}\tau_{\pi}(A_{\pi})d\pi,
\end{equation}
where $A=\bigoplus\int\limits A_{\pi}d\pi$. 
\end{ex}

The trace $\tau$ on $M$ can also be extended to the $*-$algebra $S(M)$of all $\tau$-measurable operators.
\begin{prop} 
\label{PROP:trace-spectral-computed}
Let $(M,\tau)$ be a von Neumann algebra and let $A$ be a $\tau$-measurable linear operator. Assume that $\varphi$ is a Borel function on $\sp(\left|\mathcal{L}\right|)\subset [0,+\infty)$. Then we have
\begin{equation}
\label{EQ:trace-spectral-computed}
\tau(\varphi(\left|A\right|))
=
\int\limits^{+\infty}_0 \varphi(t)d\mu(t),
\end{equation}
where $\mu_t=\tau(E_t)$ and 
$$
\left|\mathcal{L}\right|=\int\limits^{+\infty}_0 tdE_t(\left|A\right|).
$$
\end{prop}
Although the equality \eqref{EQ:trace-spectral-computed} has been used to define a trace on the algeba $S(M)$.
However, the authors prefer to prove \eqref{EQ:trace-spectral-computed} while fully acknowledging the influence of \cite{HayesPhD2014}.
\begin{proof}[Proof of Proposition \ref{PROP:trace-spectral-computed}]
For the spectral measure we can take the family $\{E_{[0,t)}\}_{t\geq 0}$ of spectral projections $E_{(0,t)}$ corresponding to the intervals $[0,t)$.  The reader can check that the spectral measure axioms hold true.

The trace $\tau$ is continuous with respect to $\tau$-measure.
In the view of the monotone convergence theorem (see \cite[Theorem 3.5]{ThierryKosaki1986})  we can  assume, without loss of generality, that $A$ is a bounded $\tau$-measurable operator. Indeed, for every $\tau$-measurable operator $\left|A\right|$ there exists a sequence $\{A_n\}$ of   of $\tau$-measurable bounded operators
$$
A_n=\int\limits^n_0 tdE_t(\left|A\right|)\leq A
$$
converging to $A$ in $\tau$-measure topology.  Then, taking the limit
$$
\lim_{n\to\infty} \tau(A_n)=\tau(A),
$$
we justify the claim.
We notice that every Borel function can be uniformly approximated by bounded Borel functions. Thus, we concentrate to establish \eqref{EQ:trace-spectral-computed} for bounded measurable $A$ and bounded Borel functions $\varphi$ on $[0,\|A\|_{B(\H)}]$.
By the spectral mapping theorem
$$
\sp(\varphi(A))
=
\varphi
(
[0,\|A\|_{B(\H)}]
)
$$
Let $0\leq \lambda_1\leq \lambda_2\leq \ldots\lambda_N$ be a partition of the interval 
$\varphi
(
[0,\|A\|_{B(\H)}]
)
$. Then the Riemannian-like sums
$$
R_N=\sum\limits^N_{k=1}\lambda_kE_{\varphi^{-1}(\lambda_{k-1},\lambda_k)}
(\left|A\right|)
$$
converge to $\varphi(A)$ in $\tau$-measure topology.
The trace $\tau$ on $R_N$ is given by
\begin{equation}
\label{EQ:Riemann}
\tau(R_N)
=
\int\limits^N_{k=1}\lambda_k\tau(E_{\varphi^{-1}(\lambda_{k-1},\lambda_k)}
(\left|A\right|)
).
\end{equation}
One can notice that sum in \eqref{EQ:Riemann} is a Lebesuge integral sum 
$$
\sum\limits^N_{k=1}\lambda_k \mu_{(\lambda_{k-1},\lambda_k)}
$$
for the integral
$$
\int\limits^{\|A\|}_{0}\varphi(t)d\mu(t),
$$
where we set the measure $\mu((a,b))=\tau(E_{(a,b)}),\quad (a,b)\subset [0,\|A\|_{B(\H)}]$.

\end{proof}
For the sake of the exposition clarity  we now formulate some properties of the distribution function $d_A$ which we will be using in the proofs.
\begin{prop}
\label{PROP:mu-t-properties}
 Let $A\in S(M)$. Then we have
\begin{align}
\label{EQ:d-s-right-continuity}
&d_A(\mu_A(t))\leq t;
\\
\label{EQ:prop-3}
&\mu_A(t)>s \quad\text{ if and only if }\quad t<d_A(s);
\\
\label{EQ:prop-16}
&\sup_{t>0}t^{\alpha}\mu_A(t)
=
\sup_{s>0}s[d_A(s)]^{\alpha} \quad \text{ for }\; 0<\alpha<\infty.
\end{align}
\end{prop}
The proof of this proposition is almost verbatim to the proof of \cite[Proposition 1.4.5 on page 46]{Grafakos_2008}. The word `almost' stands for the right-continuity of the non-commutative distribution function $d_A(s)$ which is discussed after \cite[Definition 1.3 on page 272]{ThierryKosaki1986}. Therefore, in the following proof we shall use the right-continuity of $d_A(s)$ without any justification.
\begin{proof}[Proof of Proposition \ref{PROP:mu-t-properties}]
Let $s_n\in\{s>0\colon d_A(s)\leq t\}$ be such that $s_n \searrow  \mu_A(t)$. Then $d_A(s_n)\leq t$ and the right-continuity of $d_{A}$ implies that $d_A(\mu_A(t))\leq t$. This proves \eqref{EQ:d-s-right-continuity}.
Now, we apply this property to derive \eqref{EQ:prop-3}.
If $s<\mu_A(t)=\inf\{s>0 \colon d_A(s)\leq t\}$, then 
$s$ does not belong to the set $\{s>0 \colon d_A(s)\leq t\}\implies d_A(s)>t$. 
Conversely, if for some $t$ and $s$, we had $\mu_A(t)<s$, then the application of  $d_A$ and property \eqref{EQ:prop-3} would yield the contradiciton $d_A(s)\leq d_{A}(\mu_A(t))\leq t$. Property \eqref{EQ:prop-3} is established.
Finally, we show \eqref{EQ:prop-16}. Given $s>0$, pick $\eps$ satisfying $0<\eps<s$. Property \eqref{EQ:prop-3} yields $\mu_A(d_A(s)-\eps)>s$ which implies that
\begin{equation}
\sup_{t>0}t^{\alpha}\mu_A(t)
\geq
(d_A(s)-\eps)^{\alpha}\mu_A(d_A(s)-\eps)>(d_A(s)-\eps)^{\alpha}s.
\end{equation}
We first let $\eps\to 0$ and then take the supremum over all $s>0$ to obtain one direction. Conversely, given $t>0$, pick $0<\eps<\mu_A(t)$. Property \eqref{EQ:prop-3}
yields  that $d_A(\mu_A(t)-\eps)>t$. This implies that $\sup_{s>0}s(d_A(s))^{\alpha}\geq (\mu_A(t)-\eps)(d_A(\mu_A(t)-\eps))^{\alpha}>(\mu_A(t)-\eps)t^{\alpha}$. We first let $\eps\to 0$ and then take the supremum over all $t>0$ to obtain the opposite direction of \eqref{EQ:prop-16}.
\end{proof}
We recall the following result which will be partially used.
\begin{thm}[{{\cite[Theorem 4, p. 412]{Segal1953}}}] If operators $A$ and $B$ are $\tau$-measurable with respect to a von Neumann algebra $M$, then so are $A^*,A+B$ and $AB$, i.e. the maps
\begin{align}
+\colon M \times M \ni (A,B) &\mapsto A+B\in M,
\\
\cdot \colon M\times M \ni (A,B)&\mapsto A B\in M,
\\
*\colon M \ni A &\mapsto A^*\in M
\end{align}
are well-defined.
\end{thm}
Here we formulate some properties of $\mu_t$ that we use in the proof of Theorem \ref{THM:upper-bound}.

\begin{lem}[{{\cite[Lemma 2.5, p. 275]{ThierryKosaki1986}}}] 
\label{LEM:mu-t-A-properties}
Let $A,B$ be $\tau$-measurable operators. Then the following properties  hold true.
\begin{enumerate}
\item The map $(0,+\infty)\ni t \mapsto \mu_t(A)$ is non-increasing and continuous from the right. Moreover, 
\begin{equation}
\lim_{t\to 0}\mu_t(A)=\|A\|\in[0,+\infty].
\end{equation} 
\item 
\begin{equation}\label{EQ:muadj}
\mu_t(A)=\mu_t(A^*).
\end{equation}
\item 
\label{LEM:mu-t-A-properties-3}
\begin{equation}
\label{EQ:mu-t-A-properties-3}
\mu_{t+s}(AB)\leq\mu_t(A)\mu_s(B).
\end{equation}
\item
\label{LEM:mu-t-A-properties-4}
\begin{equation}
\mu_t(ACB)
\leq
\|A\|\|B\|\mu_t(C),\quad \text{ for any $\tau$-measurable operator $C$}.
\end{equation}
\item 
\label{LEM:mu-t-A-properties-5}
For any continuous increasing function $f$ on $[0,+\infty)$ we have
\begin{equation}
\mu_t(f(|A|))
=
f(\mu_t(|A|)).
\end{equation}
\end{enumerate}
\end{lem}
In Lemma \ref{LEM:mu-t-A-properties}, we formulate only the properties we use, whereas in
 \cite[Lemma 2.5, p. 275]{ThierryKosaki1986} the reader can find more details.

\section{Nikolskii inequality on locally compact groups}
\label{SEC:Nik}

In this section we establish the Nikolskii inequality (sometimes called the reverse H\"older inequality) in the setting of locally compact groups.
It will be instrumental in proving the Lizorkin type multiplier theorem in Theorem \ref{THM:Lizorkin-LCG}.

Let $G$ be a locally compact unimodular separable group and $\VN_R(G)$ its right group von Neumann algebra with trace $\tau$. We shall denote by $\FT^R[f]$ the right Fourier transform of $f\in L^1(G)$, i.e.
\begin{equation}\label{EQ:FTR}
\FT^R[f]=R_f\colon L^2(G)\ni h\mapsto \FT^R[f](h)=R_f[h]=h\ast f\in L^2(G).
\end{equation} 
The reason to introduce the new notation $\FT^R[f]$ is to emphasise the connection with the classical Nikolskii inequality \cite{Nikolskii1951}. 

Let us denote by $\supp^R(\widehat{f})$ the subspace of $L^2(G)$ orthogonal to the kernel $\Ker(\FT^R[f])$ of the Fourier transform $\FT^R[f]$, i.e.
\begin{equation}
\supp^R[\widehat{f}]:=\Ker\left({\FT}^R[f]\right)^{\perp},
\end{equation}
where $\Ker(\FT[f])\subset L^2(G)$ is the kernel of the operator $\FT^R[f]$ in \eqref{EQ:FTR}. 

We note that the classical Nikolskii inequality is an $L^p$-$L^q$ estimate for norms of the same functions for $p<q$ so that the Fourier transforms of the functions under consideration must have bounded support. 

In \cite{NRT2014, NRT2015} Nikolskii inequality has been established on compact Lie groups and on compact homogeneous manifolds, respectively, for functions with bounded support of the noncommutative Fourier coefficients. In \cite{CardonaRuzhansky2016} the Nikolskii inequality was proved in the setting of graded groups: Let $G$ be a graded Lie group of homogeneous dimension $Q$ and let $\mathcal{R}$ be a positive Rockland operator of order $\nu$. For every $L>0$ let us consider the operator, defined by the spectral theory
\begin{equation}
T_Lf:=\chi_{L}(\mathcal{R})f,
\end{equation}
where $\chi_L$ is the characteristic function of the interval $[0,L]$.
In these notations, it was shown in \cite[Theorem 3.1]{CardonaRuzhansky2016} that we have
\begin{equation}\label{EQ:Nik-gr}
\|T_L\|_{L^q(G)}
\leq
C L^{\frac{Q}{\nu}\left(\frac1p-\frac1q\right)}\|T_L\|_{L^p(G)},\quad 1\leq p\leq q\leq \infty,
\end{equation}
with constant $C$ explicitly depending on the spectral resolution of the Rockland operator $\mathcal R$.

The main question in the setting of locally compact groups is to find an analogue of the condition for bounded support of Fourier transforms since we may not have a canonical operator to use its spectral decomposition for the definition of the bounded spectrum.

The version that will work well for our purposes is the following.

Let  $P_{\supp^R[\widehat{f}]}$ be the orthogonal projector onto the support $\supp^R[\widehat{f}]$. 
We say that $f\in L^1(G)$ has {\em bounded spectrum} if $\tau(P_{\supp^R[\widehat{f}]})<+\infty$.

\begin{thm} 
\label{THM:Nikolsky-LCG}
Let $G$ be a locally compact separable unimodular group.
Let $1< q\leq \infty$ and $1<p\leq \min(2,q)$. Assume that $\tau(P_{\supp^R[\widehat{f}]})<\infty$, where $P_{\supp^R[\widehat{f}]}$ denotes  the orthogonal projector onto the support $\supp^R[\widehat{f}]$. Then we have
\begin{equation}
\label{EQ:Nikolsky-LCG}
\|f\|_{L^q(G)}
\lesssim
\left(
\tau(P_{\supp^R[\widehat{f}]})
\right)^{\frac1p-\frac1q}
\|f\|_{L^p(G)},
\end{equation}
with the constant in \eqref{EQ:Nikolsky-LCG} independent of $f$.
\end{thm}
We shall call \eqref{EQ:Nikolsky-LCG} the {\it Nikolskii inequality} on topological groups.

In the case of graded Lie groups we have $\tau(P_{\supp^R[\widehat{f}]})=L^{\frac{Q}{\nu}}$, so that \eqref{EQ:Nikolsky-LCG} in Theorem \ref{THM:Nikolsky-LCG} recovers \eqref{EQ:Nik-gr} for the corresponding ranges of exponents $p$ and $q$.

In \cite{NRT2014, NRT2015} Nikolskii inequality has been established on compact Lie groups and on compact homogeneous manifolds, respectively.
We prove Theorem \ref{THM:Nikolsky-LCG} along the lines of the proof in \cite{NRT2014} adapting the latter to the setting of locally compact groups.

\begin{proof}[Proof of Theorem \ref{THM:Nikolsky-LCG}]
We will give the proof of \eqref{EQ:Nikolsky-LCG} in three steps. We can abbreviate $\mathcal{F}^R[f]$ in the proof to simply writing $\mathcal{F}[f]$.

\medskip
\noindent
{\underline{Step 1.}
 The case $p=2$ and $q=\infty$.}
We have (by e.g. \cite[Proposition A.1.2. p. 216]{HayesPhD2014}) that
\begin{equation}
\label{EQ:obv_Nik}
\left|\Tr(\widehat{f}(\pi)\pi(x))\right|\leq\Tr\left|\widehat{f}(\pi)\right|,\quad x\in G.
\end{equation}
We notice that
$$
\mathcal{F}[f]P_{\supp^R[\widehat{f}]}
=
\mathcal{F}[f].
$$
Then by \cite[Lemma 2.6, p. 277]{ThierryKosaki1986}, we have
\begin{equation}
\label{EQ:finite-2}
\mu_s(\mathcal{F}[f])=0,\quad s\geq \tau(P_{\supp^R[\widehat{f}]}).
\end{equation}
From now on we shall denote $t:=\tau(P_{\supp^R[\widehat{f}]})$ throughout the proof.
Further, the application of \cite[Proposition 2.7, p.277]{ThierryKosaki1986} yields
\begin{equation}
\label{EQ:finite-spectrum}
\tau(\left|\mathcal{F}[f]\right|)
=
\int\limits^{\infty}_0 \mu_s(\mathcal{F}[f])\,ds
=
\int\limits^{\tau(P_{\supp^R[\widehat{f}]})}_0 \mu_s(\mathcal{F}[f])\,ds,
\end{equation}
where we used \eqref{EQ:finite-2} in the last equality.
Combining \eqref{EQ:finite-spectrum} and \eqref{EQ:obv_Nik}, we obtain
\begin{equation}
\label{EQ:N-LCG-1}
\begin{split}
\|f\|_{L^{\infty}(G)}&\le \int\limits_{\Gh}\Tr\left|\widehat{f}(\pi)\right|d\pi
=
\tau(\left|\mathcal{F}[f]\right|)
=
\int\limits^t_0\mu_s(\mathcal{F}[f])ds
\\& 
\le
\left(
\int\limits^t_0ds
\right)^{\frac12}
\left(
\int\limits^t_0 \mu^2_s(\mathcal{F}[f])ds
\right)^{\frac12}
=
\sqrt{\tau(P_{\supp^R[\widehat{f}]})}
\|f\|_{L^2(G)},
\end{split}
\end{equation}
where in the last inequality we used the Plancherel
identity.

\medskip
\noindent
{\underline{Step 2.}
 The case $p=2$ and $2<q\leq\infty$.}
 We take
$1\leq q'<2$ so that
$\frac1q+\frac{1}{q'}=1.$
We set $r:=\frac2{q'}$ so that its dual index $r'$
satisfies $\frac{1}{r'}=1-\frac{q'}{2}.$
By the Hausdorff-Young inequality in \eqref{EQ:HY-LCG-2},
and
 by H\"older's inequality, we obtain
\[
\begin{split}
\|f\|_{L^{q}(G)}&\le \|\mathcal{F}[f]\|_{L^{q'}(\VN_R(G))}=
\left(\int\limits^t_0\mu^{q'}_s(\mathcal{F}[f])ds
\right)^{\frac{1}{q'}}\\
& \le
\left( \int\limits^t_0 ds\right)^{\frac{1}{q'r'}}
\left( \int\limits^t_0 \mu^{q'r}_s(\mathcal{F}[f])\,ds\right)^{\frac{1}{q'r}} \\
& =
\left( \int\limits^t_0 ds\right)^{\frac{1}{q'}-\frac12}
\left( \int\limits^t_0 \mu^2_s(\mathcal{F}[f])\,ds\right)^{\frac{1}{2}} \\
& 
\leq
\left( \int\limits^t_0 ds\right)^{\frac{1}{q'}-\frac12}
\left( \int\limits^{\infty}_0 \mu^2_s(\mathcal{F}[f])\,ds\right)^{\frac{1}{2}} \\
& =
\tau(P_{\supp^R[\widehat{f}]})^{\frac12-\frac1q} \|f\|_{L^{2}(G)},
\end{split}
\]
where we have used that
$\frac{q'r}{2}=1$.

\medskip
\noindent
{\underline{Step 3.}}
If $p=\min(2,q)$ and $p\not=2$, then $p=q$ and there is nothing to prove.
For $1<p<\min(2,q)$, we claim to have
$$
\|f\|_{L^{q}(G)}\le \tau(P_{\supp^R[\widehat{f}]})^{(1/p-1/q)} \|f\|_{L^{p}(G)}.
$$
Indeed, if $q=\infty$, for $f\not\equiv 0$, we get
\begin{equation}\label{EQ:df}
\begin{split}
\|f\|_{L^{2}}
&=
\||f|^{1-p/2} |f|^{p/2}\|_{L^{2}}
\le
\||f|^{1-p/2}\|_{L^{\infty}}
\||f|^{p/2}\|_{L^{2}}
\\
&=
\|f\|_{L^{\infty}}^{1-p/2}
\||f|^{p/2}\|_{L^{2}}
=
\|f\|_{L^{\infty}} \|f\|_{L^{\infty}}^{-p/2}
\||f|^{p/2}\|_{L^{2}} \\
&
=
\|f\|_{L^{\infty}} \|f\|^{-p/2}_{L^{\infty}}
\|f\|_{L^{p}}^{p/2}
\\
& \le
\tau(P_{\supp^R[\widehat{f}]})^{1/2}
\|f\|_{L^{2}} \|f\|^{-p/2}_{L^{\infty}} \|f\|_{L^{p}}^{p/2},
\end{split}
\end{equation}
where we have used \eqref{EQ:N-LCG-1} in the last line. 
Therefore, using that $f\not\equiv 0$, we have
\begin{equation}\label{EQ:auxpinf}
\begin{split}
\|f\|_{L^{\infty}}
& \le
\tau(P_{\supp^R[\widehat{f}]})^{1/p}
\|f\|_{L^{p}}.
\end{split}
\end{equation}
For $p<q<\infty$ we obtain
\begin{equation}\label{EQ:df1}
\begin{split}
\|f\|_{L^{q}}&= \||f|^{1-p/q} |f|^{p/q}\|_{L^{q}}\le
\|f\|_{L^{\infty}}^{1-p/q} \|f\|_{L^{p}}^{p/q}
\\
& \le
\tau(P_{\supp^R[\widehat{f}]})^{1/p(1-p/q)} \|f\|_{L^{p}}^{1-p/q}
\|f\|_{L^{p}}^{p/q}
=
\tau(P_{\supp^R[\widehat{f}]})^{\frac1p-\frac1q}
\|f\|_{L^p}
,
\end{split}
\end{equation}
where we have used \eqref{EQ:auxpinf}.
\end{proof}
%


Let us now discuss the closure of the space of functions having bounded spectrum. 
We show that the situation is quite subtle, already in the case of compact Lie groups.
This space will naturally appear in the formulations of Lizorkin type theorems.
It will be convenient to measure the rate of growth of the trace of projections to the support of the spectrum: given a function $w=w(t)$ controlling such growth one can use it when passing to the limit from functions with bounded spectrum.

\begin{defn}[The space $L^p_w(G)$] \label{DEF:lpw}
Let $w=w(t)\geq 0$ be a locally integrable function.
Let us denote by $L^p_w(G)$ the space of functions $f\in L^p(G)$ for which there exists a sequence $\{f_t\}_{t>0}$ of functions $f_t\in L^p(G)$ with bounded spectrum 
such that
\begin{eqnarray}
\|f-f_t\|_{L^p(G)}\to 0 \quad\textrm{ as }\quad  t\to \infty,
\\
\tau(P_{\supp^R[\widehat{f_t}]})\leq w(t).
\end{eqnarray}
\end{defn}

It has been shown by Stanton \cite{Stanton1976} that for class functions on semisimple compact Lie groups
the polyhedral Fourier partial sums $S_N f$converge to $f$ in $L^p$ provided that 
$2-\frac1{s+1}<p<2+\frac1s$. Here the number $s$ depends on the root system
$\Rcal$ of the compact Lie group $G$.
We also note that the range of indices $p$ as above is sharp, see
Stanton and Tomas \cite{StantonThomas1976,Stanton-Tomas:AJM-1978} as well as 
Colzani, Giulini and Travaglini \cite{Colzani1989}. We refer to Appendix \ref{SEC:mean_convergence} for more details.
The mean convergence of polyhedral Fourier sums has been investigated in \cite{StantonThomas1976,Stanton1976,Stanton-Tomas:AJM-1978}.  
These results allow us to describe spaces $L^p_{w}(G)$ on compact semisimple Lie groups for suitable choice of $w$ in Example \ref{EX:compact-omega} below.
\begin{ex}  
\label{EX:compact-omega}
Let $G$ be a compact semisimple Lie group of dimension $n$ and rank $l$, and let $\mathbb{T}^l$ be its maximal torus.
We say that a function $f$ is central (or class) if
$$
f(xg)=f(gx),\quad x\in \mathbb{T}^l,g\in G.
$$
We write $L^p_{\inv}(G)$ for the space of central functions $f\in L^p(G)$.
Let $\lambda_1,\lambda_2,\ldots,\lambda_1,\ldots$ denote the eigenvalues of the $n$-th order pseudo-differential operator $(I-\mathcal{L}_{G})^{\frac{n}2}$ counted with multiplicities, where $\mathcal{L}_{G}$ is the bi-invariant Laplacian (Casimir element) on $G$.
We shall enumerate elements $\pi$ of the unitary dual $\Gh$ of $G$ using the eigenvalues $\{\lambda_k\}^{\infty}_{k=1}$, i.e.
\begin{equation}
\label{EQ:enumeration}
(I-\mathcal{L}_{G})^{\frac{n}2}\pi^{k}_{mn}=\lambda_k\pi^{k}_{mn}.
\end{equation}

Let $Q_N\subset \widehat{G}$ be a finite polyhedron of $N$-th order and take $w(N)$ to be the number of the eigenvalues $\lambda_k$ enumerating the elements $\pi^k$ in $Q_N$, i.e.
\begin{equation}
w(N)
=
\sum\limits_{\substack{k\in\NN\\ \lambda_k\in Q_N}}1.
\end{equation}
Then 
\begin{equation}
L^p_{w}(G)
=
\begin{cases}
L^p_{\inv}(G),\quad 2-\frac1{1+s}<p<2+\frac1s,
\\
\varnothing, \quad \text{otherwise}.
\end{cases}
\end{equation}
We refer to Appendix \ref{SEC:mean_convergence} for more details on the approximations by trigonometric functions on compact Lie groups.
\end{ex}

\section{Lizorkin theorem}
\label{SEC:Lizorkin}

In this section we prove an analogue of the Lizorkin theorem for the $L^p$-$L^q$ boundedness of Fourier multipliers for the range of indices $1<p\leq q<\infty$. 
We recall the classical Lizorkin theorem on the real line $\RR$:

\begin{thm}[{{\cite{Lizorkin1967}}}]  \label{THM:Liz-R}
Let $1<p\leq q<\infty$ and let $A$ be a Fourier multiplier on $\RR$ with the symbol $\sigma_A$, i.e.
$$
Af(x)
=
\int_{\RR}e^{2\pi ix\xi }\sigma_A(\xi)\widehat{f}(\xi)d\xi.
$$
Assume that the symbol $\sigma_A(\xi)$ satisfies the following conditions
\begin{eqnarray}
\sup_{\xi\in\RR}|\xi|^{\frac1p-\frac1q}|\sigma_A(\xi)|\leq C<\infty,
\\
\sup_{\xi\in\RR}|\xi|^{\frac1p-\frac1q+1}\left|\frac{d}{d \xi}\sigma_A(\xi)\right|\leq C.
\end{eqnarray}
Then $A\colon L^p(\RR)\to L^q(\RR)$ is a bounded linear operator and
\begin{equation}
\|A\|_{L^p(\RR)\to L^q(\RR)}\lesssim C.
\end{equation}
\end{thm}
There have been recent works extending this statement to higher dimension, as well as to the operators on the torus, see e.g. \cite{Sarybekova2010, Ydyrys2016,Persson2008,Persson2012} deal with the same problem on $G=\TT^n$ and $G=\RR^n$. We start by proving an analogue of Theorem \ref{THM:Liz-R} on general locally compact groups. Consequently, we also derive a version of such theorem on compact Lie groups when symbolic analysis is also possible.

\subsection{Locally compact groups}
\label{SEC:Liz-lcg}

The available information about operators (Fourier multipliers) on locally compact groups is the spectral information measured in terms of the generalised $t$-singular numbers discussed in Section \ref{SEC:prelim}. The following statement gives a condition for the $L^p$-$L^q$ boundedness in the spaces $L^p_w$ which are the closure of the spaces of functions with bounded spectrum as discussed in Definition \ref{DEF:lpw}.

\begin{thm} 
\label{THM:Lizorkin-LCG}
Let $G$ be a locally compact unimodular separable group and let $A$ be a left Fourier multiplier on $G$.
Let  $1<p\leq \min(2,q)$ and $1<q\leq \infty $. Let $w=w(t)\geq 0$ be a locally integrable function.
Then we have
\begin{equation}
\label{EQ:Lizorkin-LCG}
\|A\|_{L^p_{w}(G)\to L^q_{w}(G)}
\lesssim
\sup_{t>0}
w(t)^{\frac1p-\frac1q}
\mu_t(A)
+
\int\limits^{+\infty}_0
w(t)^{\frac1p-\frac1q}
\left(-\frac{d}{dt}\right)
\mu_t(A)\,dt.
\end{equation}
\end{thm}

The function $\mu_t(A)$ is non-increasing and continuous from the right. Therefore, it has derivative $\mu_t(A)$ that exists and is finite almost everywhere with respect to the Lebesgue measure on $\RR_+$.


\begin{proof}[Proof of Theorem \ref{THM:Lizorkin-LCG}]
Let us decompose functions $f\in L^p_{w}(G)$ and $g\in L^{p'}_{w}(G)$ as
$$
f=f_1-f_2,\quad g=g_1-g_2,
$$
into functions $f_i,g_i$ with positive Fourier transform, i.e. with
$$
R_{f_i}\geq 0,\quad R_{g_j}\geq 0,\quad i,j=1,2.
$$
By the linearity of the Fourier transform $\mathcal{F}\colon f\mapsto R_f$, we have
\begin{equation*}
\left|(Af,g)_{L^2(G)}\right|
=
\left|
\sum\limits^{2}_{1}(-1)^{i+j}(Af_i,g_j)_{L^2(G)}
\right|
\leq
\sum\limits^{2}_{i,j=1}\left|(Af_i,g_j)_{L^2(G)}\right|.
\end{equation*}
Hence, without loss of generality, we may further assume that $R_f\geq 0$ and $R_g\geq0$ throughout the proof.
Now, we note that it is sufficient to establish 
\begin{equation}
\label{EQ:sufficient}
\|Af_t\|_{L^q(G)}
\lesssim
\left(
w(t)^{\frac1p-\frac1q}\mu_t(A)
+
\int\limits^t_0 w(s)^{\frac1p-\frac1q}\left[-\frac{d}{ds}\mu_s(A)\right]\,ds\right)
\|f_t\|_{L^p(G)}
\end{equation}
for arbitrary function $f_t$ with $\tau(P_{\supp^R[\widehat{f_t}]})\leq w(t)$.
By the Plancherel identity, we have
\begin{equation*}
\left|
(Af,g)_{L^2(G)}
\right|
=
\left|
\tau(R_{Af}R^*_g)
\right|.
\end{equation*}
By using this and the hypothesis that $A$ is a left Fourier multiplier, i.e. 
$
R_{Af}=AR_f
$
and inequality $|\tau(\cdot)|\leq \tau(|\cdot|)$ from e.g. \cite[Proposition A.1.2. p. 216]{HayesPhD2014}, we get
\begin{equation}
\label{EQ:step-0-Lizorkin}
\begin{aligned}
\left|
(Af,g)_{L^2(G)}
\right|
=
\left|
\tau(R_{Af}R^*_g)
\right|
\leq
\tau(\left|AR_{f} R^*_g\right|)
\\
=
\int\limits^{+\infty}_{0}\mu_s(AR_{f} R^*_g)\,ds
\\
\lesssim
\int\limits^{+\infty}_{0}\mu_s(A)\mu_s(R_{f} R^*_g)\,ds,
\end{aligned}
\end{equation}
where in the second line
we used \cite[Proposition 2.7, p. 227]{ThierryKosaki1986} to express $\tau(\left|AR_{f} R^*_g\right|)$ via the $s$-th generalised singular values $\mu_s(AR_{f} R^*_g),\,s\in\RR_+$.
In the third line in \eqref{EQ:step-0-Lizorkin} we made the substitution $s\to 2s$ and used the sub-multiplicativity $\mu_{2s}(\cdot)\leq \mu_s(\cdot)\mu_s(\cdot)$ of the generalised singular values $\mu_t$ (see Lemma \ref{LEM:mu-t-A-properties}).
By definition
\begin{equation*}
R_f P_{\supp^R[\widehat{f}]}=R_f.
\end{equation*}
Hence, by \cite[Lemma 2.6]{ThierryKosaki1986} we get
\begin{equation}
\mu_s(R_f)=0,\quad s>w(t)\geq \tau(P_{\supp^R_p[\widehat{f}]}).
\end{equation}
Therefore, we get
\begin{equation}
\label{EQ:trig}
\mu_s(R_f R^*_g)=0,\quad s>2w(t)\geq 2\tau(P_{\supp^R[\widehat{f}]})
\end{equation}
in view of $\mu_s(R_f R^*_g)\leq \mu_{\frac{s}2}(R_f)\mu_{{\frac{s}2}}(R^*_g)$.
Taking into account \eqref{EQ:trig}, we get from \eqref{EQ:step-0-Lizorkin}
\begin{equation}
\label{EQ:step-01-Lizorkin}
\left|
(Af,g)_{L^2(G)}
\right|
\leq
\int\limits^{2w(t)}_{0}\mu_s(A)\mu_s(R_{f} R^*_g)\,ds.
\end{equation}

Now, applying Abel transform to \eqref{EQ:step-01-Lizorkin}, we get
\begin{multline}
\label{EQ:step-0}
\int\limits^{2w(t)}_{0}\mu_s(A)\mu_s(R_{f} R^*_g)\,ds
\\ =
\mu_s(A)\int\limits^{s}_0 \mu_u(R_{f} R^*_g)\,du\Big|^{s=2w(t)}_{s=0}
+
\int\limits^{2w(t)}_0 \left[-\frac{d}{ds}\mu_s(A)\right]\int\limits^{s}_0 \mu_u(R_{f} R^*_g)\,du\,ds.
\end{multline}

By the Nikolskii inequality \eqref{EQ:Nikolsky-LCG}, we get
\begin{multline}
\label{EQ:important-step}
\int\limits^{2w(t)}_0 \mu_s(R_{f} R^*_g)\,ds
=
(f,g)_{L^2(G)}
\leq
\|f\|_{L^q(G)}\|g\|_{L^{q'}(G)}
\\ \lesssim
{\tau(P_{\supp^R[\widehat{f}]})}^{\frac1p-\frac1q}
\|f\|_{L^p(G)}\|g\|_{L^{q'}(G)}
\leq
w(t)^{\frac1p-\frac1q}
\|f\|_{L^p(G)}\|g\|_{L^{q'}(G)}
.
\end{multline}
Combining \eqref{EQ:step-0-Lizorkin}, \eqref{EQ:important-step} and \eqref{EQ:step-0}, we get
\begin{multline*}
\left|
(Af_t,g)_{L^2(G)}
\right|
\lesssim \\
\left(
w(t)^{\frac1p-\frac1q}\mu_t(A)
+
\int\limits^{2w(t)}_0 w(s)^{\frac1p-\frac1q}\left[-\frac{d}{ds}\mu_s(A)\right]ds
\right)
\|f_t\|_{L^p(G)}
\|g\|_{L^{q'}(G)}.
\end{multline*}
By the $L^p$-space duality, we immediately get
\begin{equation}
\label{EQ:pre-final}
\|Af_t\|_{L^q(G)}
\lesssim
\left(
w(t)^{\frac1p-\frac1q}\mu_t(A)
+
\int\limits^{2w(t)}_0 w(s)^{\frac1p-\frac1q}\left[-\frac{d}{ds}\mu_s(A)\right]\,ds
\right)\|f_t\|_{L^p(G)},
\end{equation}
for every $f$ with $\tau(\supp^R[\widehat{f}])\leq w(t)$.
Taking supremum over all such $f$, we finally obtain
\begin{equation}
\|Af\|_{L^q}
\leq
\sup_{t>0}
w(t)^{\frac1p-\frac1q}\mu_t(A)
+
\int\limits^{+\infty}_0 w(s)^{\frac1p-\frac1q}\left[-\frac{d}{ds}\mu_s(A)\right]\,ds.
\end{equation}
This completes the proof of Theorem \ref{THM:Lizorkin-LCG}.
\end{proof}

\begin{rem} It is not restrictive to take $w=w(t)$ in Theorem \ref{THM:Lizorkin-LCG} such that $w\in C^1$ is increasing and $w(0)=0$. In this case integrating by parts in \eqref{EQ:sufficient} we get
\begin{equation}
\label{EQ:sufficient2}
\|Af_t\|_{L^q(G)}
\lesssim
\left(
\int\limits^t_0 w(s)^{\frac1p-\frac1q-1} w'(s)\mu_s(A)\,ds\right)
\|f_t\|_{L^p(G)}.
\end{equation}
Passing to the limit, we get a sufficient condition for the $L^p$-$L^q$ boundedness in terms of the derivative of $w$ instead of the derivative of the $\mu_s(A)$, namely
\begin{equation}\label{EQ:sufficient3}
\|A\|_{L^p_{w}(G)\to L^q_{w}(G)}
\lesssim
\int\limits^\infty_0 w(s)^{\frac1p-\frac1q-1} w'(s)\mu_s(A)\,ds.
\end{equation}
\end{rem}

We also have the following corollary of Theorem \ref{THM:Lizorkin-LCG} and its proof, estimating the $L^p$-$L^q$ norm by the Lorentz norm of the operator in the group von Neumann algebra, see Definition \ref{DEF:Lorenz-spaces}.

\begin{cor} 
\label{COR:Lizorkin}
Let $G$ be a locally compact unimodular separable group and let $A$ be a left Fourier multiplier on $G$.
Let $1<p\leq \min(2,q)$  and take $w(t)=t$. Then we have
\begin{equation}
\|A\|_{L^p_{w}(G)\to L^q_{w}(G)}
\lesssim
\|A\|_{L^{r,1}(\VN_R(G))},
\end{equation}
where $\frac1r=\frac1p-\frac1q$.
\end{cor}
\begin{proof}[Proof of Corollary \ref{COR:Lizorkin}]
Let us recall inequality \eqref{EQ:sufficient} from the proof of Theorem \ref{THM:Lizorkin-LCG}:
\begin{equation}
\label{EQ:final-recalled}
\frac{\|Af_t\|_{L^q(G)}}{\|f_t\|_{L^p(G)}}
\lesssim
t^{\frac1p-\frac1q}\mu_t(A)
-
\int\limits^{t}_0 s^{\frac1p-\frac1q}\left[\frac{d}{ds}\mu_s(A)\right]\,ds.
\end{equation}
Integrating by parts in the right-hand side of \eqref{EQ:final-recalled}, we get
\begin{multline*}
\frac{\|Af_t\|_{L^q(G)}}{\|f_t\|_{L^p(G)}}
\leq
t^{\frac1p-\frac1q}\mu_t(A)
-
t^{\frac1p-\frac1q}\mu_t(A)
+
\left(\frac1p-\frac1q\right)
\int\limits^{t}_0 s^{\frac1p-\frac1q}\mu_s(A)\frac{ds}{s}
\\  =
\left(\frac1p-\frac1q\right)\|A\|_{L^{r,1}(\VN_R(G))},
\end{multline*}
with $\frac1r=\frac1p-\frac1q$,
in view of Definition \ref{DEF:Lorenz-spaces}.
\end{proof}

\subsection{Compact Lie groups}

Since on compact Lie groups the symbolic calculus is available it is also natural to search for conditions expressed in terms of the matrix-valued symbol of the invariant operator under consideration. 

In order to measure the regularity of the symbol, we introduce a new family of difference operators $\widehat{\partial}$ acting on Fourier coefficients and on symbols. These operators are used to formulate and prove a version of the Lizorkin theorem on compact groups.

Let $G$ be a compact Lie group of dimension $n$. Let $\lambda_1,\lambda_2,\ldots,\lambda_N,\ldots$ denote the eigenvalues of the $n$-th order elliptic pseudo-differential operator $(I-\mathcal{L}_{G})^{\frac{n}2}$ counted with multiplicities. We shall enumerate elements $\pi$ of the unitary dual $\Gh$ of $G$ via the eigenvalues $\{\lambda_k\}^{\infty}_{k=1}$, i.e.
\begin{equation}
\label{EQ:enumeration}
(I-\mathcal{L}_{G})^{\frac{n}2}\pi^{k}_{mn}=\lambda_k\pi^{k}_{mn}.
\end{equation}
Let $\sigma=\{\sigma(\pi)\}_{\pi\in\Gh}$ be a field of operators and 
define
\begin{multline}
\label{EQ:difference}
\widehat{\partial}_{\pi}\sigma(\pi^j)
:= \\
U_{\pi}
\begin{psmallmatrix}
\mu_1(\sigma(\pi^j))-\mu_1(\sigma_A(\pi^{j+1})) & 0 & \ldots & 0
\\
0 & \mu_2(\sigma(\pi^{j}))-\mu_2(\sigma_A(\pi^{j+1})) &  \ldots & 0
\\
\vdots & \vdots & \vdots & \vdots
\\
0 & 0 &  \mu_k(\sigma(\pi^{j}))-\mu_{k}(\sigma(\pi^{j+1}))
\vdots & \vdots 
\\
0 & 0 & 0 &\mu_{d_{\pi^j}}(\sigma(\pi^{j}))
\end{psmallmatrix},
\end{multline}
where $U_{\pi}$ is a partial isometry matrix in the polar decomposition $\sigma(\pi)=U_{\pi}\left|\sigma(\pi)\right|$.
Here $\mu_k(\sigma(\pi))$ are the singular numbers of $\sigma(\pi)$ written in the descending order. 

Now, using the direct sum decomposition $\Op(\sigma)=\bigoplus\limits_{\pi\in\Gh}\dpi\sigma(\pi)$, we can also lift the difference operators $\partial_{\pi}$ to $\Op(\sigma)$ by
\begin{eqnarray}
{\bf \widehat{\mathbb{\partial}}} \Op(\sigma)=\bigoplus\limits_{\pi\in\Gh}
\dpi\, {\bf \widehat{\partial}} \sigma_A(\pi).
\end{eqnarray}

Let $G$ be a compact semisimple Lie group and let $\mathbb{T}$ be its maximal torus.
We recall that a function $f$ is central if
$$
f(xg)=f(gx),\quad x\in \mathbb{T},\; g\in G.
$$
We write $L^p_{\inv}(G)$ for the space of central functions $f\in L^p(G)$.
The number $s$ in the theorem below is  defined by the root datum of $G$, see Appendix \ref{SEC:mean_convergence} for the precise definition. The appearance of the bounds for $p$ and $q$ involving $s$ is caused by the density properties of the space of polyhedral trigonometric polynomials in $L^p$ for such range of indices and the failure of such density otherwise, see Stanton \cite{Stanton1976}, which is explained in more detail in Appendix \ref{SEC:mean_convergence}.

\begin{thm} 
\label{THM:Lizorkin-CG}
Let  $2-\frac1{1+s}<p\leq q<2+\frac1s$ and let $A$ be a left Fourier multiplier on a compact semisimple Lie group $G$ of dimension $n$. Let $1-\frac1p\leq m<1$. Then we have
\begin{equation}
\label{EQ:Lizorkin-CG}
\|A\|_{L^p_{\inv}(G)\to L^q_{\inv}(G)}
\lesssim
\sup_{\pi\in\Gh}
\jp\pi^{n\left(\frac1p-\frac1q\right)}\|\sigma_A(\pi)\|_{\op}
+
\sup_{\pi\in\Gh}
\jp\pi^{n\left(\frac1p-\frac1q+m\right)}
\|\widehat{\partial}\sigma_A(\pi)\|_{\op}.
\end{equation}
We also have
\begin{equation}
\label{EQ:Lizorkin-CG-3}
\|A\|_{L^p_{\inv}(G)\to L^q_{\inv}(G)}
\lesssim
\sup_{\pi\in\Gh}
\jp\pi^{n\left(\frac1p-\frac1q\right)}\|\sigma_A(\pi)\|_{\op}
+
\sum\limits_{\pi\in\Gh}
\jp\pi^{n\left(\frac1p-\frac1q\right)}
\|\widehat{\partial}\sigma_A(\pi)\|_{\op}.
\end{equation}

\end{thm}
Here $2-\frac1{1+s}<p<2+\frac1s$ is determined by the sharp conditions on mean summability (\cite{Colzani1989}), see Appendix \ref{SEC:mean_convergence}. 



\begin{proof}[Proof of Theorem \ref{THM:Lizorkin-CG}]
We will sometimes denote $\frac1r=\frac1p-\frac1q$.
It is sufficient to establish inequality \eqref{EQ:Lizorkin-CG} for polyhedral partial sums $f_N$, i.e.
for all  $N\in\NN$,
\begin{multline}
\label{EQ:Lizorkin-CG-polyhedral}
\|Af_N\|_{L^q_{\inv}(G)}
\\ \lesssim 
\left(
\sup_{\pi\in\Gh}
\jp\pi^{n\left(\frac1p-\frac1q\right)}\|\sigma_A(\pi)\|_{\op}
+
\sum\limits_{\pi\in\Gh}
\jp\pi^{n\left(\frac1p-\frac1q\right)}
\|\widehat{\partial}\sigma_A(\pi)\|_{\op}
\right)\|f_N\|_{L^p_{\inv}(G)}.
\end{multline}

Let us decompose functions $f\in L^p(G)$ and $g\in L^{p'}(G)$ as
$$
f=f_1-f_2,\quad g=g_1-g_2,
$$
into functions $f_i,g_i$ with positive Fourier transform, i.e.
$$
R_{f_i}\geq 0,\quad R_{g_j}\geq 0,\quad i,j=1,2,
$$
by taking
\begin{eqnarray*}
f_1=\sum\limits_{\pi\in Q_1}
\dpi\Tr (\widehat{f}(\pi)\pi),
\\
f_2=\sum\limits_{\pi\in Q_2}
\dpi\Tr (\widehat{f}(\pi)\pi),
\end{eqnarray*}
where 
\begin{eqnarray*}
Q_1
=
\{
\pi\in\Gh\colon \widehat{f}(\pi)\geq 0
\}
,
\\
Q_2
=
\{
\pi\in\Gh\colon \widehat{f}(\pi)< 0
\}.
\end{eqnarray*}
By the linearity of the Fourier transform $\mathcal{F}\colon f\mapsto R_f$, we have
\begin{equation*}
\left|(Af,g)_{L^2(G)}\right|
=
\left|
\sum\limits^{2}_{1}(-1)^{i+j}(Af_i,g_j)_{L^2(G)}
\right|
\leq
\sum\limits^{2}_{i,j=1}\left|(Af_i,g_j)_{L^2(G)}\right|.
\end{equation*}
Hence, without loss of generality, we may assume that $R_f\geq 0$ and $R_g\geq0$.
Then by the Plancherel identity, we have
\begin{equation*}
\left|
(Af,g)_{L^2(G)}
\right|
=
\left|
\tau(R_{Af}R^*_g)
\right|.
\end{equation*}
By using this and the hypothesis that $A$ is a left Fourier multiplier, i.e. 
$
R_{Af}=AR_f
$
and inequality $|\tau(\cdot)|\leq \tau(|\cdot|)$ (see e.g. \cite[Proposition A.1.2. p. 216]{HayesPhD2014}), we get
$$
\left|
(Af,g)_{L^2(G)}
\right|
\leq
\left|
\tau(R_{Af}R^*_g)
\right|
\leq
\tau(\left|AR_{f}R^*_g\right|).
$$
Now, we decompose the group von Neumann algebra  $\VN_R(G)$ with respect to its center $\mathcal{C}=\VN_R(G)\cap \VN_R(G)^{!}$. This yields the trace decomposition $\tau=\oplus_{\Gh}\int\tau_{\pi}$ and we get
\begin{equation*}
\label{EQ:step-1-Lizorkin}
\begin{aligned}
\tau(\left|AR_{f} R^*_g\right|)
=
\sum\limits_{\pi\in\Gh} \dpi\tau_{\pi}\left[\sigma_A(\pi) \widehat{f}(\pi)\widehat{g}(\pi)^*\right]
=
\sum\limits_{\pi\in\Gh}\sum\limits^{\dpi}_{t=1}\dpi\mu_t\left[\sigma_A(\pi) \widehat{f}(\pi)\widehat{g}(\pi)^*\right]
\\
\leq
\sum\limits_{\pi\in\Gh}\sum\limits^{[\frac{\dpi}{2}]}_{t=1} \dpi\mu_{2t}\left[\sigma_A(\pi) \widehat{f}(\pi)\widehat{g}(\pi)^*\right]
+
\sum\limits_{\pi\in\Gh}\sum\limits^{[\frac{\dpi+1}{2}]}_{t=1}\dpi\mu_{2t-1}\left[\sigma_A(\pi) \widehat{f}(\pi)\widehat{g}(\pi)^*\right]
\\ \leq
\sum\limits_{\pi\in\Gh}\sum\limits^{[\frac{\dpi}{2}]}_{t=1}\dpi\mu_{2t}\left[\sigma_A(\pi) \widehat{f}(\pi)\widehat{g}(\pi)^*\right]
+
\sum\limits_{\pi\in\Gh}\sum\limits^{[\frac{\dpi+1}{2}]}_{t=1}\dpi\mu_{2(t-1)}\left[\sigma_A(\pi) \widehat{f}(\pi)\widehat{g}(\pi)^*\right]
\\
\leq
\sum\limits_{\pi\in\Gh}\sum\limits^{[\frac{\dpi}{2}]}_{t=1}\dpi\mu_{t}\left[\sigma_A(\pi)]\mu_{t}[ \widehat{f}(\pi)\widehat{g}(\pi)^*\right]
+
\sum\limits_{\pi\in\Gh}\sum\limits^{[\frac{\dpi+1}{2}]}_{u=1}\dpi\mu_{u-1}\left[\sigma_A(\pi)]
\mu_{u-1}[\widehat{f}(\pi)\widehat{g}(\pi)^*\right]
\\\leq
\sum\limits_{\pi\in\Gh}\sum\limits^{\dpi}_{t=1}\dpi\mu_t\left[\sigma_A(\pi)\right]\mu_t\left[\widehat{f}(\pi)\widehat{g}(\pi)^*\right],
\end{aligned}
\end{equation*}
where in the inequalities we made the substitution $t\to 2t$ and used the sub-multi\-pli\-ca\-tivity $\mu_{2t}(\cdot)\leq \mu_t(\cdot)\mu_t(\cdot)$ of the singular values $\mu_t$. 
We define all the singular numbers $\mu_t$ to be zero for $t>\dpi,\pi\in\Gh$, i.e.
\begin{equation*}
\mu_t(\widehat{f}(\pi))=0,\quad
\mu_t(\widehat{g}(\pi))=0,\quad
\mu_t(\sigma_A(\pi))=0,\quad \text{for}\quad t>\dpi,\;\pi\in\Gh.
\end{equation*}
We have thus only to show that
\begin{equation}
\label{EQ:stop-look}
\left|
(Af,g)_{L^2(G)}
\right|
\lesssim
\sum\limits_{\pi\in\Gh}\sum\limits^{\infty}_{t=1}\dpi\mu_t\left[\sigma_A(\pi)\right]\mu_t\left[\widehat{f}(\pi)\widehat{g}(\pi)^*\right],\; f\in L^p(G),\,g\in L^{q'}(G).
\end{equation}

Further, we take $f=f_N=\sum\limits_{\pi\in Q_N}\dpi\Tr(\widehat{f}(\pi)\pi)$ in \eqref{EQ:stop-look} and index the $N$-th order polyhedron $Q_N\subset \Gh$ by the eigenvalues of $(I-\mathcal{L}_{G})^{\frac{\dim(G)}2}$, i.e.
$Q_N=\{\pi^{k}\}^{b_N}_{k=a_N}\subset\Gh$ and $\{a_1,\ldots, a_N\}$ correspond to the indices $k$ of such $\lambda_k$'s that $\pi^k\in Q_N$, i.e.

$$
(I-\mathcal{L}_{G})^{\frac{\dim(G)}2}\pi^{k}_{mn}=\lambda_k\pi^{k}_{mn},\quad k=a_1,\ldots, a_N,\quad \pi^k\in Q_N.
$$
Recall that the $N$-th order polyhedron $Q_N$ is defined via the highest weight theory
$$
Q_N
=
\{
\pi\in\Gh\colon
\pi_{i}
\leq
\rho_iN,\quad i=\overline{1,\ldots, l}
\},
$$
where $l$ is the rank of the group $G$ and $(\pi_1,\pi_2,\ldots,\pi_l)$ are the highest weights of $\pi$ and $\rho$ is the half-sum of the positive roots of $G$.
%
Hence, with every subset $\{\lambda_k\}^{a_m}_{k=a_1}\{a_1,\ldots, a_m\},\quad m\leq N$ we associate a polyhedron $Q_{N_m}$ 
$$
Q_{N_{m}}=\{\pi^k_i\leq \rho_i N_m\}^{a_m}_{k=a_1},
$$
where $N_{m}$ is the minimum over all $N'$ such that $\pi^k_i\leq \rho_i N,\quad k=\overline{a_1,\ldots,a_m}$.
We shall agree that the sum over $Q_N$ runs via the eigenvalues $\lambda_k$, $k=a_1,\ldots,a_N$, with multiplicities, i.e.
\begin{equation}
\sum\limits_{\pi\in Q_N} 1=\sum\limits^{a_N}_{k=a_1}1.
\end{equation}
Changing the order of summation in \eqref{EQ:stop-look} and using the convention above, we get
\begin{equation}
\label{EQ:stop-look-1}
\left|
(Af_N,g)_{L^2(G)}
\right|
\leq
\sum\limits^{\infty}_{t=1}
\sum\limits^{k=a_N}_{k=a_1} d_{\pi^k}
\mu_t\left[\sigma_A(\pi^k)\right]\mu_t\left[\widehat{f}(\pi^k)\widehat{g}(\pi^k)^*\right].
\end{equation}
Now, we shall write
\begin{equation}
\label{EQ:proof-notation}
\alpha_{t,k}=\mu_t[\sigma_A(\pi^k)],\quad \beta_{t,k}=\mu_t[\widehat{f}(\pi^k)\widehat{g}(\pi^k)^*].
\end{equation}
Let us apply the Abel transform with respect to $k$ in the right hand side of \eqref{EQ:stop-look-1}:
\begin{equation*}
\begin{aligned}
\sum\limits^{a_N}_{\substack{k=a_1}}
\alpha_{t,k}d_{\pi^{k}}\beta_{t,k}
=
\alpha_{t,a_N}
\sum\limits^{a_N}_{w=a_1}d_{\pi^{w}}\beta_{t,w}+
\sum\limits^{a_N}_{k=a_1}
\left(
\Delta_{k}\alpha_{t,k}
\right)
\sum\limits^{a_N}_{w=a_1}
d_{\pi^{w}}
\beta_{t,w},
\end{aligned}
\end{equation*}
where
\begin{equation*}
\Delta_{k}\alpha_{t,k}=\alpha_{t,k}-\alpha_{t+1,k}.
\end{equation*}
Combining this with \eqref{EQ:stop-look}, we get
\begin{equation*}
\begin{aligned}
\left|
(Af,g)_{L^2(G)}
\right|
&
&
\\
\leq
\sum\limits^{\infty}_{t=1}
\alpha_{t\,a_N}
\sum\limits^{a_N}_{w=a_1}d_{\pi^w}
\beta_{t,w}
&
+
\sum\limits^{\infty}_{t=1}
\sum\limits^{a_{n-1}}_{k=a_1}
\Delta_{k}
\alpha_{t,k}
\sum\limits^k_{w=a_1}
d_{\pi^w}
\beta_{t,w}
&
\\
&
\leq
\alpha_{1\,a_N}
\sum\limits^{\infty}_{t=1}
\sum\limits^{a_N}_{w=a_1}d_{\pi^w}
\beta_{t,w}
+
\sum\limits^{a_{N-1}}_{k=a_1}
\sup_{t\in\NN }
\Delta_{k}
\alpha_{t,k}
\sum\limits^{\infty}_{t=1}
\sum\limits^k_{w=a_1}
d_{\pi^w}
\beta_{t,w}.
&
\end{aligned}
\end{equation*}
Interchanging the order of summation and applying the Plancherel formula, we get
\begin{equation*}
\begin{aligned}
\sum\limits^{\infty}_{t=1}
\sum\limits^k_{w=a_1}d_{\pi^{w}}\,\mu_t
\left[
\widehat{f}(\pi^{w})\widehat{g}(\pi^{w})^*
\right]
=
\sum\limits^k_{w=a_1}
\sum\limits^{\infty}_{t=1}
d_{\pi^{w}}\,\mu_t
\left[
\widehat{f}(\pi^{w})\widehat{g}(\pi^{w})^*
\right]
\\ =
\sum\limits^k_{w=a_1}d_{\pi^{w}}
\sum\limits^{d_{\pi^w}}_{t=1}
\mu_t\left[\widehat{f}(\pi^w)\widehat{g}(\pi^w)^*\right]
\\=
\sum\limits^k_{w=a_1}d_{\pi^{w}}
\Tr[\widehat{f}(\pi^w)\widehat{g}(\pi^w)^*]
=
(f_{(a_1,k)},g)_{L^2(G)}
\leq
\|f_{(a_1,k)}\|_{L^q(G)}
\|g\|_{L^{q'}(G)},
\end{aligned}
\end{equation*}
where we write
\begin{equation}
f_{k}
=
\sum\limits^{k}_{w=a_1}d_{\pi^{w}}
\Tr(\widehat{f}(\pi^w)\pi^w).
\end{equation}
Collecting these estimates, we obtain
\begin{multline*}
\left|
(Af_N,g)_{L^2(G)}
\right|
\\ \leq
\left(
\alpha_{1\,a_N}
\|f_{N}\|_{L^q(G)}
+
\sum\limits^{a_{N-1}}_{k=a_1}
\sup_{t\in\NN}
\Delta_{k}
\alpha_{t,k}
\|f_k\|_{L^q(G)}
\right)
\|g\|_{L^{q'}(G)}.
\end{multline*}
By the duality of $L^p$-spaces we immediately get
\begin{equation}
\label{EQ:pre-final}
\|Af_{N}\|_{L^q(G)}
\lesssim
\alpha_{1\,a_N}
\|f_{N}\|_{L^q(G)}
+
\sum\limits^{a_{N-1}}_{k=a_1}
\sup_{t\in\NN} \Delta_{\pi}
\alpha_{t,k}
\|f_k\|_{L^q(G)}
\end{equation}
We now claim that the following version of the Nikolskii-Bernshtein inequality holds true:
\begin{equation}
\label{EQ:Nikolsky-as-we-need}
\|f_{N}\|_{L^q(G)}
\lesssim
\lambda^{\frac1p-\frac1q}_{a_N}
\|f_{N}\|_{L^p(G)},\quad 1<p\leq q\leq \infty,\,a,b\in \mathbb{N}.
\end{equation}
Moreover, the composition of the Banach-Steinhaus theorem and inequality \eqref{EQ:Nikolsky-as-we-need} yields 
\begin{equation}
\label{EQ:uniform-Nikolsky}
\|f_{N}\|_{L^q(G)}
\lesssim
\lambda^{\frac1p-\frac1q}_{a_N}
\|f\|_{L^p(G)},\quad 1<p\leq q\leq \infty,\,N\in \mathbb{N}.
\end{equation}
Assuming this to be true for a moment, and using inequality \eqref{EQ:Nikolsky-as-we-need}, we get
\begin{equation}
\label{EQ:pre-final-2}
\|Af_{N}\|_{L^q(G)}
\lesssim
\alpha_{1\,a_N}
\lambda^{\frac1p-\frac1q}_{a_N}
\|f\|_{L^p(G)}
+
\sum\limits^{a_{N-1}}_{k=a_1}
\sup_{t\in\NN} \Delta_{\pi}
\alpha_{t,k}
\|f_k\|_{L^q(G)}
\end{equation}
At this point, we note that by Nikolsky inequality \eqref{EQ:Nikolsky-as-we-need} and the Banach-Steinhaus theorem, we immediately get from \eqref{EQ:pre-final-2}
\begin{equation*}
\|Af_{N}\|_{L^q(G)}
\lesssim
\left(
\alpha_{1\,a_N}
\lambda^{\frac1p-\frac1q}_{a_N}
\|f\|_{L^p(G)}
+
\sum\limits^{a_{N-1}}_{k=a_1}
\sup_{t\in\NN} \Delta_{\pi}
\alpha_{t,k}
\right)
\|f\|_{L^q(G)}.
\end{equation*}
We show how to pass to the limit $N\to\infty$ in $\|Af_N\|_{L^q(G)}$.
The Fourier series
$$
\sum\limits_{\pi\in Q_N}\dpi\Tr\sigma_A(\pi)\widehat{f}(\pi)\pi(x)
$$
is absolutely convergent since $\|\sigma_A(\pi)\|_{\op}\leq \frac1{\jp\pi^{n\left(\frac1p-\frac1q\right)}}$ and $f\in C^{\infty}_0(G)$. Indeed, we have
\begin{align}
\begin{split}
\sum\limits_{\pi\in Q_N}\dpi
\left|\Tr\sigma_A(\pi)\widehat{f}(\pi)\pi(x)\right|
\leq
\sum\limits_{\pi\in Q_N}\dpi
\|\sigma_A(\pi)\|_{\op}
\|\widehat{f}(\pi)\|_{\HS}
\|\pi\|_{\HS}
\\
\sum\limits_{\pi\in Q_N}d^{2}_{\pi}
\frac{d^{1/2}_{\pi}}{\jp\pi^{\frac{n}{r}}}
\|\widehat{f}(\pi)\|_{\HS}
\end{split}
\end{align}
Hence, for every $x\in G$
$$
\lim_{N\to\infty}
\left|
\sum\limits_{\pi\in Q_N}\dpi
\Tr\sigma_A(\pi)\widehat{f}(\pi)\pi(x)
\right|
=
\left|
\sum\limits_{\pi\in\in\Gh}\dpi
\Tr\sigma_A(\pi)\widehat{f}(\pi)\pi(x)
\right|
$$
and by the Fatou theorem
\begin{align}
\begin{split}
\left\|
\sum\limits_{\pi\in\Gh}\dpi\Tr\sigma_A(\pi)\widehat{f}(\pi)\pi(x)
\right\|_{L^q(G)}
\\\leq
\lim_{N\to\infty}
\left\|
\sum\limits_{\pi\in Q_N}\dpi\Tr\sigma_A(\pi)\widehat{f}(\pi)\pi(x)
\right\|_{L^q(G)}
\\\leq
\left(
\sup_{\pi\in\Gh}
\jp\pi^{n\left(\frac1p-\frac1q\right)}
\|\sigma_A(\pi)\|_{\op}
+
\sum\limits_{\pi\in\Gh}
\jp\pi^{nm}
\|\widehat{\partial}\sigma_{A}(\pi)\|_{\op}
\right)\|f\|_{L^p(G)}.
\end{split}
\end{align}
Now, we concentrate on establishing \eqref{EQ:Lizorkin-CG}.
Modulus technical details, we shall interpolate between two Nikolsky inequalities in order to estimate the second sum in \eqref{EQ:pre-final-2}.
Let
$p_0<p<p_1$ 
and 
$\frac1p=\frac{1-\theta}{p_0}+\frac{\theta}{p_1},\,0<\theta<1$.
We take 
\begin{equation}
\label{EQ:gamma}
\gamma=(1-\theta)\gamma_1=(1-\theta)\left(\frac1{p_0}-\frac1{p_1}\right).
\end{equation}
We divide and multiply the sum by $\lambda^{\gamma}_k$
\begin{align}
\begin{split}
\label{EQ:discre-sum-2}
\sum\limits^{a_{N-1}}_{k=a_1}
\sup_{t\in\NN} \Delta_{\pi}
\alpha_{t,k}
\|f_k\|_{L^q(G)}
=
\sum\limits^{a_{N-1}}_{k=a_1}
\lambda^{\frac1p-\frac1q+1-\gamma}_k
\sup_{t\in\NN} \Delta_{\pi}
\alpha_{t,k}
\lambda^{\gamma}_k
\frac{
\|f_{k}\|_{L^q(G)}
}
{
\lambda^{\frac1p-\frac1q}
}
\frac1{\lambda_k}
\\\leq
\left(
\sup_{k\in\NN}
\lambda^{\frac1p-\frac1q+1-\gamma}_k
\sup_{t\in\NN} \Delta_{\pi}\alpha_{t,k}
\right)
\sum\limits^{\infty}_{k=1}
\lambda^{\gamma}_k
\frac{
\|f_{k}\|_{L^q(G)}
}
{
\lambda^{\frac1p-\frac1q}_k
}
\frac1{\lambda_k}.
\end{split}
\end{align}
Let us denote by $\overline{f}(t)$ the quantity given by
\begin{equation}
\label{EQ:mean}
\overline{f}(t)
=
\sup_{\lambda_n\geq t}\frac{\|f_n\|_{L^q(G)}}{\lambda^{\delta}_n},
\end{equation}
where
\begin{equation}
\label{EQ:delta}
\delta=\frac1{p_0}-\frac1q.
\end{equation}

Using the fact that the function $t\to \overline{f}(t)$ is non-increasing function of $t>0$, we get
\begin{align}
\begin{split}
\label{EQ:long-line-cont-discrete}
\sum\limits^{\infty}_{k=1}
\lambda^{\gamma}_k
\frac{
\|f_{k}\|_{L^q(G)}
}
{
\lambda^{\frac1p-\frac1q}_k
}
\frac1{\lambda_k}
\leq
\sum\limits^{\infty}_{k=1}
\lambda^{\gamma}_k
\overline{f}(\lambda_k)
\frac1{\lambda_k}
=
\sum\limits_{s\in\ZZ}
\sum\limits_{
\substack{k\in\NN \colon 2^s\leq \lambda_k\leq 2^{s+1}
}}\lambda^{\gamma}_k
\overline{f}(\lambda_k)
\frac1{\lambda_k}
\\
\leq
\sum\limits_{s\in\ZZ}
2^{\gamma\left(s+1\right)}
\overline{f}(2^s)
\sum\limits_{\substack{k\in\NN \colon 2^s\leq \lambda_k\leq 2^{s+1}}}
\frac1{\lambda_k}
=
2^{2\gamma}
\sum\limits_{s\in\ZZ}
2^{\gamma\left(s-1\right)}
\overline{f}(2^s)
\\
\leq
2^{2\gamma}
\sum\limits_{s\in\ZZ}
\int\limits^{2^s}_{2^{s-1}}
t^{\gamma}
\overline{f}(t)
\int\limits^{2^s}_{2^{s-1}}\frac{dt}{t}
=
2^{2\gamma}
\int\limits^{\infty}_{0}
t^{\gamma}
\overline{f}(t)\frac{dt}{t}.
=
\\
\int\limits^{\infty}_{0}
t^{\gamma}
\overline{f}(t)\frac{dt}{t}
=
\int\limits^{\infty}_{0}
t^{-\theta\gamma_1}
t^{\gamma_1}
\overline{f}(t)\frac{dt}{t}
\leq
\int\limits^{\infty}_{0}
t^{-\theta\gamma_1}
\left(
\sup_{u\leq t}
u^{\gamma_1}
\overline{f}(u)
\right)
\frac{dt}{t}
\\
=\{v=t^{\gamma_1}\}
\\
\int\limits^{\infty}_{0}
v^{-\theta}
\left(
\sup_{u\leq v^{\frac1{\gamma_1}}}
u^{\gamma_1}
\overline{f}(u)
\right)
\frac{dt}{t}.
\end{split}
\end{align}
Now, we shall interpolate between two Nikolsky inequalities.
\begin{eqnarray}
\label{EQ:two-Nik}
\|f_k\|_{L^{q}(G)}
\leq
\lambda^{\frac1{p_0}-\frac1{q}}_k\|f\|_{L^{p_0}(G)},
\\
\|f_k\|_{L^{q}(G)}
\leq
\lambda^{\frac1{p_1}-\frac1{q}}_k\|f\|_{L^{p_1}(G)}.
\end{eqnarray}
Rescalling in the second inequality in \eqref{EQ:two-Nik}, we get
\begin{align}
\label{EQ:two-Nik-0}
\|f_k\|_{L^{q}(G)}
&\leq
\lambda^{\frac1{p_0}-\frac1{q}}_k
\|f\|_{L^{p_0}(G)},
\\
\label{EQ:two-Nik-1}
\|f_k\|_{L^{q}(G)}
&\leq
\lambda^{\frac1{p_1}-\frac1{p_0}}_k
\lambda^{\frac1{p_0}-\frac1{q}}_k
\|f\|_{L^{p_1}(G)}.
\end{align}
Thus, using \eqref{EQ:mean}, we rewrite inequalities \eqref{EQ:two-Nik-0},\eqref{EQ:two-Nik-1}
\begin{eqnarray}
\label{EQ:mean-f-0}
\overline{f}(\lambda_k)
\leq
\|f\|_{L^{p_0}(G)},
\\
\label{EQ:mean-f-1}
\lambda^{\gamma_1}_k
\overline{f}(\lambda_k)
\leq
\|f\|_{L^{p_1}(G)}.
\end{eqnarray}
Let $f=f^0+f^1$ be an arbitrary decomposition. 
From \eqref{EQ:mean-f-0} and \eqref{EQ:mean-f-1} we obtain
\begin{align}
\begin{split}
\sup_{u\leq v^{\frac1{\gamma_1}}}
u^{\gamma_1}
\overline{f}(u)
\leq
\sup_{u\leq v^{\frac1{\gamma_1}}}
u^{\gamma_1}
\overline{f^0}(u)
+
\sup_{u\leq v^{\frac1{\gamma_1}}}
u^{\gamma_1}
\overline{f^1}(u)
\\\leq
\sup_{u\leq v^{\frac1{\gamma_1}}}
u^{\gamma_1}
\overline{f^1}(u)
+
v
\sup_{u>0}
\overline{f^0}(u)
\leq
\|f^1\|_{L^{p_1}(G)}
+
v
\|f^0\|_{L^{p_0}(G)}.
\end{split}
\end{align}
Since the decomposition $f=f^0+f^1$ is arbitrary, we take the infimum and get
\begin{equation}
\label{EQ:mean-K}
\sup_{u\leq v^{\frac1{\gamma_1}}}
u^{\gamma_1}
\overline{f}(u)
\leq
K(t,f; L^{p_1}(G),L^{p_0}(G)),
\end{equation}
where 
the functional $K(t,f)$ is given by
\begin{equation}
K(v,f; L^{p_1}(G),L^{p_0}(G)),
=
\inf_{f=f^0+f^1}
\left\{
\|f^1\|_{L^{p_1}(G)}+v\|f^0\|_{L^{p_0}(G)}
\right\}.
\end{equation}
Composing \eqref{EQ:long-line-cont-discrete} and \eqref{EQ:mean-K}, we obtain
\begin{equation}
\label{EQ:interpolated}
\sum\limits^{\infty}_{k=1}
\lambda^{\gamma}_k
\frac{
\|f_{k}\|_{L^q(G)}
}
{
\lambda^{\frac1p-\frac1q}_k
}
\frac1{\lambda_k}
\leq
\|f\|_{L^p(G)},
\end{equation}
where in the last equality we used that $L^{p\,q}(G)$ are the interpolation spaces and the embedding of the Lorentz spaces.
Composing \eqref{EQ:discre-sum-2} and \eqref{EQ:interpolated}, we obtain
\begin{align*}
\begin{split}
\label{EQ:interpolation}
\sum\limits^{a_{N-1}}_{k=a_1}
\sup_{t\in\NN} \Delta_{\pi}
\alpha_{t,k}
\|f_k\|_{L^q(G)}
\leq
\left(
\sup_{k\in\NN}
\lambda^{\frac1p-\frac1q+1-\gamma}_k
\sup_{t\in\NN} \Delta_{\pi}\alpha_{t,k}
\right)
\|f\|_{L^p(G)}
.
\end{split}
\end{align*}
Using this and recalling \eqref{EQ:pre-final-2}, we get
\begin{equation}
\label{EQ:pre-limit-estimate}
\|Af_{N}\|_{L^q(G)}
\leq
\left(
\sup_{k\in\NN}
\alpha_{1\,k}
\lambda^{\frac1p-\frac1q}_{k}
+
\sum\limits^{a_{N-1}}_{k=a_1}
\lambda^{\frac1p-\frac1q+1-\gamma}_k
\sup_{t\in\NN} \Delta_{\pi}
\alpha_{t,k}
\right)
\|f\|_{L^p(G)}.
\end{equation}

Let us denote
\begin{equation}
m=1-\gamma.
\end{equation}
Recalling \eqref{EQ:gamma}, we obtain the range
\begin{equation}
1-\frac1p\leq m<1.
\end{equation}

Recalling notation \eqref{EQ:proof-notation} we get
\begin{equation*}
\|Af_{N}\|_{L^q(G)}
\lesssim
\left(
\sup_{k\in\NN}
\lambda^{\frac1p-\frac1q}_{k}
\|\sigma_A(\pi^k)\|_{\op}
+
\sum\limits^{\infty}_{k=1}
\lambda^{\frac1p-\frac1q+m}_k
\|\widehat{\partial}\sigma_{A}(\pi^k)\|_{\op}
\right)
\|f\|_{L^p(G)}
,
\end{equation*}
where we used the fact that $\|\widehat{\partial}\sigma_A(\pi^k)\|_{\op}=\sup_{t=1,\ldots,d_{\pi^k}}\Delta_{\pi}\alpha_{t,k}$.
Returning to the `unitary dual notation', and using that $\lambda_k\cong \jp{\pi^k}^{n}$ with $n=\dim G$, we finally obtain
\begin{equation}
\label{EQ:pre-limit}
\|Af_{N}\|_{L^q(G)}
\lesssim
\left(
\sup_{\pi\in\Gh}
\jp\pi^{n\left(\frac1p-\frac1q\right)}
\|\sigma_A(\pi)\|_{\op}
+
\sum\limits_{\pi\in\Gh}
\jp\pi^{nm}
\|\widehat{\partial}\sigma_{A}(\pi)\|_{\op}
\right)\|f\|_{L^p(G)}.
\end{equation}
Passing to the limit as above, we obtain \eqref{EQ:Lizorkin-CG}.
Now, it remains to show that inequality \eqref{EQ:Nikolsky-as-we-need} holds true.
It has been shown in \cite{NRT2014} that for every trigonometric polynomial
\begin{equation}
f_L
=
\sum\limits_{\substack{\xi\in\Gh\\ \langle \xi \rangle \leq L}}\dpi\Tr(\widehat{f}(\xi)\xi)
\end{equation}
a version of the Nikolskii-Bernshtein inequality can be written as
\begin{equation*}
\|f_L\|_{L^q(G)}
\leq
N(\rho L)^{\frac1p-\frac1q}
\|f\|_{L^p(G)},\quad 1<p<q<\infty,
\end{equation*}
where
\begin{equation*}
N(\rho L)
:=
\sum\limits_{\substack{\xi\in\Gh\\ \langle \xi \rangle \leq \rho L}}d^2_{\xi},\quad \rho=\min(1,[p/2]),
\end{equation*}
where $[p/2]$ is the integer part of $p/2$.
The application of Weyl's asymptotic law to the counting function $N(\rho L)$ of the $n$-th order elliptic pseudo-differential operator $(I-\mathcal{L}_{G})^{\frac{1}2}$ with $L=\langle \pi\rangle$ yields
\begin{equation*}
N(\rho\langle \pi \rangle ) \cong \rho^n \langle \pi \rangle^n,\quad n=\dim(G).
\end{equation*}
Hence, we immediately obtain
\begin{equation*}
\label{EQ:Nikolsky-as-we-need-proved}
\|f_{\jp{\pi}}\|_{L^q(G)}
\lesssim
\langle\pi\rangle^{n\left(\frac1p-\frac1q\right)}
\|f_{\jp{\pi}}\|_{L^p(G)}.
\end{equation*}
This completes the proof of \eqref{EQ:Nikolsky-as-we-need}.
\end{proof}

\subsection{Spectral multipliers and non-invariant operators}

Let $A$ be a left Fourier multiplier. We now  give an illustration of Theorem \ref{THM:Lizorkin-CG} related to spectral multipliers $\varphi(A)$ for
a monotone continuous function $\varphi$ on $[0,+\infty)$.
In particular, the $L^p$-$L^q$ boundedness is reduced to a condition involving the behaviour of the singular numbers of the symbol $\sigma_A$ of $A$ compared to the eigenvalues of the Laplacian on $G$.

For convenience of the following formulation we enumerate the representations of $G$ according to the growth of the corresponding eigenvalues of the Laplacian. More precisely, for the non-decreasing eigenvalues $\{\lambda_{j}\}_{j}$ 
of $(I-\mathcal{L}_{G})^{\frac{\dim(G)}2}$ with multiplicities taken into account,
we denote by $\pi^j$ the corresponding representations such that
$$
(I-\mathcal{L}_{G})^{\frac{\dim(G)}2}\pi^{j}_{ml}=\lambda_j\pi^{j}_{ml}
$$
holds for all $1\leq m,l\leq d_{\pi^j}$. From the Weyl asymptotic formula for the eigenvalue counting function we get that
\begin{equation}\label{EQ:sp-js}
\lambda_j\cong \jp{\pi^j}^{n}\cong j,\quad \textrm{ with } n=\dim G.
\end{equation} 
Therefore, we can formulate a spectral multipliers corollary of Theorem \ref{THM:Lizorkin-CG}.

\begin{cor} 
\label{THM:varphi-Lizorkin-CG}
Let $A$ be a left Fourier multiplier on a compact semisimple Lie group $G$ of dimension $n$.  
Let  $2-\frac1{1+s}<p\leq q<2+\frac1s$. Assume that $\varphi$ is a monotone function on $[0,+\infty)$. Then $\varphi(A)$ is a left-invariant operator and, moreover, we have
\begin{multline}
\label{EQ:varphi-Lizorkin-CG}
\|\varphi(A)\|_{L^p_{\inv}(G)\to L^q_{\inv}(G)}
\\ \lesssim
\sup_{j\in\NN}
j^{\frac{1}{p}-\frac1q}
\sup_{t=1,\ldots,d_{\pi^j}}
\left|\varphi(\alpha_{t\,j})\right|
+
\sum\limits_{j\in\NN}
j^{\frac{1}{p}-\frac1q}
\sup_{t=1,\ldots,d_{\pi^j}}
\left|
\varphi(\alpha_{t\,j})-\varphi(\alpha_{t+1\,j+1})
\right|,
\end{multline}
where  $\alpha_{t\,j}$ are the singular numbers of the symbol $\sigma_A(\pi^j),\,\pi^j\in\Gh$, $t=1,\ldots,d_{\pi^j}$.
\end{cor}

It will be clear from the proof of Corollary \ref{THM:varphi-Lizorkin-CG} that the condition that $\varphi$ is monotone is not essential and is needed only for obtaining a simpler expression under the sum in \eqref{EQ:varphi-Lizorkin-CG}. We leave it to the reader to formulate the analogous statement without assuming the monotonicity of $\varphi$.

\begin{ex}
Let $G$ be a compact semisimple Lie group of dimension $n$ and let $\varphi $ be as in Corollary \ref{THM:varphi-Lizorkin-CG}. Let us also assume that  $\varphi$ is boundedly differentiable $\|\varphi'\|_{L^{\infty}}<+\infty$.
As a very rough illustration of 
Corollary \ref{THM:varphi-Lizorkin-CG}
assume that the symbol $\sigma_{A}(\pi)$ ``decays" sufficiently fast with respect to $\jp\pi^n$, i.e. that
\begin{equation}
\|\sigma_A(\pi)\|_{\op}\lesssim \frac{1}{\langle\pi\rangle^{\alpha}},\quad \textrm{ for some } \alpha>n\left(\frac1p-\frac1q+1\right).
\end{equation}
Then  $\varphi(A)$ is $L^p$-$L^q$ bounded. 
Indeed, in this case the series in \eqref{EQ:varphi-Lizorkin-CG} is convergent in view of
$$
\left|
\varphi(\alpha_{t\,j})-\varphi(\alpha_{t+1\,j+1})
\right|\lesssim \|\varphi'\|_{L^{\infty}}
(\|\sigma_A(\pi^j)\|_{\op}+\|\sigma_A(\pi^{j+1})\|_{\op})\lesssim \frac{1}{\langle\pi^{j+1}\rangle^{\alpha}}\cong \frac{1}{j^{\alpha/n}}.
$$
\end{ex}

\begin{proof}[Proof of Corollary \ref{THM:varphi-Lizorkin-CG}]
By the functional calculus for affiliated unbounded operators \cite[Proposition 4.2]{DNSZ2016}, we get
\begin{equation}
\label{EQ:varphi-A}
\varphi(A)
=
\bigoplus\limits^{\infty}_{j=1}d_{\pi^j}\sum\limits^{d_{\pi^j}}_{k=1}\varphi(\sigma_A(\pi^j)).
\end{equation}
Let us denote by $\alpha_{t,\pi^j}$ the singular values of $\sigma_A(\pi^j)$. 
From \eqref{EQ:varphi-A}, we get that the singular numbers $\beta_{t,\pi^j}$ of $\varphi(\sigma_A(\pi^j))$ are given by
$$
\{\beta_{t,\pi^j}\}^{d_{\pi^j}}_{t=1}
=
\{\varphi(\alpha_{1,\pi^j}),\varphi(\alpha_{2,\pi^j}),\ldots,\varphi(\alpha_{d_{\pi^j},\pi^j})\}
$$
if $\varphi$ is increasing, and by
$$
\{\beta_{t,\pi^j}\}^{d_{\pi^j}}_{t=1}
=
\{\varphi(\alpha_{\dpi,\pi}),\varphi(\alpha_{\dpi-1,\pi}),\ldots,\varphi(\alpha_{1,\pi})\}
$$
otherwise.
By definition \eqref{EQ:difference}, we get
\begin{equation}
\label{EQ:partial-differencce-varphi-A}
\widehat{\partial}\varphi(\sigma_A(\pi^j))
=
\diag(
\beta_{1,j}-\beta_{1,j+1}
,
\cdots
,
\beta_{k,j}-\beta_{k,j+1}
,
\cdots,
\beta_{d_{\pi^j}-1,j}-\beta_{d_{\pi^j},j+1},
\beta_{d_{\pi^j}}
).
\end{equation}
By Theorem \ref{THM:Lizorkin-CG} we have
\begin{equation}
\label{EQ:varphi-A-upper}
\|\varphi(A)\|_{L^p_{\inv}(G)\to L^q_{\inv}(G)}
\lesssim
\sup_{\pi\in\Gh}
\jp\pi^{\frac{n}{r}}\|\varphi(\sigma_A(\pi))\|_{\op}
+
\sum\limits_{\pi\in\Gh}
\jp\pi^{\frac{n}{r}}
\|\widehat{\partial}\varphi(\sigma_A(\pi))\|_{\op}.
\end{equation}
Combining \eqref{EQ:varphi-A-upper} and \eqref{EQ:partial-differencce-varphi-A} we obtain 
\begin{multline*}
\|\varphi(A)\|_{L^p_{\inv}(G)\to L^q_{\inv}(G)}
\\ \lesssim
\sup_{\pi\in\Gh}
\jp\pi^{\frac{n}{r}}
\sup_{t=1,\ldots,\dpi}
\left|\varphi(\alpha_{t\,\pi})\right|
+
\sum\limits_{\pi\in\Gh}
\jp\pi^{\frac{n}{r}}
\sup_{t=1,\ldots,\dpi}
\left|
\varphi(\alpha_{t\,\pi})-\varphi(\alpha_{t+1\,\pi})
\right|.
\end{multline*}
By using \eqref{EQ:sp-js} and the numbering of the representations as explained before Corollary \ref{THM:varphi-Lizorkin-CG}, it establishes \eqref{EQ:varphi-Lizorkin-CG} and completes the proof.
\end{proof}


Finally we note that as a corollary of Theorem \ref{THM:Lizorkin-CG} on compact Lie groups we get the boundedness result also for non-invariant operators. Indeed, a rather standard argument (see e.g. the proof of Theorem \ref{THM:Lpq-G}) immediately yields:

\begin{thm} 
\label{THM:Lizorkin-CG-noninvariant}
Let $G$ be a compact semisimple connected Lie group of dimension $n$. 
Let $2-\frac1{1+s}<p\leq q<2+\frac1s$ and suppose that $l>\frac{q}{n}$ is an integer. Let $A$ be a continuous linear operator on $C^{\infty}(G)$. Then we have
\begin{multline}
\|A\|_{L^p_w(G)\to L^q_w(G)}
 \lesssim
\sum\limits_{|\alpha|\leq l}
\sup_{u\in G}
\sup_{\pi\in\widehat{G}}
\jp\pi^{n(\frac1p-\frac1q)}
\|\partial^{\alpha}_u\sigma_A(u,\pi)\|_{\op}
\\ +\sum\limits_{|\alpha|\leq l}\sup_{u\in G}
\sum\limits_{\pi\in\widehat{G}}
\jp\pi^{n(\frac1p-\frac1q)}
\|\partial^{\alpha}_{u}\widehat{\partial}\sigma_A(u,\pi)\|_{\op}.
\end{multline}
\end{thm}

\section{H\"ormander's multiplier theorem on locally compact groups}
\label{SEC:Horm-LCG}

It is possible to refine Theorem \ref{THM:Lizorkin-LCG} for the range $1<p\leq 2\leq q<+\infty$. The statement that we prove can be viewed as a locally compact groups analogue of the H\"ormander $L^p$-$L^q$ multiplier theorem \cite[p. 106, Theorem 1.11]{Hormander:invariant-LP-Acta-1960}, however, because of the general setting of locally compact groups, the spectral rather than symbolic information is used. However, our statement in Theorem \ref{THM:upper-bound} implies both the H\"ormander theorem and the known results on compact Lie groups.

In the following statements, to unite the formulations,
we adopt the convention that the sum or the integral over an
empty set is zero, and that $0^{0}=0$.

\begin{thm}
\label{THM:upper-bound}
 Let $1<p\leq 2 \leq q<+\infty$ and suppose that $A$ is a Fourier multiplier on a locally compact separable unimodular  group $G$. Then we have
\begin{equation}
\label{EQ:LCG-FM-upper}
\|A\|_{L^p(G)\to L^q(G)}
\lesssim
\sup_{s>0}
s
\left[
\int\limits_{\substack{t\in\RR_+\colon \mu_t(A)\geq s}}dt
\right]^{\frac1p-\frac1q}.
\end{equation}
For $p=q=2$ inequality \eqref{EQ:LCG-FM-upper} is sharp, i.e.
\begin{equation}
\label{EQ:LCG-FM-upper:sharpness}
\|A\|_{L^2(G)\to L^2(G)}
=
\sup_{t\in\RR_+}\mu_t(A).
\end{equation}
Using the noncommutative Lorentz spaces $L^{r,\infty}$ 
with  $\frac1r=\frac1p-\frac1q$, $p\not=q$,
we can also write 
\eqref{EQ:LCG-FM-upper} as
\begin{equation}
\|A\|_{L^p(G)\to L^q(G)}
\lesssim
\|A\|_{L^{r,\infty}(VN_R(G))}.
\end{equation}
\end{thm}

We recall Definition \ref{DEF:Lorenz-spaces} for the noncommutative Lorentz spaces.
\begin{rem} 
\label{REM:Lebesgue-implies-Lorentz}
We notice that inequality \eqref{EQ:LCG-FM-upper} holds true for $L^{p\theta}-L^{q\theta}$-Fourier multipliers
\label{REM:extension-Lorenz}
\begin{equation}
\|A\|_{L^{p\theta}(G)\to L^{q\theta}(G)}
\leq
\|A\|_{L^{r\,\infty}(\VN_R(G))},\quad 1\leq \theta <\infty.
\end{equation}
\end{rem}
\begin{proof}[Proof of Remark \ref{REM:extension-Lorenz}]
Let us assume $p<2<q$ and fix $p_0,p_1,q_0,q_1$ such that
\begin{eqnarray}
p_0<p<p_1,\quad q_0<q<q_1,
\\
p_0<2<q_0,\quad p_1<2<q_1.
\end{eqnarray}
Applying inequality \eqref{EQ:LCG-FM-upper} for $p=p_0,\quad q=q_0$ and $p=p_1,\quad q=q_1$, we get
\begin{equation}
\|Af\|_{L^{q_i}(G)}
\leq
\sup_{s>0}s\left(\int\limits_{\substack{t\in\RR_+\\ \mu_t(A)\geq s}}dt\right)^{\frac1{p_i}-\frac1{q_i}}
\|f\|_{L^{p_i}(G)},\quad i=0,1.
\end{equation}
A standard interpolation argument yields
\begin{equation}
\|A\|_{L^{p\theta}(G)\to L^{q\theta}(G)}
\leq
\|A\|^{1-\theta}_{L^{r_0\,\infty}(\VN_R(G))}
\|A\|^{\theta}_{L^{r_1\,\infty}(\VN_R(G))}.
\end{equation}
We show that
$$
\|A\|^{1-\theta}_{L^{r_0\,\infty}(\VN_R(G))}
\|A\|^{\theta}_{L^{r_1\,\infty}(\VN_R(G))}.
\leq
\|A\|_{L^{r\,\infty}(\VN_R(G))},
$$
where $\frac1r_i=\frac1{p_i}-\frac1{q_i}$ and $\frac1r=\frac{1-\theta}{r_0}+\frac{\theta}{r_1}$.
Let us recall that
$$
\|A\|_{L^{r_0\,\infty}(\VN_R(G))}
=
\sup_{t>0}t^{\frac1{r_0}}\mu_t(A).
$$
Direct calculations yield that
$$
\left(\sup_{t>0}t^{\frac1{r_0}}\mu_t(A)\right)^{1-\theta}
\left(\sup_{t>0}t^{\frac1{r_1}}\mu_t(A)\right)^{\theta}
\leq
\sup_{t>0}t^{\frac1{r}}\mu_t(A).
$$
This completes the proof.
This completes the proof.
\end{proof}
\begin{proof}[Proof of Theorem \ref{THM:upper-bound}]
Since the algebra $S(\VN_R(G))$ of left Fourier multipliers $A$ is closed under taking the adjoint $S(\VN_R(G))\ni A \mapsto A^*\in S(\VN_R(G))$ (see \cite[Theorem 4, p. 412]{Segal1953} or \cite[Theorem 28 on p. 4]{Terp1981}),
and 
\begin{equation}
\|A\|_{L^{p}(G)\to L^q(G)}
=
\|A^*\|_{L^{{q'}}(G)\to L^{p'}(G)},
\end{equation}
we may assume that $p\leq q'$, for otherwise we have $q'\leq (p')'=p$ and use 
\eqref{EQ:muadj} ensuring that $\mu_t(A^*)=\mu_t(A)$. When $f\in L^p(G)$, dualising the Hausdorff-Young inequality \eqref{EQ:HY-LCG-2} gives, since $q'\leq 2$,
\begin{equation}
\|Af\|_{L^q(G)}
\leq
\left(
\int\limits^{+\infty}_0 [\mu_t(R_{Af})]^{q'}dt
\right)^{\frac1{q'}}.
\end{equation}
%
%
%
By the left-invariance of $A$ (e.g. \cite[Proposition 3.1 on page 31]{Terp1980}) we have
\begin{equation*}
R_{Af}=AR_{f},\quad f\in L^2(G).
\end{equation*} 

By our assumptions, $A$ and $R_{f}$ are measurable with respect to ${\VN}_R(G)$. This makes it possible to apply Lemma \ref{LEM:mu-t-A-properties} to obtain the estimate
\begin{equation}
\mu_t(R_{Af})
=
\mu_t(AR_{f})
\leq
\mu_t(A)
\mu_t(R_{f}).
\end{equation}
Thus, we obtain
\begin{equation}
\label{EQ:step-1-HY}
\|Af\|_{L^q(G)}
\leq
\left(
\int\limits^{+\infty}_0 [\mu_t(A)\mu_t(R_{f})]^{q'}dt
\right)^{\frac1{q'}}.
\end{equation}

Now, we are in a position to apply the Hausdorff-Young-Paley inequality in Theorem \ref{THM:HYP-LCG}. 
With 
$\varphi(t)=\mu_t(A)^{r}$ for $\frac1r=\frac1p-\frac1q$,
the assumptions of Theorem \ref{THM:HYP-LCG} are then satisfied, and since $\frac1{q'}-\frac1{p'}=\frac1p-\frac1q=\frac1r$, we obtain
\begin{equation}
\label{EQ:step-2-Paley}
\left(
\int\limits^{+\infty}_0
[\mu_t(R_{f})\mu_t(A)]^{q'}dt
\right)^{\frac1{q'}}
\leq
\sup_{s>0}
\left[
s
\int\limits_{\substack{t\in\RR_+\\ \mu_t(A)^r\geq s}}dt
\right]^{\frac1r}
\|f\|_{L^p(G)}.
\end{equation}
Further, it can be easily checked that
\begin{equation}
\left(
\sup_{s>0}
s
\int\limits_{\substack{t\in\RR_+ \\ \mu_t(A)^r\geq s}}dt
\right)^{\frac1r}
=
\left(
\sup_{s>0}
s^r
\int\limits_{\substack{t\in\RR_+ \\ \mu_t(A)\geq s}}dt
\right)^{\frac1r}
=
\sup_{s>0}
s
\left(
\int\limits_{\substack{t\in\RR_+ \\ \mu_t(A)\geq s}}dt
\right)^{\frac1r}.
\end{equation}
Thus, we have established inequality \eqref{EQ:LCG-FM-upper}. This completes the proof.
\end{proof}

\subsection{The case of $\RR^n$}

Here we relate the statement of Theorem \ref{THM:upper-bound} to the classical H\"ormander theorem.

\begin{rem} 
\label{REM:AR-H}
As a special case with $G=\RR^n$, Theorem \ref{THM:upper-bound} implies the H\"ormander multiplier estimate \eqref{EQ:Hormander-estimate} established in \cite[p. 106, Theorem 1.11]{Hormander:invariant-LP-Acta-1960}, and we have
\begin{equation}
\|A\|_{L^{r,\infty}(\VN_R(\RR^n))}
=
\|\sigma_A\|_{L^{r,\infty}(\RR^n)}.
\end{equation}
\end{rem}

\begin{proof}[Proof of Remark \ref{REM:AR-H}]
Indeed, we identify the algebra $\VN_R(\RR^n)$ via the Fourier transform $\FT_{\RR^n}$ with the algebra $Z=\{M_{\varphi}\}_{\varphi\in L^{\infty}(\RR^n)}$ of the multiplication operators 
$$
M_{\varphi}\colon L^2(\widehat{\RR}^n)\colon h \mapsto M_{\varphi}h=\varphi h\in L^2(\widehat{\RR}^n),
$$
see Example \ref{EX:abelian-measurability}.
Given an element $A$ of $\VN_R(\RR^n)$ which acts on $L^2(\RR^n)$ by the convolution with its convolution kernel $K_A$,
$$
A\colon L^2(\RR^n)\ni f \mapsto Af=K_A\ast f,
$$
we associate with $A$ the multiplication operator $M_{\sigma_A}$ acting on 
$L^2(\widehat{\RR}^n)$ 
via the multiplication by the symbol 
$\sigma_A=\widehat{K_A}$,
$$
M_{\sigma_A}
\colon 
L^2(\widehat{\RR}^n)
\ni \widehat{f} 
\mapsto 
M_{\sigma_A}\widehat{f}
=
\sigma_A(\xi)\widehat{f}(\xi)\in L^2(\widehat{\RR}^n).
$$
Then by Example \ref{EX:mu-t} with $\varphi(\xi)=\sigma_A(\xi)$, we get
\begin{equation}
\mu_t(M_{\sigma_A})
=
\sigma_{A}^*(t).
\end{equation}
Thus, Definition \ref{DEF:Lorenz-spaces} on the noncommutative Lorentz spaces yields
\begin{multline}
\|M_{\sigma_A}\|_{L^{r,\infty}(\VN_R(G))}
=
\sup_{t>0}t^{\frac1r}\mu_t(A) 
=
\sup_{t>0}t^{\frac1r}\sigma_A^*(t) \\
=
\sup_{t>0}s[d_{M_{\sigma_A}(s)}]^{\frac1p-\frac1q}
=
\sup_{t>0}
s
\left(
\int\limits_{\substack{\xi\in\RR^n\\ |\sigma_A(\xi)|\geq s}}d\xi
\right)^{\frac1p-\frac1q},
\end{multline}
where in the equality between the first and the second lines we used 
Proposition \ref{PROP:mu-t-properties} with $M=\VN_R(\RR^n)$.
This completes the proof.
\end{proof}

\subsection{The case of compact Lie groups}

In this section we compare Theorem \ref{THM:upper-bound} with known results in the case of $G$ being a compact Lie group.
The global symbolic calculus for operators $A$ acting on compact Lie groups has been introduced and consistently developed in \cite{Ruzhansky+Turunen-IMRN,RT}, to which we refer to
further details on global matrix symbols on compact Lie groups.
Here we also note that with this matrix global symbol, the Fourier multiplier $A$ must act by
multiplication on the Fourier transform side
$$\widehat{Af}(\xi)=\sigma_{A}(\xi)\widehat{f}(\xi), \;\xi\in\Gh,$$
where $\widehat{f}(\xi)=\int_G f(x)\xi(x)^* dx$ is the Fourier coefficient of $f$ at the representation $\xi\in\Gh$, where for simplicity we may identify $\xi$ with its equivalence class. 
As we have mentioned in \eqref{EQ:comp}, the $L^p$-$L^q$ boundedness of Fourier multipliers on compact Lie groups can be controlled by its symbol $\sigma_A(\xi)$. However,
Theorem \ref{THM:upper-bound} gives a better
result than the known estimate \eqref{EQ:comp}; for completeness we recall the exact statement:

\begin{thm}[\cite{ANRNotes2016}]
\label{THM:cmp}
Let $1<p\leq 2 \leq q <\infty$ and suppose that $A$ is a Fourier multiplier on the compact Lie group $G$. Then we have
\begin{equation}
\label{EQ:CG-FM-upper}
\|A\|_{L^p(G)\to L^q(G)}
\lesssim
\sup_{\substack{s\geq 0}}
s
\left(
\sum\limits_{\substack{\xi\in\Gh\colon \|\sigma_A(\xi)\|_{\op}\geq s}}
d^{2}_{\xi}
\right)^{\frac1p-\frac1q},
\end{equation}
where $\sigma_A(\xi)=\xi^*(g)A\xi(g)\big|_{g=e}\in {\mathbb C}^{d_{\xi}\times d_{\xi}}$ 
is the matrix symbol of $A$.
\end{thm}


The fact that Theorem \ref{THM:upper-bound} implies Theorem \ref{THM:cmp}
follows from the following result relating the noncommutative Lorentz norm to the global symbol
of invariant operators in the context of compact Lie groups:
\begin{prop} 
\label{PROP:comparison}
Let $1<p\leq 2 \leq q <\infty$ and let $p\not=q$ and $\frac1r=\frac1p-\frac1q$. Suppose $G$ is a compact Lie group and $A$ is a Fourier multiplier on $G$. Then we have
\begin{equation}
\label{EQ:comparison}
\|A\|_{L^{r,\infty}(VN_R(G))}
\leq
\sup_{s\geq 0}
s
\left(
\sum\limits_{\substack{\xi\in\Gh \\ \|\sigma_A(\xi)\|_{\op}\geq s}}d^2_{\xi}
\right)^{\frac1p-\frac1q},
\end{equation}
where $\sigma_A(\xi)=\xi^*(g)A\xi(g)\big|_{g=e}\in {\mathbb C}^{d_{\xi}\times d_{\xi}}$ 
is the matrix symbol of $A$.
\end{prop}

\begin{rem} 
\label{REM:tau-measurability-CG}
If $G$ is a compact Lie group,
the sufficient condition \eqref{EQ:CG-FM-upper} on the Fourier multiplier $A$ implies $\tau$-measurability of $A$ with respect to $\VN_R(G)$, so we do not need to assume it explicitly in the setting of compact Lie groups.
Indeed, the condition of $\tau$-measurability does not arise in the setting of compact Lie groups due to the fact \cite[Proposition 21, p. 16]{Terp1981} that
\begin{equation*}
A \text{ is $\tau$-measurable with respect to $M$} 
\end{equation*}
if and only if
\begin{equation}
\label{EQ:d_s-goes-to-zero}
\lim_{\lambda\to+\infty}d_{\lambda}(A)=0.
\end{equation}
Now, if the right hand side of \eqref{EQ:comparison} is finite, the latter condition holds.
Indeed, by Definition \ref{DEF:Lorenz-spaces}  we get
\begin{align}
\begin{split}
&\sup_{s>0}s[d_{s}(A)]^{\frac1r}
=
\sup_{t>0}t^{\frac1r}\mu_t(A)
\\
&=
\|A\|_{L^{r,\infty}(VN_L(G))}
\leq
\sup_{s>0}
s
\left(
\sum\limits_{\substack{\xi\in\Gh \\ \|\sigma_A(\xi)\|\geq s}}d^2_{\xi}
\right)^{\frac1p-\frac1q}
<+\infty,
\end{split}
\end{align}
where in the first equality we used \eqref{EQ:prop-16} with $\alpha=\frac1r$ from Proposition \ref{PROP:mu-t-properties}.
Thus, we have
\begin{equation}
d_{s}(A)
\leq
\frac{C}{s^r}.
\end{equation}
As a consequence, we obtain \eqref{EQ:d_s-goes-to-zero}.
This completes the proof.
\end{rem}

\begin{proof}[Proof of Proposition \ref{PROP:comparison}]
We first compute the norm $\|A\|_{L^{r,\infty}(VN_R(G))}$
with $\frac1r=\frac1p-\frac1q$, $p\not=q$.
By definition, we have
\begin{equation}
\label{EQ:def}
\|A\|_{L^{r,\infty}(VN_R(G))}
=
\sup_{t>0}t^{\frac1p-\frac1q}\mu_A(t).
\end{equation}
The application of the property \eqref{EQ:prop-16} from Proposition \ref{PROP:mu-t-properties} yields
\begin{equation*}
\sup_{t>0}t^{\frac1r}\mu_A(t)
=
\sup_{s>0}s[d_A(s)]^{\frac1p-\frac1q}.
\end{equation*}
Therefore, it is sufficient to show that
\begin{equation}
\label{EQ:sufficient-ff}
\sup_{s>0}s[d_A(s)]^{\frac1p-\frac1q}
\leq
\sup_{s>0}
s
\left(
\sum\limits_{\substack{\xi\in\Gh \\ \|\sigma_A(\xi)\|\geq s}}d^2_{\xi}
\right)^{\frac1p-\frac1q}.
\end{equation}
The polar decomposition for arbitrary closed densely defined possibly unbounded operators $A$ acting on a Hilbert space $\H$ has been established in \cite{Neumann1932}.
Thus, we apply  \cite[page 307, Theorem 7]
{Neumann1932} to get
\begin{equation}
\label{EQ:polar-decomposition}
A=W|A|,
\end{equation}
where $W$ is a partial isometry.
This means that the operators $W^*W$ and $WW^*$ are projections in $\H$.
If  $A$ is a left Fourier multiplier, then its modulus $|A|$ is affiliated with $\VN_R(G)$ as well:
\begin{lem}[{{\cite[p. 33, Lemma 4.4.1]{RO1936}}}] 
Let $M$ be a von Neumann algebra. Suppose $A$ is affiliated with $M$. Then $|A|$ is affiliated with $M$ as well and $W\in M$.
\end{lem}

To proceed, we will use the following property:

\begin{claim} 
\label{claim}
Let $A\in S(\VN_R(G))$ and let  $E_{[s,+\infty)}(|A|)$ be the spectral measure of $|A|$ corresponding to the interval $[s,+\infty)$. Then we have
\begin{equation}
\label{EQ:claim}
d_A(s)
=
\sum\limits_{\xi\in\Gh}d_{\xi}\sum\limits_{\substack{n=1,\ldots,d_{\xi} \\ s_{n,\xi}\geq s}}1,
\end{equation}
where for fixed $n=1,\ldots,d_{\xi}$, the number $s_{n,\xi}$ is the joint eigenvalue for the eigenfunctions $\xi_{kn}$, $k=1,\ldots,d_{\xi}$, of $|A|$. These functions $\xi_{kn},k=1,\ldots,d_{\xi}$, generate the subspace $\H^{n,\xi}=\Span\{\xi_{kn}\}^{d_{\xi}}_{k=1}$.

\end{claim}
We note that by Remark \ref{REM:FM-affiliation}, in view of the left invariance of
the operators $A$ and $|A|$, by the Peter-Weyl theorem 
they leave the spaces 
$\H^{n,\xi}$ invariant
(for the discussion of the spaces $\H^{n,\xi}$ in the context of the Peter-Weyl theorem
we refer to \cite[Theorem 7.5.14 and Remark 7.5.16]{RT}).

\medskip
Assuming Claim \ref{claim} for the moment, the proof proceeds as follows.
Without loss of generality, 
we can reorder, for each $\xi\in\Gh$, the numbers $s_{n,\xi}$ putting them in a decreasing order with respect to $n=1,\ldots,d_{\xi}$ (thus,
also reordering the corresponding eigenfunctions). Then we can estimate
\begin{equation*}
d_A(s)
=
\sum\limits_{\xi\in\Gh}d_{\xi}\sum\limits_{\substack{n=1,\ldots,d_{\xi} \\ s_{n,\xi}\geq s}}1
\leq
\sum\limits_{\xi\in\Gh}d_{\xi}\sum\limits_{\substack{n=1,\ldots,d_{\xi} \\ s_{1,\xi}\geq s}}1
=
\sum\limits_{\substack{\xi\in\Gh \\ s_{1,\xi}\geq s}}d^2_{\xi},
\end{equation*}
where in the first inequality we used the inclusion
\begin{equation}
\{
\xi\in\Gh, n=1,\ldots,d_{\xi}
\colon
s_{n,\xi}
\geq
s
\}
\subset
\{
\xi\in\Gh, n=1,\ldots,d_{\xi}
\colon
s_{1,\xi}
\geq
s
\}
\end{equation}
since for fixed $\xi\in\Gh$ the sequence $\{s_{n,\xi}\}^{d_{\xi}}_{n=1}$ monotonically decreases.
We notice that
$$
s_{1,\xi}=\|\sigma_A(\xi)\|_{\op}.
$$
Thus, we obtain
\begin{equation*}
d_A(s)
\leq
\sum\limits_{\substack{\xi\in\Gh \\ \|\sigma_A(\xi)\|_{\op} \geq s}}d^2_{\xi}.                                                                                                                                              
\end{equation*}
From this, we get
\begin{equation}
\label{EQ:almost-proved}
s
[d_A(s)]^{\frac1p-\frac1q}
\leq
s
\left(
\sum\limits_{\substack{\xi\in\Gh \\ \|\sigma_A(\xi)\|_{\op} \geq s}}d^2_{\xi}                                                                                                                                             
\right)^{\frac1p-\frac1q}.
\end{equation}
Taking supremum in the right-hand side of \eqref{EQ:almost-proved}, we get
\begin{equation}
\label{EQ:almost-proved-1}
s
[d_A(s)]^{\frac1p-\frac1q}
\leq
\sup_{s>0}
s
\left(
\sum\limits_{\substack{\xi\in\Gh \\ \|\sigma_A(\xi)\|_{\op} \geq s}}d^2_{\xi}                                                                                                                                           
\right)^{\frac1p-\frac1q}.
\end{equation}
Then taking again the supremum in the left-hand side of \eqref{EQ:almost-proved-1}, we finally obtain
$$
\sup_{s>0}
s
[d_A(s)]^{\frac1p-\frac1q}
\leq
\sup_{s>0}
s
\left(
\sum\limits_{\substack{\xi\in\Gh \\ \|\sigma_A(\xi)\|_{\op} \geq s}}d^2_{\xi}                                                                                                                                             
\right)^{\frac1p-\frac1q}.
$$
This proves \eqref{EQ:sufficient}.
Now, it remains to justify \eqref{EQ:claim} in Claim \ref{claim}.

Since $G$ is compact, its von Neumann algebra $M=VN_R(G)$ is a type I factor.
The trace $\tau$ for type I factors $M$ 
(and we denote it by $\Tr$ in this case)
can be given \cite[page 478]{Najmark1972} by
\begin{equation}
\Tr(A)=\int\limits^{+\infty}_{-\infty}\lambda \,d D_{M}(E_{\lambda}),\quad A\in M,
\end{equation}
where
\begin{equation}
D_M \colon M_+ \to \{1,2,3,\ldots\}
\end{equation}
is the dimension function introduced in \cite{RO1936,RO1937} and $M_+$ is the set of all hermitian ($A^*=A$) operators $A\in M$.
For each value of $\lambda$ the projection $E_{\lambda}$ is the sum of minimal mutually orthogonal projection operators, hence the value $D_M(E_{\lambda})$ can increase only in jumps and its points of growth $s_n$ are the characteristic values of the operator $A$. 
Thus, we get
\begin{equation}
\label{EQ:trace-A}
\Tr(A)=\sum\limits_{n\in\NN}m_ns_n,
\end{equation}
where $m_n$ is the corresponding jump of the function $D_M(E_{\lambda})$.
Further, we determine the singular values of $A$, or equivalently we will look for the eigenvalues of $|A|$.

Indeed, we recall that $|A|\big|_{\bigoplus_{k=1}^{d_{\xi}}\H^{k,\xi}}=\sigma_{|A|}(\xi)$ and use the fact that
$s_{1,\xi}=\|\sigma_{|A|}(\xi)\|_{\op}$.
 It is convenient to enumerate the singular values $s_{k,\xi}$ by two elements $(k,\xi)$, $k=1,\ldots,d_{\xi}$, in view of the decomposition into the closed subspaces invariant under the group action. We rewrite \eqref{EQ:trace-A} once again as the usual trace
\begin{equation}
\Tr(|A|)
=
\sum\limits_{\pi\in\Gh}d_{\xi}\sum\limits^{d_{\xi}}_{n=1}s_{n,\xi},
\end{equation}
where we write $s_{n,\xi}$ for the eigenvalue of the restriction $|A|\big|_{\bigoplus_{n=1}^{d_{\xi}}\H^{n,\xi}}$ of $|A|$ to the subspaces $\H^{k,\xi}$ which are spanned by the eigenfunctions $\xi_{kn}$, $n=1,\ldots,d_{\xi}$, corresponding to $s_{k,\xi}$.
In other words, the multiplicity of $s_{k,\xi}$ is $d_{\xi}$.
From this place, we write $\pi$ rather than $\xi$ to emphasize our choice of an element $\xi$ from the equivalence class $[\pi]$.
Each element $\pi\in\Gh$ can be realised as a finite-dimensional matrix via some choice of a basis in the representation space. Denote by $\pi_{kn}$ the matrix elements of $\pi$, i.e.
\begin{equation}
\pi\colon G\ni g \mapsto \pi(g)=[\pi_{kn}(g)]^{d_{\pi}}_{k,n=1} \times \C^{\dpi\times\dpi}.
\end{equation}
By the Peter-Weyl theorem (see e.g. \cite[Theorem 7.5.14]{RT}), we have the decomposition
\begin{equation}
L^2(G)
=
\bigoplus\limits_{\pi\in\Gh}
\bigoplus\limits^{\dpi}_{n=1}\Span\{\pi_{kn}\}^{\dpi}_{k=1}.
\end{equation}
In other words, we can write
\begin{equation}
L^2(G)\ni f 
= 
\sum\limits_{\pi\in\Gh}\dpi
\sum\limits^{\dpi}_{n=1}
\sum\limits^{\dpi}_{k=1}
(f,\pi_{kn})_{L^2(G)}\pi_{kn}\in
\bigoplus\limits_{\pi\in\Gh}
\bigoplus\limits^{\dpi}_{n=1}\Span\{\pi_{kn}\}^{\dpi}_{k=1}.
\end{equation}
The action of $A$ can be written in the form
\begin{equation}
Af
=
\sum\limits_{\pi\in\Gh}
\dpi
\sum\limits^{\dpi}_{n=1}
\sum\limits^{\dpi}_{k=1}
\sum\limits^{\dpi}_{s=1}
\sigma_A(\pi)_{ns}
(f,\pi_{ks})_{L^2(G)}
\pi_{kn},
\end{equation}
This implies
\begin{equation}
A=\bigoplus\limits_{\pi\in\Gh}\bigoplus\limits^{\dpi}_{n=1}\sigma_A(\pi),
\end{equation}
where $\sigma_A(\pi)$ is the global matrix symbol of $A$
(cf. \cite{Ruzhansky+Turunen-IMRN,RT}).
Then for the modulus $|A|=\sqrt{AA^*}$ we get
\begin{equation}
|A|=\bigoplus\limits_{\pi\in\Gh}\bigoplus\limits^{\dpi}_{n=1}|\sigma_A(\pi)|.
\end{equation}
Choosing a representative $\xi\in[\pi]$ from the equivalence class $[\pi]$,
we can diagonalise the matrix $|\sigma_A(\pi)|$ as
\begin{equation}
\sigma_{|A|}(\xi)
=
\left(
\begin{matrix}
s_{1,\xi} &  0 \ldots  & 0 \\
0 &  s_{2,\xi} \ldots &  0 \\
\vdots & \vdots& \vdots \\
0 &   \ldots &  s_{d_{\xi},\xi} \\
\end{matrix}
\right).
\end{equation}
Thus, we obtain
\begin{equation}
|A|f
=
\sum\limits_{\xi\in\Gh}
d_{\xi}
\sum\limits^{d_{\xi}}_{n=1}
s_{k,\xi}\cdot
\sum\limits^{d_{\xi}}_{k=1}
(f,\xi_{kn})_{L^2(G)}\xi_{kn}.
\end{equation}
Each $s_{n,\xi}$ is a joint eigenvalue of $|A|$ with the eigenfunctions $\xi_{kn}$
\begin{equation}
|A|\xi_{k,n}
=
s_{k,\xi}
\xi_{k,n},\quad n=1,\ldots,d_{\xi}.
\end{equation}
Since each singular value $s_{k,\xi},\,k=1,\ldots,d_{\xi}$, has the multiplicity $d_{\xi}$, we obtain
\begin{equation}
E_{[t,+\infty)}(|A|)
=
\bigoplus\limits_{\xi\in\Gh}
\bigoplus_{\substack{k=1,\ldots,d_{\xi}\\ s_{k,\xi}\geq t}}
E^{n,\xi},
\end{equation}
where $E^{n,\xi}$ is the projection to the left-invariant subspace $\Span\{\xi_{kn}\}^{d_{\xi}}_{k=1}$.
Consequently, we have
\begin{equation}
\Tr(E_{[t,+\infty)}(|A|))
=
\sum\limits_{\xi\in\Gh}d_{\xi}
\sum\limits_{\substack{k=1,\ldots,d_{\xi} \\ s_{k,\xi}\geq t}}1.
\end{equation}
The proof is now complete.
\end{proof}

\subsection{The case of non-invariant operators}

Theorem \ref{THM:upper-bound} can be extended to non-invariant operators, and also to the boundedness in Lorentz spaces.

For the formulation it is convenient to use the Schwartz-Bruhat spaces $\mathcal{S}(G)$ that have been developed by Bruhat \cite{Bruhat1961} as a way of doing distribution theory on locally compact groups. We briefly mention its basic properties and refer to \cite{Bruhat1961} for further details.
The space $\mathcal{S}(G)$ is a barrelled, bornological and complete locally convex topological vector space. It is continuously and densely contained in space $C_{c}(G)$ of compactly supported continuous functions. 
The space $\mathcal{S}(G)$ is dense in every $L^p(G)$, which follows from the fact that $C_c(G)$ is dense in $L^p(G)$.

\begin{thm} \label{THM:Lpq-G}
Let $G$ be a locally compact unimodular separable group.
Let $\mathcal{D}$ be a closed densely defined operator affiliated with $\VN_R(G)$ such that its  inverse $\mathcal{D}^{-1}$ is measurable with respect to $\VN_R(G)$ and such that for some $1<\beta\leq 2$ we have
\begin{equation}
\label{EQ:embedding-condition}
\|\mathcal{D}^{-1}\|_{L^{\beta}(\VN_R(G))}<+\infty.
\end{equation}
Let $A$ be a linear continuous operator on the Schwartz-Bruhat space $\mathcal{S}(G)$. 
Then for any $1<p\leq 2 \leq q< \infty$  and any $0<\theta<1$ we have
\begin{equation}
\label{EQ:Lpq-G}
\|A\|_{L^{p}(G)\to L^{q}(G)}
\lesssim
\left(
\int\limits_{G}
\left(
\|\mathcal{D}\circ A_u\|_{L^{r,\infty}(\VN_R(G))}
\right)^{\beta}
du
\right)^{\frac1{\beta}},
\end{equation}
where $\frac1r=\frac1p-\frac1q$.
\end{thm}
Here $\{A_u\}$ is the field of operators generated by varying the Schwartz kernel $K_A$ of $A$, for more details we refer to the proof of Theorem \ref{THM:Lpq-G}. But first we observe that choosing various $\mathcal{D}$, we get different inequalities in \eqref{EQ:embedding-condition}. 
Thus, before proving Theorem \ref{THM:Lpq-G}, we illustrate it in a few examples.

\begin{ex} Let $G$ be a compact Lie group of dimension $n$ and let $\mathcal{L}_{G}$ be the Laplace operator on $G$. Let us take $\mathcal{D}=(I-\mathcal{L}_{G})^{\frac{n}2}$. By the Weyl's asymptotic law, we get
$$
\lambda_k\cong k,
$$
where $\lambda_k$ are the eigenvalues of $\mathcal{D}$.
Then, up to constant, we obtain
$$
\|\mathcal{D}^{-1}\|^\beta_{L^{\beta}(\VN_R(G))}
\simeq
\sum\limits^{\infty}_{k=1}\frac1{k^{\beta}}<+\infty,
$$
for any $\beta>1$.
Thus, condition \eqref{EQ:embedding-condition} is satisfied.
\end{ex}

\begin{ex} 
Let us take $G$ to be the Heisenberg group $\HH^n$ with the homogeneous dimension $Q=2n+2$, and let $\mathcal{L}^{sub}_{\HH^n}$ be the canonical sub-Laplacian on $\HH^n$. It can be computed (see \eqref{EQ:tau-HH}) that
$$
\tau(E_{(0,s)}(-\mathcal{L}^{sub}_{\HH^n})=C_n s^{\frac{Q}2}.
$$
Using this and Definition \ref{DEF:mu-t}, it can be shown that
$$
\mu_t(
(I-\mathcal{L}^{sub}_{\HH^n})^{-\alpha}
)
=
\frac1
{\left(1+t^{\frac{2}{Q}}\right)^{\alpha}}.
$$
From this we obtain
\begin{equation}
\label{EQ:integral}
\|(I-\mathcal{L}^{sub}_{\HH^n})^{-\alpha}\|^\beta_{L^{\beta}(\VN_R(\HH^n))}
=
\int\limits^{+\infty}_0 \frac1{\left(1+t^{\frac{2}{Q}}\right)^{\alpha\beta}}dt,
\end{equation}
where we used the formula
$$
\tau(|A|^p)=\int\limits^{+\infty}_0\mu^p_t(A)\,dt
$$
established in \cite[Corollary 2.8, p. 278]{ThierryKosaki1986}.
The integral in \eqref{EQ:integral} is convergent if and only if $\alpha\beta>\frac{Q}2$.
\end{ex}
\nocite{Pesenson2008}

\begin{proof}[Proof of Theorem \ref{THM:Lpq-G}]
	  Let us define
	  $$
	  	A_uf(g)
	  	:=
		L_{K_A(u)}f(g)
		=
		\int\limits_{G}K_A(u,gt^{-1})f(t)dt,
	  $$
	  so that $A_gf(g)=Af$.
For each fixed $u\in G$ the operator $A_u$ is affiliated with $\VN_R(G)$. 
	   Then
\begin{equation}
\|Af\|_{L^q(G)}
=
\left(
\int\limits_{G}
|Af(g)|^q\,dg
\right)^{\frac1q}
\leq
\left(
\int\limits_{G}
\sup_{u\in G}
|A_{u}f(g)|^q\,dg
\right)^{\frac1q}.
\end{equation}

The subsequent proof will rely on the following theorem that we now assume to hold, and will prove it later:
\begin{thm} 
\label{THM:upper-bound-q-infty-1}
Let $G$ be a locally compact unimodular separable group and let $A$ be a left Fourier multipler on $G$. Let Let $1\leq \beta\leq 2$. Then we have
\begin{equation}
\|A\|_{L^{\beta}(G)\to L^{\infty}(G)}
\leq
\|A\|_{L^{\beta}(\VN_R(G))}.
\end{equation}
\end{thm}

By assuming Theorem \ref{THM:upper-bound-q-infty-1} for a moment and applying it to $A=\mathcal{D}^{-1}$
we get
\begin{equation}
 \sup_{u\in G} |A_u f(g)|
 =
 \sup_{u\in G} |\mathcal{D}^{-1}\,\mathcal{D}A_u f(g)|
  \leq
  \|\mathcal{D}^{-1}\|_{L^{\beta}(\VN_R(G)}
 \|\mathcal{D}A_u f\|_{L^{\beta}_u(G)}.
\end{equation}

Therefore, using the Minkowski integral inequality to change the order of integration, we obtain
\begin{eqnarray*}
    \| Af \|_{L^q(G)} 
\lesssim
  \left(
  \int\limits_{G}
  \left(
    \int\limits_{G}
    |\mathcal{D}\,A_uf(g)|^{\beta}
\,du
\right)^{\frac{q}{\beta}}dg
\right)^{\frac1q}
&=&
\\
\left[
\left\|
\int\limits_{G}
|\mathcal{D}\,A_uf(g)|^{\beta}\,du
\right\|_{L^{\frac{q}{\beta}}(G)}
\right]^{\frac{1}{\beta}}
\leq
\left[
\int\limits_{G}
\left\|
|\mathcal{D}\,A_uf(g)|^{\beta}
\right\|_{L^{\frac{q}{\beta}}(G)}
\,du
\right]^{\frac{1}{\beta}}
&=&
\\
\left(
\int\limits_{G}
\left(
\int\limits_{G}|\mathcal{D}\,A_uf(g)|^q\,dg
\right)^{\frac{\beta}{q}}
du
\right)^{\frac1{\beta}}
&\leq &
\\
\left(
\int\limits_{G}
\left(
\|\mathcal{D}\,A_u\|_{L^{r,\infty}(\VN_R(G))}
\right)^{\beta}
du
\right)^{\frac1{\beta}}
\|f\|_{L^p(G)},
\end{eqnarray*}
where the last inequality holds due to 
Theorem \ref{THM:upper-bound}.

So, it now remains to prove Theorem \ref{THM:upper-bound-q-infty-1}:

\begin{proof}[Proof of Theorem \ref{THM:upper-bound-q-infty-1}]
By $\Gh$ we shall mean the quasi-dual in the sense of \cite{Ernest1961,Ernest1962}. There is a canonical central decomposition
\begin{equation}
\label{EQ:ernest-decomposition}
Af(g)
=
\int\limits_{\Gh}\tau_{\pi}\left(\sigma_A(\pi)\widehat{f}(\pi)\pi(g)\right)d\mu(\pi).
\end{equation}
The uniqueness in \eqref{EQ:ernest-decomposition} is up to the quasi-equivalence \cite{Ernest1961,Ernest1962}.
For each $(\pi,\H^{\pi})\in\Gh$, the operator $\sigma_A(\pi)\widehat{f}(\pi)\pi(g)$ acts in the Hilbert space $\H^{\pi}$. Since $G$ is unimodular, every factor $\VN^{\pi}_R(G)=\{\pi(g)\}^{!!}_{g\in G}$ is either of type $I$ or type $II$. Hence, there always exists a trace $\tau_{\pi}$ on $\VN^{\pi}_R(G)$. 
By H\"older's inequality, we have

\begin{equation}
|\tau_{\pi} (\sigma_A(\pi)\widehat{f}(\pi)\pi(g))|
\leq
\bigl(\tau_{\pi}|\sigma_A(\pi)|^{\beta}\bigl)^{\frac1{\beta}}
\left(\tau_{\pi}|\widehat{f}(\pi)\pi(g)|^{\beta'}\right)^{\frac1{{\beta'}}}.
\end{equation}
The application of \cite[Corollary 2.8 on p. 278]{ThierryKosaki1986} to $\tau_{\pi}|\widehat{f}(\pi)\pi(g)|^{\beta'}$ yields
\begin{equation}
\tau_{\pi}|\widehat{f}(\pi)\pi(g)|^{\beta'}
=
\int\limits^{+\infty}_{0}\mu_t(\widehat{f}(\pi)\pi(g))^{\beta'}\,dt.
\end{equation}
Using property \eqref{EQ:mu-t-A-properties-3} of Lemma \ref{LEM:mu-t-A-properties}, we estimate
\begin{equation}
\mu_t(\widehat{f}(\pi)\pi(g))
\leq
\mu_t(\widehat{f}(\pi)),\quad g\in G.
\end{equation}
 The absolute value trace $\tau_{\pi} (\sigma_A(\pi)\widehat{f}(\pi)\pi(g))$ of $\sigma_A(\pi)\widehat{f}(\pi)\pi(g)$ can then be estimated from above
 \begin{equation}
 \left|\tau_{\pi}(\sigma_A(\pi)\widehat{f}(\pi)\pi(g)\right|)
 \leq
\bigl(\tau_{\pi}|\sigma_A(\pi)|^\beta\bigl)^{\frac1{\beta}}
\left(\tau_{\pi}|\widehat{f}(\pi)|^{\beta'}\right)^{\frac1{{\beta'}}}.
 \end{equation}
 Thus, we get
 \begin{multline}
 |Af(g)|
 \leq
 \int\limits_{\Gh}
  \left|\tau_{\pi}(\sigma_A(\pi)\widehat{f}(\pi)\pi(g)\right|d\mu(\pi))
 \leq
 \int\limits_{\Gh}  
 \bigl(\tau_{\pi}|\sigma_A(\pi)|^\beta\bigl)^{\frac1{\beta}}
\left(\tau_{\pi}|\widehat{f}(\pi)|^{\beta'}\right)^{\frac1{{\beta'}}}
d\mu(\pi)
 \\\leq
 \left(
 \int\limits_{\Gh}  
\tau_{\pi}|\sigma_A(\pi)|^\beta d\mu(\pi) 
\right)^{\frac1{\beta}}
\left(
  \int\limits_{\Gh}  
\tau_{\pi}|\widehat{f}(\pi)|^{\beta'}
d\mu(\pi) 
\right)^{\frac1{{\beta'}}},
 \end{multline}
where  the last inequality is H\"older inequality.
Borel calculus and reduction theory for unbounded affiliated operators have been investigated in \cite[Section 4]{DNSZ2016}.
It can be shown \cite[Proposition 4.2, p.8]{DNSZ2016} that
\begin{equation}
|A|^p
=
\bigoplus\limits_{\Gh}
\int
\left|\sigma_A(\pi)\right|^pd\pi.
\end{equation}
Then, by \cite[Lemma 5.3, p.12]{DNSZ2016}, we get
\begin{equation}
\|A\|_{L^{\beta}(\VN_R(G))}
=
\left(
\tau(\left|A\right|^{\beta})
\right)^{\frac1{\beta}}
=
 \left(
 \int\limits_{\Gh}  
\tau_{\pi}(|\sigma_A(\pi)|^\beta)
d\mu(\pi)
\right)^{\frac1{\beta}}.
\end{equation}
By the Hausdorff-Young inequality \cite{Kunze1958} we have
\begin{equation}
\left(
  \int\limits_{\Gh}  
\tau_{\pi}|\widehat{f}(\pi)|^{\beta'}
d\mu(\pi) 
\right)^{\frac1{{\beta'}}}
\leq
\|f\|_{L^\beta(G)},\quad 1<\beta\leq 2.
\end{equation}
Finally, collecting all the inequalities, we obtain
\begin{equation}
\|Af\|_{L^{\infty}(G)}
\leq
\|A\|_{L^\beta(\VN_R(G))}
\|f\|_{L^\beta(G)},\quad 1<\beta\leq 2.
\end{equation}
The argument above can be modified for the case $\beta=1$ as well.
This completes the proof of Theorem \ref{THM:upper-bound-q-infty-1}.
\end{proof}

And this also completes the proof of Theorem \ref{THM:Lpq-G}.
\end{proof}

\section{Spectral multipliers on locally compact groups}

In this and next section we will give an application of Theorem \ref{THM:upper-bound} to spectral multipliers.  

The classical Laplace operator $\Delta_{\RR^n}$ is affiliated with the von Neumann algebra $\VN(\RR^n)=\VN_L(\RR^n)=\VN_R(\RR^n)$ of all convolution operators, but is not measurable on $\VN(\RR^n)$. However,  the Bessel potential $(I-\Delta_{\RR^n})^{-\frac{s}2}$ is measurable with respect to $\VN(\RR^n)$. Therefore, one of the aims of spectral multiplier theorems is to ``renormalise" operators in Hilbert space $\H$ making them not only measurable but also bounded.
In the next theorem we first describe such a relation for general semifinite von Neumann algebras, and then in Corollary \ref{COR:sp-m} give its application to spectral multipliers.
\begin{thm}
\label{THM:varphi-L}
Let $\mathcal{L}$ be a closed unbouned operator affiliated with a semifinite von Neumann algebra $M\subset B(\H)$. Assume that $\varphi$ is a monotonically decreasing continuous function on $[0,+\infty)$ such that
\begin{eqnarray}
\label{EQ:phi-normalization}
\varphi(0)=1,
\\
\label{EQ:phi-empty-energy}
\lim_{u\to+\infty}\varphi(u)=0.
\end{eqnarray}
Then for every $1\leq r<\infty$ we have the equality
\begin{equation}
\|\varphi(|\mathcal{L}|)\|_{L^{r,\infty}(M)}
=
\sup_{u>0}
\left(\tau(E_{(0,u)}(|\mathcal{L}|))\right)^{\frac1r} \varphi(u)<+\infty.
\end{equation}
\end{thm}
Let $\mathcal{L}$ be an arbitrary unbounded linear operator affiliated with $(M,\tau)$. Then Theorem \ref{THM:varphi-L} says that the function $\varphi(|\mathcal{L}|)$ is necessarily affiliated with $(M,\tau)$
 and  $\varphi(|\mathcal{L}|)\in (M,\tau)$ if and only if the $r$-th power $\varphi^r$ of $\varphi$ grows at infitiy not faster than $\frac1{\tau(E_{(0,u)}(|\mathcal{L}|))}$, i.e. if we have the estimate
\begin{equation}
\label{EQ:counting-function}
\varphi(u)^r
\lesssim
\frac1{ \tau(E_{(0,u)}(|\mathcal{L}|))}.
\end{equation}

We now give a corollary of Theorem \ref{THM:varphi-L} for $M=\VN_R(G)$ being the right von Neumann algebra of a locally compact unimodular group. This is formulated in Theorem \ref{THM:varphi-L-intro} but we recall it here for readers' convenience.

\begin{cor} \label{COR:sp-m}
Let $G$ be a locally compact unimodular separable group and let $\mathcal{L}$ be a left Fourier multiplier on $G$.
Let $\varphi$ be as in Theorem \ref{THM:varphi-L}  Then we have the inequality
\begin{equation}
\|\varphi(|\mathcal{L}|)\|_{L^p(G)\to L^q(G)}
\lesssim
\sup_{u>0}\varphi(u)
\left[\tau(E_{(0,u)}(|\mathcal{L}|))\right]^{\frac1p-\frac1q},\quad 1<p\leq 2 \leq q<\infty.
\end{equation}
\end{cor}
This corollary follows immediately from combining Theorem \ref{THM:upper-bound} and Theorem \ref{THM:varphi-L} with $M=\VN_R(G)$, also proving Theorem \ref{THM:varphi-L-intro}.

For completeness, we give another corollary (of the proof of Theorem \ref{THM:varphi-L}) without assuming that $\varphi$ is monotone, continuous, and satisfies conditions \eqref{EQ:phi-normalization}-\eqref{EQ:phi-empty-energy}.
It is these conditions that allow us to rewrite Corollary \ref{COR:gen} in the more applcable form of Corollary \ref{COR:sp-m}. 

\begin{cor} \label{COR:gen}
Let $G$ be a locally compact unimodular separable group and let $\mathcal{L}$ be a left Fourier multiplier on $G$.
Let $\varphi$ be a Borel measurable function on the spectrum $\Sp(\left|\mathcal{L}\right|)$. 
Then we have the inequality
\begin{equation}
\|\varphi(|\mathcal{L}|)\|_{L^p(G)\to L^q(G)}
\lesssim
\sup_{s>0}s[\tau(E_{(s,+\infty)})(\varphi(|\mathcal{L}|))]^{\frac1p-\frac1q},\quad 1<p\leq 2 \leq q<\infty.
\end{equation}
\end{cor}
We will prove this corollary together with the proof of Theorem \ref{THM:varphi-L}.

\begin{proof}[Proof of Theorem \ref{THM:varphi-L}]
By defintion
\begin{equation*}
\|\varphi(|\mathcal{L}|)\|_{L^{r,\infty}(M)}
=
\sup_{t>0} t^{\frac1p-\frac1q}\mu_t(\varphi(|\mathcal{L}|)),\quad \frac1r=\frac1p-\frac1q.
\end{equation*}
Using Property \eqref{EQ:prop-16} from
Proposition \ref{PROP:mu-t-properties}, we get
\begin{equation*}
\sup_{t>0} t^{\frac1p-\frac1q}\mu_t(\varphi(|\mathcal{L}|))
=
\sup_{s>0}s[\tau(E_{(s,+\infty)})(\varphi(|\mathcal{L}|))]^{\frac1p-\frac1q}.
\end{equation*}
Hence, we have
\begin{equation}
\label{EQ:prove-this}
\|\varphi(|\mathcal{L}|)\|_{L^{r,\infty}(M)}
=
\sup_{s>0}s[\tau(E_{(s,+\infty)})(\varphi(|\mathcal{L}|))]^{\frac1p-\frac1q}.
\end{equation}
Since $\mathcal{L}$ is affiliated with $M$ the spectral projections $E_{\Omega}(|\mathcal{L}|)$ belong to $M$.
Let $\langle \mathcal{L}\rangle$ be an abelian subalgebra of $M$ generated by the spectral projectors $E_{(\lambda,+\infty)}(|\mathcal{L}|)$. Let $\varphi$ be a Borel measurable function on the spectrum $\Sp(\left|\mathcal{L}\right|)$. Then by Borel functional calculus \cite[Section 2.6]{Arveson2006} it is possible to construct the operator  $\varphi(|\mathcal{L}|)$. This operator is a strong limit of the spectral projections $E_{\Omega}(|\mathcal{L}|)\in M$. Therefore $\varphi(|\mathcal{L}|)$ is affiliated with $M$. 
The distribution function of the operator $\varphi(|\mathcal{L}|)$ is given by
\begin{equation}
d_{s}(\varphi(|\mathcal{L}|))
=
\tau(E_{(s,+\infty)}(\varphi(|\mathcal{L}|)).
\end{equation}
This proves Corollary \ref{COR:gen}.

Using \cite[Corollary 5.6.29, p.363]{KR1997} and the spectral mapping theorem (see \cite[Theorem 4.1.6]{KR1997}), we obtain
\begin{equation}
\label{EQ:comp-map}
\tau(E_{(s,+\infty)}(\varphi(|\mathcal{L}|))) 
= 
\tau(E_{\varphi^{-1}(s,+\infty)}(\varphi^{-1}\circ \varphi(|\mathcal{L}|))) 
=
\tau(E_{(0,\varphi^{-1}(s))}(|\mathcal{\mathcal{L}}|).
\end{equation}
From the hypothesis  \eqref{EQ:phi-empty-energy} imposed on $\varphi$ and using \eqref{EQ:comp-map}, we get
%
\begin{equation}
\lim_{s\to+\infty} 
\tau(E_{(s,+\infty)}(\varphi(|\mathcal{L}|))) 
= 
\lim_{s\to+\infty}  
\tau(E_{(0,\varphi^{-1}(s))}(|\mathcal{\mathcal{L}}|) = 0.
\end{equation}
Hence, the operator $\varphi(|\mathcal{L}|)$ is $\tau$-measurable with respect to $\VN_R(G)$.
Combining \eqref{EQ:prove-this} and \eqref{EQ:comp-map}, we finally obtain
\begin{eqnarray*}
\|\varphi(|\mathcal{L}|)\|_{L^{r,\infty}(M)}
=
\sup_{t>0} t^{\frac1p-\frac1q}\mu_t(\varphi(|\mathcal{L}|))
=
\sup_{s>0}s[\tau(E_{(s,+\infty)})(\varphi(|\mathcal{L}|))]^{\frac1p-\frac1q}
\\=
\sup_{s>0}s
[\tau(E_{(0,\varphi^{-1}(s))}(|\mathcal{L}|)]^{\frac1p-\frac1q}
=
\sup_{u>0} \varphi(u)[\tau(E_{(0,u)}(|\mathcal{L}|)]^{\frac1p-\frac1q},
\end{eqnarray*}
where in the last equality we used the monotonicity of $\varphi$. 
This completes the proof of Theorem \ref{THM:varphi-L}.
\end{proof}

\section{Heat kernels and embedding theorems}
\label{SEC:heat}

In this section we show that the spectral multipliers estimate \eqref{COR:sp-m} may be also used to relate spectral properties of the operators with the time decay rates for propagators for the corresponding evolution equations.
We illustrate this in the case of the heat equation, when the the functional calculus and the application of Theorem \ref{THM:varphi-L} to a family of functions $\{e^{-ts}\}_{t>0}$ yield the time decay rate for the solution $u=u(t,x)$ to the heat equation
$$
\partial_t u+{\mathcal L}u=0,\quad u(0)=u_0.
$$
For each $t>0$, we apply Borel functional calculus \cite[Section 2.6]{Arveson2006} to get
\begin{equation}\label{EQ:L-heat-equation}
u(t,x)=e^{-t\mathcal{L}}u_0.
\end{equation}
One can  check that $u(t,x)$ satisfies equation \eqref{EQ:L-heat-equation} and the initial condition.
Then by Theorem \ref{THM:upper-bound}, we get
\begin{equation}
\|u(t,\cdot)\|_{L^{q}(G)}
\leq
\|e^{-t\mathcal{L}}\|_{L^{r,\infty}(\VN_R(G)}
\|u_0\|_{L^p(G)},
\end{equation}
reducing the $L^p$-$L^q$ properties of the propagator to the time asymptotics of its noncommutative Lorentz space norm.

\begin{cor}[The $\mathcal{L}$-heat equation]
\label{EX:heat-equation}
Let $G$ be a locally compact unimodular separable group and let $\mathcal{L}$ be an unbounded positive  operator affiliated with $\VN_R(G)$ such that for some $\alpha$ we have
\begin{equation}
\label{EQ:asymptotics-condition}
\tau(E_{(0,s)}(\mathcal{L}))\lesssim s^{\alpha},\quad s\to \infty.
\end{equation}
Then for any $1<p\leq 2\leq q<\infty$ we have
\begin{equation}
\|e^{-t\mathcal{L}}\|_{L^p(G)\to L^q(G)}
\leq
C_{\alpha,p,q}
{t^{-\alpha \left(\frac1p-\frac1q\right) }},\quad t>0.
\end{equation}
\end{cor}

\begin{proof}[Proof of Theorem \ref{EX:heat-equation}]
The application of Theorem \ref{THM:varphi-L} yields
\begin{equation*}
\|e^{-t\mathcal{L}}\|_{L^{r,\infty}(\VN_R(G))}
=
\sup_{s>0}[\tau(E_{(0,s)}(|\mathcal{L}|)]^{\frac1r}e^{-ts}.
\end{equation*}
Now, using this and hypothesis \eqref{EQ:asymptotics-condition}, we get
\begin{equation*}
\|e^{-t\mathcal{L}}\|_{L^{r,\infty}(\VN_R(G))}
\lesssim
\sup_{s>0}s^{\frac{\alpha}r}e^{-ts}.
\end{equation*}
The standart theorems of mathematical analysis yield that
\begin{equation}
\label{EQ:mathan}
\sup_{s>0}
s^{\frac{\alpha}r}e^{-ts}
=
\left(
\frac{\alpha}{tr}
\right)^{\frac{\alpha}{r}}
e^{-\frac{\alpha}{r}}.
\end{equation}
Indeed, let us consider a function
$$
\varphi(s)
=
s^{\frac{\alpha}r}e^{-ts}.
$$
We compute its derivative 
$$
\varphi'(s)
=
s^{\frac{\alpha}r-1}e^{-ts}\left(\frac{\alpha}r-st\right).
$$
The only zero is $s_0 = \frac{\alpha}{rt}$ and the derivative $\varphi'(s)$ changes its sign from positive to negative at $s_0$. Thus, the point $s_0$ is a point of maximum. This shows \eqref{EQ:mathan} and completes the proof.
\end{proof}

Let us now show an application of Theorem \ref{THM:varphi-L} in the case of $\varphi(s)=\frac1{(1+s)^{\gamma}}$, $s\geq 0$.
It shows that for the range $1<p\leq 2\leq q<\infty$, the Sobolev type embedding theorems for an operator $\mathcal L$ depend only on the spectral behaviour of $\mathcal L$.

\begin{cor}[Embedding theorems]
\label{EX:embedding}
Let $G$ be a locally compact unimodular separable group and let $\mathcal{L}$ be an unbounded positive operator affiliated with $\VN_R(G)$ such that for some $\alpha$ we have
\begin{equation}
\label{EQ:asymptotics-conditiont}
\tau(E_{(0,s)}(\mathcal{L}))\lesssim s^{\alpha},\quad s\to \infty.
\end{equation}
Then for any $1<p\leq 2\leq q<\infty$ we have
\begin{equation}
\|f\|_{L^q(G)}
\leq
C
\|(1+\mathcal L)^{\gamma}f\|_{L^p(G)},
\end{equation}
provided that
\begin{equation}\label{EQ:0-s-embedding2t}
\gamma\geq \alpha\left(\frac1p-\frac1q\right),\quad 1<p\leq 2 \leq q<\infty.
\end{equation}
\end{cor}
\begin{proof}
By Theorem \ref{THM:varphi-L} with $\varphi(s)=\frac1{(1+s)^{\gamma}}$ and $\frac1r=\frac1p-\frac1q$ we have 
\begin{equation*}
\|(1+\mathcal L)^{-\gamma}\|_{L^p(G)\to L^q(G)}\lesssim
\|(1+\mathcal L)^{-\gamma}\|_{L^{r,\infty}(\VN_R(G))}
\lesssim
\sup_{s>0}s^{\frac{\alpha}r}(1+s)^{-\gamma}.
\end{equation*}
This supremum is finite for 
$\gamma\geq \frac{\alpha}r$, giving the condition \eqref{EQ:0-s-embedding2t}.
\end{proof}

Now, we illustrate Theorem \ref{THM:varphi-L} and Corollary \ref{EX:heat-equation} on a number of further examples, showing that the spectral estimate \eqref{EQ:asymptotics-condition} required for the $L^p$-$L^q$ estimate can be readily obtained in different situations.

In Example \ref{EX:a-hom-mu} below we illustrate condition \eqref{EQ:asymptotics-condition} in Theorem \ref{EX:heat-equation} for the homogeneous operator $\Op(a)\colon L^2(\RR^n)\to L^2(\RR^n)$ of order $\mu\in\RR$.
\begin{ex}
\label{EX:a-hom-mu}
Let $a(\xi)$ be a homogeneous function of degree $\mu$ and let $\Op(a)$ be the linear operator given by
$$
\widehat{\Op(a)}(\xi)=a(\xi)\widehat{f}(\xi),\quad f\in S(\RR^n),\xi\in\RR^n.
$$
According to the general theory (see \cite[Theorem 5.6.26, p.360]{KR1997} and \cite[Corollary 5.6.29, p.363]{KR1997}), the spectral projection $E_{(0,s)}(\left|\Op(a)\right|)$ corresponds to the multiplication by $\chi_{(0,s)}(\left|a(\xi)\right|)$, where $\chi_{(0,s)}(u)$ is the characteristic function of the inteval $(0,s)$. Then the trace $\tau(E_{(0,s)}(\left|\Op(a)\right|))$ can be computed as follows
\begin{equation}
\tau(E_{(0,s)}(\left|\Op(a)\right|))
=
\int\limits_{\substack{\RR^n\\ \left|a(\xi)\right|\leq s}}d\xi
=
\int\limits_{\substack{u\in\RR^n\\ \left|a(u)\right|\leq 1}}
s^{\frac{n}{\mu}}d\,u
=
Cs^{\frac{n}{\mu}},
\end{equation}
where we made the substitution $\xi\to s^{\frac1{\mu}}u$.
Hence, we get
\begin{equation}
\tau(E_{(0,s)}(\left|\Op(a)\right|)
=
Cs^{\frac{n}{\mu}},
\end{equation}
where $C=\int\limits_{\substack{\xi\in\RR^n\\ \left|a(\xi)\right|\leq 1}}d\xi$.
The application of Theorem \ref{EX:heat-equation} yields that
if
$$
C=\int\limits_{\substack{\xi\in\RR^n\\ \left|a(\xi)\right|\leq 1}}d\xi<\infty,
$$
then
$$
\|e^{-t\Op(a)}\|_{L^p(\RR^n)\to L^q(\RR^n)}
\leq
c_{\alpha\,p\,q}t^{-\frac{n}{\mu}\left(\frac1p-\frac1q\right)},\quad 1<p\leq 2 \leq q<\infty.
$$
\end{ex}

\subsection{Sub-Riemannian structures on compact Lie groups}

First we consider the example of sub-Laplacians on compact Lie groups in which case the number $\alpha$ in  \eqref{EQ:asymptotics-condition} can be related to the Hausdorff dimension generated by the control distance of the sub-Laplacian. Moreover, we illustrate Theorem \ref{THM:varphi-L} with examples of other functions $\varphi$ than in Corollary \ref{EX:heat-equation}, for example $\varphi(s)=\frac1{(1+s)^{\alpha/2}}$, leading to the Sobolev embedding theorems.

\begin{ex} 
\label{EX:subLaplace}
Let $\mathcal{L}=-\Delta_{sub}$ be the sub-Laplacian on a compact Lie group $G$, with discrete spectrum $\lambda_k$. 
Then by \cite{HK2015} the trace of the spectral projections $E_{(0,s)}(\mathcal{L})$ 
has the following asymptotics 
\begin{equation}
\tau(E_{(0,s)}(\mathcal{L}))
\lesssim s^{\frac{Q}2},\quad \text{ as $s\to+\infty$},
\end{equation}
where $Q$ is the Hausdorff dimension of $G$ with respect to the control distance generated by the sub-Laplacian.	
Let $u(t)$ be the solution to $\Delta_{sub}$-heat equation
\begin{eqnarray*}
\label{EQ:subLaplace-heat-equation}
\frac{\partial}{\partial t}u(t,x)-{\Delta_{sub}} u(t,x)=0,\quad t>0,\\
\label{EQ:subLaplace-heat-condition}
u(0,x)=u_0(x),\quad u_0\in L^p(G),\quad 1<p\leq 2.
\end{eqnarray*}
Then by Corollary \ref{EX:heat-equation}, we obtain
\begin{equation}
\|u(t,\cdot)\|_{L^q(G)}
\leq
C_{n,p,q}
{t^{-\frac{Q}{2}\left(\frac1p-\frac1q\right)}}
\|u_0\|_{L^p(G)},\quad 1<p\leq 2\leq q<+\infty.
\end{equation}

Let us now take $\varphi(s)=\frac1{(1+s)^{a/2}}$, $s\geq 0$.
Then by Theorem \ref{THM:varphi-L} the operator $\varphi(-\Delta_{sub})=(I-\Delta_{sub})^{-a/2}$ is $L^p(G)$-$L^q(G)$ bounded and the inequality 
\begin{equation}
\label{EQ:0-s-embedding}
\|f\|_{L^q(G)}
\leq
C
\|(1-\Delta_{sub})^{a/2})f\|_{L^p(G)}
\end{equation}
holds true provided that
\begin{equation}\label{EQ:0-s-embedding2}
a\geq Q\left(\frac1p-\frac1q\right),\quad 1<p\leq 2 \leq q<\infty.
\end{equation}
Here the constant $C$ in \eqref{EQ:0-s-embedding} is given by
$$
C:=\|(I-\Delta_{sub})^{-a/2}\|_{L^{r,\infty}(\VN_R(G))}.
$$
One can always associate with $\Delta_{sub}$ a version of Sobolev spaces. 
Let us define
\begin{equation}
\|f\|_{W^{a,p}_{\Delta_{sub}}(G)}:=\|(I-\Delta_{sub})^{a/2}f\|_{L^p(G)}.
\end{equation}
Then the Borel functional calculus (see e.g. \cite{Arveson2006}) together with \eqref{EQ:0-s-embedding}-\eqref{EQ:0-s-embedding2}  immediately yield
\begin{equation}
\label{EQ:sub-Laplacian-Sobolev-embedding}
\|f\|_{W^{b,q}_{\Delta_{sub}}(G)}
\leq
C
\|f\|_{W^{a,p}_{\Delta_{sub}}(G)},\quad a-b\geq Q\left(\frac1p-\frac1q\right).
\end{equation}
Each sub-Riemannian structure yields a sub-Laplacian $\Delta_{sub}$ on $G$ .
If we fix a group von Neumann algebra $\VN_R(G)$, then inequality \eqref{EQ:sub-Laplacian-Sobolev-embedding} depends only on the values of the trace $\tau$ on the algebra $\VN_R(G)$ and not on a particular choice of a sub-Laplacian $\Delta_{sub}$.
Similarly, the Sobolev spaces $W^{a,p}_{\Delta_{sub}}(G)$ do not depend on a particular choice of a sub-Laplacian.
\end{ex}

\subsection{Rockland operators on the Heisenberg group}
Let $G$ be the simply connected Heisenberg group and $\pi$ be an irreducible unitary representation of $G=\HH$.  

Let $X_1,X_2,\ldots, X_n, Y_1,Y_2,\ldots, Y_n, H$ be a basis in the Lie algebra $\mathfrak{h^n}$ of the Heisenberg group $\HH^n$ such that $[X_k,Y_k]=H$.
Let $\mathcal{R}=\sum\limits^n_{k=1}X^{2j}_k+\sum\limits^n_{k=1}Y^{2j}_k$ be the positive Rockland operator and its symbol $\sigma_{\mathcal{R}}(\pi)$ is given \cite[p.532]{FR2016} by 
\begin{equation}
\sigma_{\mathcal{R}}(\pi)
=
\left|\lambda\right|^k
\left(
\sum\limits^n_{k=1}
\partial^{2k}_u
-\left|u\right|^{2k}\right).
\end{equation}
It can then be shown \cite[Theorem 5.1]{Ter1994} that the asymptotics of the eigenvalues $s^{\lambda}_m$ is as follows
\begin{equation}
s^{\lambda}_m
\cong
\left|\lambda\right|^k \prod\limits^n_{k=1} {m_k}^{2j}.
\end{equation}
Thus, we get
\begin{equation}
\tau(E_{(0,s)}(\left|\mathcal{R}\right|)
=
\int\limits_{\widehat{\HH}}
\tau_{\lambda}(E_{(0,s)}(\left|\sigma_{\mathcal{R}}(\pi^{\lambda})\right|)
=
\int\limits_{\substack{\lambda\in\RR\\ \lambda\neq 0}}
|\lambda|^{n}d\lambda
\sum\limits_{\substack{m\in\NN^n \\ s^{\lambda}_m\leq s}}1
\cong s^{\frac{Q}{2j}}
\end{equation}
determining the value of $\alpha$ in \eqref{EQ:asymptotics-condition}.
\subsection{Sub-Laplacian on the Heisenberg group}

Here we look at the example of the Heisenberg group determining the value of $\alpha$ in  \eqref{EQ:asymptotics-condition} for the sub-Laplacian. The interesting point here is that while the spectrum of the sub-Laplacian is continuous, Theorem \ref{THM:varphi-L} can be effectively used in this situation as well.

\begin{ex} 
\label{COR:subLaplacian-functions}
Let $\mathcal{L}$ be the positive sub-Laplacian on the Heisenberg group $\HH^n$ and let $Q=2n+2$ be the homogeneous dimension of $\HH^n$. We claim that
\begin{equation}
\label{EQ:tau_E_H}
\tau(E_{(0,s)}(\mathcal{L})\simeq s^{Q/2}.
\end{equation}
Thus, under conditions of Theorem \ref{THM:varphi-L} on $\varphi$, the spectral multiplier $\varphi(\mathcal{L})$ is 
$\tau$-measu\-rable with respect to $\VN_R(\HH^n)$ and 
\begin{equation}
\label{EQ:L-r-infty-computed-HH}
\|\varphi(\mathcal{L})\|_{L^{r,\infty}(\VN_R(\HH^n))}\simeq \sup_{u>0} u^{\frac{Q}{2r}}\varphi(u),\quad \frac1r=\frac1p-\frac1q.
\end{equation}
For example, by choosing  $\varphi(u)=\frac1{(1+u)^{a/2}}$, $\alpha>0$, we recover the Sobolev embedding inequalities 
%
\begin{equation}
\label{EQ:sobolev-embedding-HH}
\|(I+\mathcal{L})^{b/2}f\|_{L^q(\HH^n)}
\leq
C
\|(I+\mathcal{L})^{a/2}f\|_{L^p(\HH^n)},
\end{equation}
provided
\begin{equation}
\label{EQ:alpha-beta-embedding}
a-b\geq Q\left(\frac1p-\frac1q\right).
\end{equation}
Inequality \eqref{EQ:sobolev-embedding-HH} has been 
established by Folland \cite{Folland1975}, and it can be extended further  
for Rockland operators on general graded Lie groups \cite{FR2016}.
\end{ex}

\begin{proof}[Proof of Example \ref{COR:subLaplacian-functions}]
By Theorem \ref{THM:upper-bound}, we get
\begin{equation}
\label{EQ:THM-COR}
\|\varphi(\mathcal{L})\|_{L^p(\HH^n)\to L^q(\HH^n)}
\lesssim
\|\varphi(\mathcal{L})\|_{L^{r,\infty}(\VN_R(\HH^n))}.
\end{equation}
Hence it is sufficient to find the conditions on $\varphi$ so that the right-hand side in \eqref{EQ:THM-COR} is finite. By Theorem \ref{THM:varphi-L} we have
\begin{equation}
\label{EQ:L-r-infty-sub-Laplacian}
\|\varphi(\mathcal{L})\|_{L^{r,\infty}(\VN_R(G))}
=
\sup_{u>0}[\tau(E_{(0,u)}(|\mathcal{L}))]^{\frac1r}\varphi(u).
\end{equation}
We shall now show \eqref{EQ:tau_E_H}.
Since $\mathcal{L}$ is affiliated with $\VN_R(\HH^n)$ it can be 
decomposed (\cite[Theorem 1 on page 187]{VNA-Dixmier-1981})
\begin{equation}
\mathcal{L}
=
\bigoplus\limits_{\widehat{\HH^n}}\int \mathcal{L}_{\lambda}d\nu(\lambda)
\end{equation}
with respect to the center $$C=\VN_R(\HH^n)\cap\VN_R(\HH^n)^{!}$$ of the group von Neumann algebra $\VN_R(\HH^n)$.
Here the collection $\{\mathcal{L}_{\lambda}\}_{\lambda\in\widehat{\HH^n}}$ of the (densely defined) operators $\mathcal{L}_{\lambda}\colon L^2(\RR^n)\to L^2(\RR^n)$ can be interpreted as the global symbol of the operator $\mathcal{L}$, as developed in \cite{FR2016}.

Hence, the spectral projections $E_{(0,s)}(\mathcal{L})$ can be decomposed
\begin{equation}
E_{(0,s)}(\mathcal{L})
=
\bigoplus\limits_{\widehat{\HH^n}}\int
E_{(0,s)}(\mathcal{L}_{\lambda})|\lambda|^nd\lambda.
\end{equation}
As a consequence \cite[Theorem 1 on page 225]{VNA-Dixmier-1981},  we get 
\begin{equation}
\label{EQ:trace-decomposition}
\tau(E_{(0,s)}(\mathcal{L})
=
\int\limits_{\HH^n}
\tau(E_{(0,s)}[\mathcal{L}_{\lambda}])|\lambda|^nd\lambda.
\end{equation}
The global  symbol $\mathcal{L}_{\lambda}\colon {\mathcal S}(\RR^n)\subset L^2(\RR^n)\to L^2(\RR^n)$ of the sub-Laplacian $\mathcal{L}$ can be found in \cite[Lemma 6.2.1]{FR2016}
\begin{equation}
\mathcal{L}_{\lambda}f(u)
=
-|\lambda|(\Delta_{\RR^n}f(u)-|u|^2f(u)),\quad f\in \mathcal S(\RR^n),\quad u\in \RR^n,
\end{equation}
and is a rescaled harmonic oscillator on $\RR^n$, see also Folland \cite{Folland-bk}.
It is known that for each $\lambda\in\RR\setminus\{0\}$ the operator $\mathcal{L}_{\lambda}$ has purely discrete spectrum 
$$
\Sp(\mathcal{L}_{\lambda})=\{
s_{1,\lambda}\leq s_{2,\lambda}\leq\ldots\leq s_{m,\lambda}\leq\ldots\}.
$$
Since $\HH^n$ is of type I, we have
\begin{equation}
\label{EQ:trace-type-I}
\tau(E_{(0,s)}[\mathcal{L}_{\lambda}])
=
\sum\limits_{\substack{k\in\NN^n\\ s_{k,\lambda} < s}} 1.
\end{equation}
The eigenvalues $s_{k,\lambda}$ are well-known and are given by  
\begin{equation}
\label{EQ:eigenvalues-computed-proof}
s_{k,\lambda}=\lambda\prod\limits^n_{j=1}(2k_j+1),
\end{equation}
see e.g.  \cite{Nicola2010}.
Thus, collecting \eqref{EQ:trace-decomposition}, \eqref{EQ:trace-type-I} and \eqref{EQ:eigenvalues-computed-proof}, we finally obtain
\begin{eqnarray*}
\tau(E_{(0,s)}(\mathcal{L}))
=
\int\limits_{\widehat{\HH^n}}
\sum\limits_{\substack{k\in\NN^n\\ s_{k,\lambda}<s}} 1
|\lambda|^nd\lambda
&=&
\\
\int\limits_{\widehat{\HH^n}}
\sum\limits_{\substack{k\in\NN^n\\ |\lambda| \prod ^n_{j=1}(2k_j+1)<s}} 1
|\lambda|^nd\lambda
&=&
\\
\sum\limits_{\substack{k\in\NN^n}}
\int\limits_{\substack{\widehat{\HH^n}\\ |\lambda| \leq \frac{s}{\prod\limits^n_{j=1}\left(2k_{j}+1\right)} }}
|\lambda|^nd\lambda
&=&
\frac{s^{n+1}}{n+1}
\prod\limits^{n}_{j=1}\sum\limits_{k_j\in\NN}\frac1{(2k_j+1)^{n+1}}.
\end{eqnarray*}

Summarising, we have
\begin{equation}
\label{EQ:tau-HH}
\tau(E_{(0,s)}(\mathcal{L}))
=
C_{n}
s^{\frac{Q}2},
\end{equation}
where we used the fact the homogeneous dimension $Q$ of the Heiseneberg group $\HH^n$ equals $2n+2$, i.e.
$$
Q=2n+2.
$$
Finally, using \eqref{EQ:L-r-infty-sub-Laplacian} it can be seen that $\|\varphi(\mathcal{L})\|_{L^{r,\infty}(\VN_R(\HH^n))}$ is finite if and only if condition \eqref{EQ:alpha-beta-embedding} holds.
\end{proof}

\appendix
\section{Mean convergence on compact groups}
\label{SEC:mean_convergence}
It has been shown by Stanton \cite{Stanton1976} that for class functions on semisimple compact Lie groups
the polyhedral Fourier partial sums $S_N f$converge to $f$ in $L^p$ provided that 
\begin{equation}\label{EQ:app-s}
 2-\frac1{s+1}<p<2+\frac1s.
\end{equation}
 Here the number $s$ depends on the root system
$\Rcal$ of the compact Lie group $G$, in the way we now describe.
We also note that the range of indices $p$ as above is sharp, see
Stanton and Tomas \cite{StantonThomas1976,Stanton-Tomas:AJM-1978} as well as 
Colzani, Giulini and Travaglini \cite{Colzani1989}. 
 
Let $G$ be a compact semisimple Lie group and let $T$ be a maximal torus of $G$, 
with Lie algebras $\mathfrak{g}$ and $\mathfrak{t}$, respectively. 
Let $n=\dim G$ and $l=\dim T=\rank G$. 
We define a positive definite inner product on $\mathfrak{t}$ by putting 
$(\cdot,\cdot)=-B(\cdot,\cdot)$, where $B$ is the Killing form.
Let $\Rcal$ be the set of roots of $\mathfrak{g}$. Choose in $\Rcal$ a system $\Rcal_+$ of positive roots (with cardinality $r$) and let $S=\{\alpha_1,\ldots,\alpha_{l}\}$ be the corresponding simple system. 
We define $\rho:=\frac12\sum\limits_{\alpha\in \Rcal_+}\alpha$.

For every 
$
\lambda\in{\iu t}^*
$
there exists a unique 
$
H_{\lambda}\in\mathfrak{t}$ such that $\lambda(H)=\iu(H_{\lambda},H)
$
for every 
$H\in\mathfrak{t}$.
The vectors 
$
H_{j}=\frac{4\pi \iu H_{\alpha_j}}{\alpha_j(H_{\alpha_j})}
$ 
generate the lattice sometimes denoted by ${\rm Ker(exp)}$.
The elements of the set 
$$
\Lambda
=
\{
\lambda\in{\iu t}^*\colon \lambda(H)\in 2\pi \iu\ZZ,\;\textrm{ for any }\; H\in {\rm Ker(exp)}
\}
$$ 
are called the weights of $G$ and the fundamental weights are defined by the relations 
$\lambda_j=2\pi \iu\delta_{jk}$, $j,k=1,\ldots,l$.
The subset 
$$
\mathfrak{D}
=
\{
\lambda\in\Lambda\colon \lambda=\sum^l_{j=1}m_j\nu_j,\,m_j\in\NN
\}
$$
of the set $\Lambda$ with positive coordinates $m_j$ is called the set of dominant weights. Here, the word `dominant' means that with respect to a certain partial order on the set $\Lambda$ every weight $\lambda=\sum^l_{j=1}m_j\nu_j$ with $m_j>0$ is maximal.
There exists a bijection between $\Gh$ and the semilattice $\mathfrak{D}$ of the dominant weights of $G$, i.e.
$$
\mathfrak{D}\ni\lambda=(\nu_1,\ldots,\nu_{l})\longleftrightarrow
\pi\in\Gh.
$$ 
Therefore, we will not distinguish between $\pi$ and the corresponding dominant weight $\lambda$ and will write 
\begin{equation*}
\label{EQ:definition_explaination}
\pi=(\pi_1,\ldots,\pi_{l}),
\end{equation*}
where we agree to set $\pi_i=\nu_i$.
With $\rho=\frac12\sum\limits_{\alpha\in \Rcal_+}\alpha$, for a natural number $N\in\NN$, we set
\begin{equation*}\label{EQ:QN}
Q_{N}:=\{\xi\in\Gh \colon \; \xi_i\leq N\rho_i, \; i=1,\ldots,l\}.
\end{equation*}
We call $Q_{N}$ a finite polyhedron of $N^{\rm th}$ order and denote by $\M_0$ the set of all finite polyhedrons in $\Gh$.

Now, fix an arbitrary fundamental weight $\lambda_j$, $j=1,\ldots,l$,
and set 
$
\Rcal_{\lambda_j}^{\perp}:=\{\alpha \in \Rcal_+ \colon (\alpha,\lambda_j)=0\},$
and $\Rcal_+=\Rcal_{\lambda_j}\oplus \Rcal_{\lambda_j}^{\perp}$. 
The number $s$ appearing in \eqref{EQ:app-s} is defined by
\begin{equation}
\label{EQ:s}
s:=\max\limits_{\substack{j=1,\ldots,l}} \card \Rcal_{\lambda_j}.
\end{equation}

\end{document}